\documentclass[11pt]{amsart}
\usepackage{amsmath}
\usepackage{amsfonts}
\usepackage{amssymb}
\usepackage[all]{xy}           %xypic macro for latex2.09
\usepackage{xcolor}
\usepackage{bbding}
\usepackage{txfonts}
\usepackage{amscd}

\usepackage[shortlabels]{enumitem}
\usepackage{ifpdf}
\ifpdf
  \usepackage[colorlinks,final,%backref=page,
  hyperindex]{hyperref}
\else
  \usepackage[colorlinks,final,%backref=page,
  hyperindex]{hyperref}
\fi
\usepackage{tikz}
\usepackage[active]{srcltx}

\makeatletter

\newtheorem{thm}{Theorem}[section]
\newtheorem{lem}[thm]{Lemma}
\newtheorem{cor}[thm]{Corollary}
\newtheorem{pro}[thm]{Proposition}
\newtheorem{ex}[thm]{Example}

\newtheorem{defi}[thm]{Definition}

\numberwithin{equation}{section}

\newcommand{\kt}{\mathfrak{t}}
\newcommand{\g}{\mathfrak{g}}
\newcommand{\gl}{\mathrm{End}}
\newcommand{\kl}{\mathfrak{l}}
\newcommand{\kr}{\mathfrak{r}}

\newcommand{\bz}{\mathbb{Z}}

\newcommand{\fl}{\mathbf{l}}
\newcommand{\fr}{\mathbf{r}}

\newcommand{\id}{\mathrm{id}}

\newcommand{\ad}{\mathrm{ad}}

\newcommand{\ZYBE}{\mathrm{ZYBE}}
\newcommand{\LYBE}{\mathrm{LYBE}}
\newcommand{\CYBE}{\mathrm{CYBE}}
\newcommand{\AYBE}{\mathrm{AYBE}}
\newcommand{\DAYBE}{\mathrm{DAYBE}}

\begin{document}

\title[Some constructions of ASI bialgebras and their applications to Lie bialgebras]
{Some constructions of antisymmetric infinitesimal bialgebras and their applications
to Lie bialgebras}

\author{Bo Hou}
\address{School of Mathematics and Statistics, Henan University, Kaifeng 475004,
China}
\email{houbo@henu.edu.cn, bohou1981@163.com}

%\author{Ru Li}
%\address{School of Mathematics and Statistics, Henan University, Kaifeng 475004,
%China}
%\email{13037698973@163.com}
\vspace{-5mm}

%\date{\today}

\begin{abstract}
In this paper, we mainly provide some methods for constructing antisymmetric
infinitesimal (ASI) bialgebras and Lie bialgebras using Zinbiel bialgebras and
diassociative bialgebras. We first show that there is an ASI bialgebra structure
on the tensor product of a diassociative bialgebra and a quadratic Zinbiel algebra,
while there is an infinite-dimensional ASI bialgebra structure on the tensor product
of a Zinbiel bialgebra and a quadratic $\bz$-graded diassociative algebra. For a
special quadratic $\bz$-graded Leibniz algebra, the property that its tensor product
with a Zinbiel bialgebra forms an ASI bialgebra characterizes the Zinbiel bialgebra.
By examining the relationship between solutions of the (classical) Yang-Baxter
equation in a Zinbiel algebra and the induced associative algebra, we prove that
the induced ASI bialgebra is quasi-triangular (resp. triangular, factorizable)
whenever the original Zinbiel bialgebra is quasi-triangular (resp. triangular,
factorizable). These conclusions enable us to provide a method for constructing Lie
bialgebras from Zinbiel bialgebras and two approaches for constructing Lie bialgebras
from diassociative bialgebras. We also provide specific descriptions of the
connections between the solutions of the Yang-Baxter equations and the
connections between the $O$-operators corresponding to these constructions.
\end{abstract}

\makeatletter
\@namedef{subjclassname@2020}{\textup{2020} Mathematics Subject Classification}
\makeatother

\maketitle

\vspace{-10mm}
\tableofcontents %Ŀ¼

%\setcounter{section}{0}

%\allowdisplaybreaks

%\end{document}
\vspace{-10mm}
%%%%%%%%%%%%%%%%%%%%%%%%%%%%%%%%%%%%%%%%%%%%%%%%%%%%%%%%%%%%%%%%%%%%%%%%%%%%%%%%%
%    section  1   Introduction
%%%%%%%%%%%%%%%%%%%%%%%%%%%%%%%%%%%%%%%%%%%%%%%%%%%%%%%%%%%%%%%%%%%%%%%%%%%%%%%%%%%%%%
\section{Introduction}\label{sec:intr}
The aim of this paper is to provide some construction methods for antisymmetric
infinitesimal bialgebras and Lie bialgebras by discussing the relationship between Lie
bialgebras, antisymmetric infinitesimal bialgebras, diassociative bialgebra, Zinbiel
bialgebras and Leibniz bialgebra.

%%%%%%%%%%%%%%%%%%%%%%%%%%%%%%%%%%%%%%%%%%%%%%%%%%%%%%%%%%%%%%%%%%%%%%%%%%%%%%%%
\smallskip\noindent
1.1. {\bf Bialgebra theory.}
Lie bialgebra was first introduced by Drinfeld in the context of the theory of
Yang-Baxter equations, and are closely related to quantum groups as deformations of
universal enveloping algebras \cite{Dri,CP}. Etingof and Kazhdan proved that every Lie
bialgebra has a corresponding quantized universal enveloping algebra, that is, there
exists a quantization for every Lie bialgebra \cite{EK}.
Later Lie bialgebras have found applications in many other areas of mathematics and
mathematical physics such as the theory of Hopf algebra deformations of universal
enveloping algebras \cite{ES}, string topology and symplectic field theory \cite{CFL},
Goldman-Turaev theory of free loops in Riemann surfaces with boundaries \cite{Gol},
and so on. The associative analog of Lie bialgebras was quickly proposed and developed
rapidly. Infinitesimal bialgebras first appeared in the work of Joni and Rota to give
an algebraic framework for the calculus of divided differences \cite{JR}, and studied
in \cite{Agu}. The antisymmetric version of infinitesimal bialgebras was introduced
in \cite{Zhe} by using the name associative $\mathrm{D}$-bialgebra. Later this structure
was studied systematically by Bai under the name antisymmetric infinitesimal
bialgebra (ASI bialgebra) \cite{Bai}. %The ASI bialgebra can be characterized
%by the well-known matched pair of associative algebras and a double construction of
%a Frobenius algebra.

Generally speaking, a bialgebra structure is a vector space equipped with both
an algebra structure and a coalgebra structure satisfying certain compatibility conditions.
In recent years, the bialgebra theories of various algebra structures have been extensively
developed, such as left-symmetric bialgebras (also called pre-Lie bialgebras)\cite{Bai1},
dendriform $\mathrm{D}$-bialgebras \cite{Bai}, Jordan bialgebras \cite{Zhe}, perm bialgebras
\cite{Hou,LZB}, Leibniz bialgebras \cite{TS}, Novikov bialgebras \cite{HBG},
Jacobi-Jordan bialgebras \cite{BCHM}, Zinbiel bialgebras \cite{Wan},
diassociative bialgebras (also called associative bi-dialgebras) \cite{HLLZ,Lu}, and so on.
As an important class of quasi-triangular Lie bialgebras, factorizable Lie bialgebras
are used to connect classical $r$-matrices with certain factorization problems, and have
various applications in integrable systems \cite{BGN,RS}. Recently, the factorizable Lie
bialgebras have received further research \cite{LS}, and the factorizable ASI
bialgebras \cite{SW}, factorizable Leibniz bialgebras \cite{BLST},
factorizable Zinbiel bialgebras \cite{Wan}, factorizable Novikov bialgebras \cite{CH},
factorizable diassociative bialgebras \cite{Lu} have been studied.

In recent years, the relationship between different types of bialgebra structures
has received a lot of attention. Since the operad of (left) Novikov algebras
and the operad of right Novikov algebras are Koszul dual, in \cite{HBG}, Hong, Bai
and Guo have proposed a method for constructing infinite-dimensional Lie bialgebras
using the affinization of Novikov bialgebras. Similarly, Lin, Zhou and Bai have
constructed infinite-dimensional Lie bialgebras by using the pre-Lie bialgebras
and perm bialgebras, respectively \cite{LZB}. Hou and Lin have constructed Lie bialgebras
by using the Leibniz bialgebras and Zinbiel bialgebras \cite{HL}. In \cite{Hou1}, we have
provided a method for constructing infinite-dimensional ASI bialgebras using the
affinization of dendriform $\mathrm{D}$-bialgebras, and given some methods for
constructing Lie bialgebras by discussing in detail the connections between
dendriform $D$-bialgebras, pre-Lie bialgebras, ASI bialgebras and Lie bialgebras.
Recently, we have discussed the close connections between Lie bialgebras, Leibniz
bialgebras, diassociative bialgebras and ASI bialgebras in \cite{HL}. In this paper,
we further investigate the connections between these bialgebra structures and provide
some methods for constructing ASI bialgebras and Lie bialgebras.

%%%%%%%%%%%%%%%%%%%%%%%%%%%%%%%%%%%%%%%%%%%%%%%%%%%%%%%%%%%%%%%%%%%%%%%%%%%%%%%%
\smallskip\noindent
1.2. {\bf Antisymmetric infinitesimal bialgebras via diassociative bialgebras and
Zinbiel bialgebras.} First, by showing that the tensor product of a diassociative
(co)algebra and a Zinbiel (co)algebra admits an (co)associative (co)algebra structure,
we are able to construct ASI bialgebras on the tensor product of diassociative
bialgebras and quadratic Zinbiel algebras. Moreover, by establishing the relationship
between solutions of the (classical) associative Yang-Baxter equation ($\AYBE$) in
the induced associative algebra and solutions of the diassociative Yang-Baxter equation
($\DAYBE$) in the original diassociative algebra, we can construct triangular Lie
bialgebras from the corresponding triangular diassociative bialgebras.

\smallskip\noindent
{\bf Theorem I } ( Theorems \ref{thm:dias-asbia} and \ref{thm:indu-triASI} )
{\it Let $(D, \dashv, \vdash, \theta_{\dashv}, \theta_{\vdash})$ be a diassociative
bialgebra and $(B, \diamond, \varpi)$ be a quadratic Zinbiel algebra.
Define a linear map $\Delta: D\otimes B\rightarrow(D\otimes B)\otimes(D\otimes B)$ by
$$
\Delta(d\otimes b)=\theta_{\vdash}(d)\bullet\nu_{\varpi}(b)
+\theta_{\dashv}(d)\bullet\tau(\nu_{\varpi}(b))
$$
for any $d\in D$ and $b\in B$, where $\tau: B\otimes B\rightarrow B\otimes B$ is given by
$\tau(b_{1}\otimes b_{2})=b_{2}\otimes b_{1}$ for any $b_{1}, b_{2}\in B$.
Then $(D\otimes B, \cdot, \Delta)$ is an ASI bialgebra.

In particular, the induced ASI bialgebra $(D\otimes B, \cdot, \Delta)$ is triangular
if $(D, \dashv, \vdash, \theta_{\dashv}, \theta_{\vdash})$ is triangular.}

\smallskip
A crucial construction of infinite-dimensional Lie algebras is known as the process
of affinization. Roughly speaking, the affinization of a given algebra structure consists
of defining an infinite-dimensional algebra structure and then obtaining another algebra
structure on the tensor product of the original algebra with the infinite-dimensional
one, which in turn can resolve the original algebra structure. Here, we use the
affinization of Zinbiel bialgebras to construct infinite-dimensional ASI bialgebras.
There is a $\bz$-graded quadratic diassociative algebra constructed by the Laurent
polynomials in \cite{Lu}. We obtain an affinization characterization
of Zinbiel algebras by this $\bz$-graded diassociative algebra, and show that there
exists a completed coassociative coalgebra structure on the tensor product of a Zinbiel
coalgebra and a completed diassociative coalgebra, which could give a characterization
of the Zinbiel coalgebra by its affinization with a completed diassociative coalgebra.
Therefore, we obtain that there is a natural completed ASI bialgebra structure on the
tensor product of a Zinbiel bialgebra and a quadratic $\bz$-graded diassociative algebra.
If the quadratic $\bz$-graded diassociative algebra constructed by the Laurent
polynomials, an affinization characterization of a Leibniz bialgebra is given.

\smallskip\noindent
{\bf Theorem II } ( Theorem \ref{thm:den-perm-ass} )
{\it Let $(B, \diamond, \nu)$ be a finite-dimensional Zinbiel bialgebra, $(D=\oplus_{i\in\bz}
D_{i}, \dashv, \vdash, \omega)$ be a quadratic $\bz$-graded diassociative algebra and
$(D\otimes B, \cdot)$ be the induced $\bz$-graded associative algebra from $(B, \diamond)$
by $(D=\oplus_{i\in\bz}D_{i}, \dashv, \vdash)$. Define a linear
map $\Delta: D\otimes B\rightarrow(D\otimes B)\,\hat{\otimes}\,(D\otimes B)$ by
$$
\Delta(d\otimes b)=\theta_{\vdash,\omega}(d)\bullet\nu(b)
+\theta_{\dashv,\omega}(d)\bullet\tau(\nu(b))
$$
for any $d\in D$ and $b\in B$. Then $(D\otimes B, \cdot, \Delta)$ is a completed
ASI bialgebra.

Moreover, if $(D=\oplus_{i\in\bz}D_{i}, \dashv, \vdash, \omega)$ is the quadratic
$\bz$-graded diassociative algebra given in Example \ref{ex:grdiaco}, then $(D\otimes B,
\cdot, \Delta)$ is a completed ASI bialgebra if and only if $(B, \diamond, \nu)$ is
a Zinbiel bialgebra.}

\smallskip
In particular, if the quadratic $\bz$-graded diassociative algebra $(D=\oplus_{i\in\bz}
D_{i}, \dashv, \vdash, \omega)$ is a finite-dimensional quadratic diassociative algebra,
we get the induced ASI bialgebra $(D\otimes B, \cdot, \Delta)$ is is a quasi-triangular
(resp. triangular, factorizable) if $(B, \diamond, \nu)$ is quasi-triangular (resp.
triangular, factorizable).

%%%%%%%%%%%%%%%%%%%%%%%%%%%%%%%%%%%%%%%%%%%%%%%%%%%%%%%%%%%%%%%%%%%%%%%%%%%%%%%%
\smallskip\noindent
1.3. {\bf Some constructions of Lie bialgebras.} In \cite{Bai}, Bai have shown that
there is naturally a Lie bialgebra structure on every ASI bialgebra. Based on the
construction of ASI bialgebras we have provided above, we present three approaches
to constructing Lie bialgebras. First, starting from a Zinbiel bialgebra
$(B, \diamond, \nu)$ and using its tensor product with a quadratic diassociative algebra
$(D, \dashv, \vdash, \omega)$, we can obtain an ASI bialgebra, which naturally leads
to a Lie bialgebra:
$$
\xymatrix@C=2cm@R=0.6cm{
\txt{$(B, \diamond, \nu)$ \\
{\tiny a diassociative bialgebra}}
\ar[r]^-{\mbox{\tiny Cor. \ref{cor:indassbia}}}
&\txt{$(D\otimes B, \cdot, \Delta)$ \\ {\tiny an ASI bialgebra}}
\ar[r]^-{\mbox{\tiny Pro. \ref{pro:ASI-Liebia}}}
& \txt{$(D\otimes B, [-,-], \delta)$ \\ {\tiny a Lie bialgebra}}}
$$
This construction method retains many properties of the original Zinbiel bialgebra.
For instance, we can prove that the constructed Lie bialgebra is quasi-triangular
(resp. triangular, factorizable) if the original Zinbiel bialgebra is
quasi-triangular (resp. triangular, factorizable). %This method also provides a way
%to construct solutions of $\CYBE$ in a Lie algebra.

Second, starting from a diassociative bialgebra $(D, \dashv, \vdash, \theta_{\dashv},
\theta_{\dashv})$ and using its tensor product with a quadratic Zinbiel algebra
$(B, \diamond, \varpi)$, we can obtain an ASI bialgebra, which naturally leads
to a Lie bialgebra:
$$
\xymatrix@C=2cm@R=0.6cm{
\txt{$(D, \dashv, \vdash, \theta_{\dashv}, \theta_{\vdash})$ \\
{\tiny a diassociative bialgebra}}
\ar[r]^-{\mbox{\tiny Thm. \ref{thm:dias-asbia}}}
&\txt{$(D\otimes B, \cdot, \Delta)$ \\ {\tiny an ASI bialgebra}}
\ar[r]^-{\mbox{\tiny Pro. \ref{pro:ASI-Liebia}}}
& \txt{$(D\otimes B, [-,-], \delta)$ \\ {\tiny a Lie bialgebra}}}
$$
For this structure, we can obtain that the induced Lie bialgebra is triangular if the
original diassociative bialgebra is triangular.
Recently, in both \cite{HLLZ} and \cite{Lu}, it was obtained that there is naturally a
Leibniz bialgebra structure on a diassociative bialgebra. This result provides us
with a third method for constructing Lie bialgebras. Given a diassociative bialgebra
$(D, \dashv, \vdash, \theta_{\dashv}, \theta_{\vdash})$, we obtain a Leibniz bialgebra
$(D, \ast, \vartheta)$. By using the tensor product of this Leibniz bialgebra and a
quadratic Zinbiel algebra, we can obtain a Lie bialgebra:
$$
\xymatrix@C=2cm@R=0.6cm{
\txt{$(D, \dashv, \vdash, \theta_{\dashv}, \theta_{\vdash})$ \\
{\tiny a diassociative bialgebra}}
\ar[r]^-{\mbox{\tiny Pro. \ref{pro:DASI-Leibbia}}}
&\txt{$(D, \ast, \vartheta)$ \\ {\tiny a Leibniz bialgebra}}
\ar[r]^-{\mbox{\tiny Thm. \ref{thm:liebia-LZ}}}
& \txt{$(D\otimes B, [-,-], \delta)$ \\ {\tiny a Lie bialgebra}}}
$$
This construction still retains the triangulation of bialgebras. It is worth noting
that the Lie bialgebras obtained from these two constructions starting from a
diassociative bialgebra are consistent, that is, there is a commutative diagram:
$$
\xymatrix@C=2cm@R=0.7cm{
\txt{$(D, \dashv, \vdash, \theta_{\dashv}, \theta_{\vdash})$ \\
{\tiny a diassociative bialgebra}}
\ar[d]_{{\rm Pro.}~\ref{pro:DASI-Leibbia}} \ar[r]^{{\rm Thm.}~\ref{thm:dias-asbia}}
&\txt{$(D\otimes B, \cdot, \Delta)$\\  {\tiny an ASI bialgebra}}
\ar[d]^{{\rm Pro.}~\ref{pro:ASI-Liebia}} \\
\txt{$(D, \ast, \vartheta)$ \\ {\tiny a Leibniz bialgebra}}
\ar[r]^{{\rm Thm.}~\ref{thm:liebia-LZ}\quad}
& \txt{$(D\otimes B, [-,-], \delta)$ \\ {\tiny a Lie bialgebra}}}
$$
Moreover, this commutative diagram also holds for triangular diassociative bialgebras,
triangular ASI bialgebras, triangular Leibniz bialgebras and triangular Lie bialgebra.

In \cite{Kup}, Kupershmidt found that the $\CYBE$ in tensor form on Lie algebras can be
converted into an $\mathcal{O}$-operator associated to the coregular representation.
That is, there is a one-to-one correspondence between skew-symmetric solutions of the
$\CYBE$ and a class of $\mathcal{O}$-operators. Moreover, triangular Lie bialgebras
are closely related to the skew-symmetric solutions of the $\CYBE$: a skew-symmetric
solutions of the $\CYBE$ gives rise to a natural Lie algebra structure to the dual
space of a Lie algebra, and the original Lie algebra becomes a triangular Lie bialgebra.
And these conclusions have been proven to be valid for algebraic structures such as
associative algebras, diassociative algebras, Leibniz algebras, Zinbiel algebras, and so on.
In this paper, by discussing the connection between the solutions of the $\DAYBE$ in a
diassociative algebra $(D, \dashv, \vdash)$ and the solutions of the (calssical) Leibniz
Yang-Baxter equation in the induced Leibniz algebra $(D, \ast)$, the solutions of the
$\CYBE$ in the induced Lie algebra $(D\otimes B, [-,-])$, as well as the solutions of
the $\AYBE$ in the related associative algebra $(D\otimes B, \cdot)$, we present the
close connections among triangular diassociative bialgebras, triangular Leibniz bialgebras,
triangular ASI bialgebras and triangular Lie bialgebras, as well as the connections among
$\mathcal{O}$-operators on these algebra structures. These results, when put together,
yield the following commutative diagram:

{\small
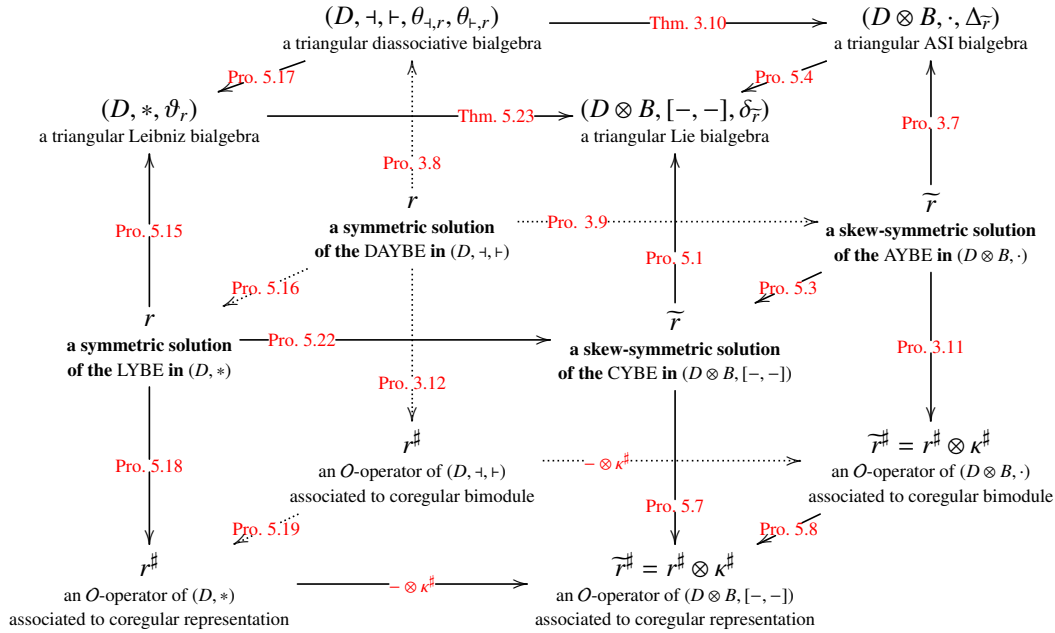
\begin{displaymath}
\xymatrix@R=0.4cm@C=-0.3cm{
&\txt{{\small $(D, \dashv, \vdash, \theta_{\dashv,r}, \theta_{\vdash,r})$ }\\
{\tiny a triangular diassociative bialgebra}}
\ar@{->}[rr]|-{\txt{\tiny\textcolor{red}{Thm.}~\ref{thm:indu-triASI}}}
\ar@{->}[ld]|-{\txt{\tiny\textcolor{red}{Pro.}~\ref{pro:qtdiAss-qtLeib}}}
\ar@{<.}[dd]|-(0.8){\txt{\tiny\textcolor{red}{Pro.}~\ref{pro:tri-di}}}&
&\txt{{\small $(D\otimes B, \cdot, \Delta_{\widetilde{r}})$} \\
{\tiny a triangular ASI bialgebra}}
\ar@{->}[ld]|-{\txt{\tiny\textcolor{red}{Pro.}~\ref{pro:qtAss-qtLie}}}
\ar@{<-}[dd]|-{\txt{\tiny\textcolor{red}{Pro.}~\ref{pro:quasass-bia}}}&\\
%2
\txt{{\small $(D, \ast, \vartheta_{r})$} \\ {\tiny a triangular Leibniz bialgebra}}
\ar@{<-}[dd]|-{\txt{\tiny\textcolor{red}{Pro.}~\ref{pro:sLib-bia}}}
\ar@{->}[rr]|-(0.75){\txt{\tiny\textcolor{red}{Thm.}~\ref{pro:sLeib-sLie}}}&
&\txt{{\small $(D\otimes B, [-,-], \delta_{\widetilde{r}})$}\\
{\tiny a triangular Lie bialgebra}}
\ar@{<-}[dd]|-(0.7){\txt{\tiny\textcolor{red}{Pro.}~\ref{pro:splie-bia}}} \\
%3
&\txt{{\small\bf $r$}\\ {\tiny\bf a symmetric solution} \\
{\tiny\bf of the $\DAYBE$ in $(D, \dashv, \vdash)$}}
\ar@{.>}[dd]|-(0.75){\txt{\tiny\textcolor{red}{Pro.}~\ref{pro:o-dia}}}
\ar@{.>}[rr]|-(0.2){\rm\textcolor{red}{Pro.}~\ref{pro:DAYBE-AYBE}}&
&\txt{{\small\bf $\widetilde{r}$} \\  {\tiny\bf a skew-symmetric solution} \\
{\tiny\bf of the $\AYBE$ in $(D\otimes B, \cdot)$}}
\ar@{->}[dd]|-{\txt{\tiny\textcolor{red}{Pro.}~\ref{pro:o-ass}}}\\
%4
\txt{{\small\bf $r$}\\ {\tiny\bf a symmetric solution} \\
{\tiny\bf of the $\LYBE$ in $(D, \ast)$}}
\ar@{->}[dd]|-{\txt{\tiny\textcolor{red}{Pro.}~\ref{pro:o-leib}}}
\ar@{<.}[ru]|-{\txt{\tiny\textcolor{red}{Pro.}~\ref{pro:diass-Leib-YBE}}}
\ar@{->}[rr]|-(0.2){\txt{\tiny\textcolor{red}{Pro.}~\ref{pro:LYBE-CYBE}}}&
&\txt{{\small\bf $\widetilde{r}$}\\ {\tiny\bf a skew-symmetric solution} \\
{\tiny\bf of the $\CYBE$ in $(D\otimes B, [-,-])$}}
\ar@{<-}[ru]|-{\txt{\tiny\textcolor{red}{Pro.}~\ref{pro:ass-Lie-YBE}}}
\ar@{->}[dd]|-(0.75){\txt{\tiny\textcolor{red}{Pro.}~\ref{pro:o-lie}}}\\
%5
&\txt{{\small $r^{\sharp}$}\\{\tiny an $\mathcal{O}$-operator of $(D, \dashv, \vdash)$} \\
{\tiny associated to coregular bimodule}}
\ar@{.>}[ld]|-{\txt{\tiny\textcolor{red}{Pro.}~\ref{pro:o-diass-leib}}} &
&\txt{{\small $\widetilde{r}^{\sharp}=r^{\sharp}\otimes\kappa^{\sharp}$}\\
{\tiny an $\mathcal{O}$-operator of $(D\otimes B, \cdot)$} \\
{\tiny associated to coregular bimodule}}
\ar@{<.}[ll]|-(0.75){\txt{\tiny\textcolor{red}{$-\otimes\kappa^{\sharp}$}}}\\
%6
\txt{{\small $r^{\sharp}$}\\{\tiny an $\mathcal{O}$-operator of $(D, \ast)$} \\
{\tiny associated to coregular representation}}&
&\txt{{\small $\widetilde{r}^{\sharp}=r^{\sharp}\otimes\kappa^{\sharp}$}\\
{\tiny an $\mathcal{O}$-operator of $(D\otimes B, [-,-])$} \\
{\tiny associated to coregular representation}}
\ar@{<-}[ru]|-{\txt{\tiny\textcolor{red}{Pro.}~\ref{pro:o-ass-lie}}}
\ar@{<-}[ll]|-{\txt{\tiny\textcolor{red}{$-\otimes\kappa^{\sharp}$}}}&}
\end{displaymath}
}

%%%%%%%%%%%%%%%%%%%%%%%%%%%%%%%%%%%%%%%%%%%%%%%%%%%%%%%%%%%%%%%%%%%%%%%%%%%%%%%%
\smallskip\noindent
1.3. {\bf Outline of the paper.}
This paper is organized as follows. In Section \ref{sec:Prelim} we recall the notions
of diassociative algebras and bimodules over diassociative algebras. We show in
Proposition \ref{pro:comm-diag} that there is a commutative diagram of diassociative
algebras, the induced associative algebras (by the tensor product with Zinbiel algebra),
the induced Leibniz algebras and Lie algebras. In Section \ref{sec:ASI} we analyze
the Lie bialgebra structure on the tensor product of a diassociative bialgebra and
a quadratic Zinbiel algebra, and provide in Theorem \ref{thm:indu-triASI} that the Lie
bialgebra structure is triangular if the original diassociative bialgebra is triangular.
In Section \ref{sec:zinb-ASI} we show that the tensor product of a finite-dimensional
Zinbiel bialgebra and a quadratic $\bz$-graded diassociative algebra can be naturally
endowed with a completed ASI bialgebra. The converse of this result also holds when the
quadratic $\bz$-graded diassociative algebra is special (see Theorem \ref{thm:den-perm-ass}).
In particular, if the quadratic $\bz$-graded diassociative algebra is a finite-dimensional
quadratic diassociative algebra, we get in Theorem \ref{thm:indu-qutriass} that the
induced ASI bialgebra is quasi-triangular (resp. triangular, factorizable) if the original
Zinbiel bialgebra is quasi-triangular (resp. triangular, factorizable).
In Section \ref{sec:diass-lie} we present a method for constructing Lie bialgebras
from Zinbiel bialgebras and two approaches for constructing Lie bialgebras from
diassociative bialgebras. The two approaches for constructing Lie bialgebras from
diassociative bialgebras are essentially consistent, and by discussing the relationship
between the solutions of (classical) Yang-Baxter equation in these algebra structures,
the three-dimensional commutative diagram above is presented.

Throughout this paper, we fix $\Bbbk$ as a field of characteristic zero.
All the vector spaces, algebras are over $\Bbbk$ and are finite-dimensional
unless otherwise specified, and all tensor products are also over $\Bbbk$.
We denote the identity map by $\id$. For any finite-dimensional vector spaces $A$
and $V$, we denote $V^{\ast}$ the dual space of $V$, and fix the following notations:
\begin{enumerate}
\item[$(i)$] Let $f: V\rightarrow V$ be a linear map. Define a linear map $f^{\ast}:
      V^{\ast}\rightarrow V^{\ast}$ by
      $\langle f^{\ast}(\xi),\; v\rangle=\langle\xi,\; f(v)\rangle$,
      for any $v\in V$ and $\xi\in V^{\ast}$.
\item[$(ii)$] Let $\mu: A\rightarrow\gl(V)$ be a linear map. Define a linear
      map $\mu^{\ast}: A\rightarrow\gl(V^{\ast})$ by
      $\langle\mu^{\ast}(a)(\xi),\; v\rangle=-\langle\xi,\; \mu(a)(v)\rangle$,
      for any $a\in A$, $v\in V$ and $\xi\in V^{\ast}$.
\item[$(iii)$] For any $r\in A\otimes A$, we define a linear map
      $r^{\sharp}: A^{\ast}\rightarrow A$ by
      $\langle\xi_{2},\; r^{\sharp}(\xi_{1})\rangle=\langle\xi_{1}\otimes\xi_{2},\;
      r\rangle$, for any $\xi_{1}, \xi_{2}\in A^{\ast}$.
\end{enumerate}

%%%%%%%%%%%%%%%%%%%%%%%%%%%%%%%%%%%%%%%%%%%%%%%%%%%%%%%%%%%%%%%%%%%%%%%%%%%%%%%%%
%    section  2   Preliminaries
%%%%%%%%%%%%%%%%%%%%%%%%%%%%%%%%%%%%%%%%%%%%%%%%%%%%%%%%%%%%%%%%%%%%%%%%%%%%%%%%%%%%%%
\section{Preliminaries on diassociative algebras and related algebra structures}
\label{sec:Prelim}
In this section, we recall the notions of diassociative algebras and bimodules over a
diassociative algebra. The definitions of some related algebra structures are given,
such as associative algebras, Leibniz algebras, Zinbiel algebras and Lie algebras.
Recall that an {\bf associative algebra} $(A, \cdot)$ is a vector space $A$ equipped with
a binary operation $\cdot: A\otimes A\rightarrow A$ such that for any $a_{1}, a_{2}, a_{3}
\in A$,
$$
a_{1}\cdot(a_{2}\cdot a_{3})=(a_{1}\cdot a_{2})\cdot a_{3}.
$$
Diassociative algebra (also called associative dialgebra) is a generalization of
associative algebra, which was first introduced by Loday in the early 1990s in connection
with  periodicity phenomena in algebraic $K$-theory.

\begin{defi}\label{de:dialg}
A {\bf diassociative algebra} $(D, \dashv, \vdash)$ is a vector space $D$ equipped with
two bilinear maps called respectively left product and right product: $\dashv, \vdash:
D\otimes D\rightarrow D$, such that $(D, \dashv)$ and $(D, \vdash)$
are associative algebras and satisfying the following axioms:
$$
(d_{1}\dashv d_{2})\dashv d_{3}=d_{1}\dashv(d_{2}\vdash d_{3}),\qquad
(d_{1}\vdash d_{2})\dashv d_{3}=d_{1}\vdash(d_{2}\dashv d_{3}),\qquad
(d_{1}\dashv d_{2})\vdash d_{3}=(d_{1}\vdash d_{2})\vdash d_{3},
$$
for any $d_{1}, d_{2}, d_{3}\in D$
\end{defi}

Recall that a {\bf (left) Zinbiel algebra} $(B, \diamond)$ is a vector space $B$ together
with a bilinear map $\diamond: B\times B\rightarrow B$ satisfying the following
Zinbiel identity:
$$
b_{1}\diamond(b_{2}\diamond b_{3})=(b_{1}\diamond b_{2})\diamond b_{3}
+(b_{2}\diamond b_{1})\diamond b_{3},
$$
for any $b_{1}, b_{2}, b_{3}\in B$. Each associative algebra can be view as a
diassociative algebra by define $\dashv=\cdot=\vdash$. Conversely, we can get an associative
algebra by the tensor product of a diassociative algebra and a Zinbiel algebra.

\begin{pro}\label{pro:diass-ass}
Let $(D, \dashv, \vdash)$ be a diassociative algebra and $(B, \diamond)$ be a
Zinbiel algebra. Define a binary operation $\cdot: (D\otimes B)\otimes(D\otimes B)
\rightarrow D\otimes B$ by
$$
(d_{1}\otimes b_{1})\cdot(d_{2}\otimes b_{2})=(d_{1}\vdash d_{2})\otimes(b_{1}\diamond b_{2})
+(d_{1}\dashv d_{2})\otimes(b_{2}\diamond b_{1}),
$$
for any $d_{1}, d_{2}\in D$ and $b_{1}, b_{2}\in B$. Then $(D\otimes B, \cdot)$
is an associative algebra, which is called {\bf the induced associative algebra from
$(D, \dashv, \vdash)$ and $(B, \diamond)$}.
\end{pro}

\begin{proof}
Since $(D, \dashv, \vdash)$ is a diassociative algebra and $(B, \diamond)$ is a
Zinbiel algebra, for any $d_{1}, d_{2}, d_{3}\in D$ and $b_{1}, b_{2}, b_{3}\in B$,
by direct calculation, we have
\begin{align*}
&\;((d_{1}\otimes b_{1})\cdot(d_{2}\otimes b_{2}))\cdot(d_{3}\otimes b_{3})\\
%=&\;((d_{1}\vdash d_{2})\otimes(b_{1}b_{2})+(d_{1}\dashv d_{2})\otimes(b_{2}b_{1}))
%\cdot(d_{3}\otimes b_{3})\\
%=&\;((d_{1}\vdash d_{2})\vdash d_{3})\otimes((b_{1}b_{2})b_{3})
%+((d_{1}\vdash d_{2})\dashv d_{3})\otimes(b_{3}(b_{1}b_{2}))\\[-1mm]
%&\qquad+((d_{1}\dashv d_{2})\vdash d_{3})\otimes((b_{2}b_{1})b_{3})
%+((d_{1}\dashv d_{2})\dashv d_{3})\otimes(b_{3}(b_{2}b_{1}))\\
=&\;((d_{1}\vdash d_{2})\vdash d_{3})\otimes((b_{1}b_{2})b_{3})
+((d_{1}\vdash d_{2})\dashv d_{3})\otimes((b_{3}b_{1})b_{2})
+((d_{1}\vdash d_{2})\dashv d_{3})\otimes((b_{1}b_{3})b_{2})\\[-1mm]
&\quad+((d_{1}\dashv d_{2})\vdash d_{3})\otimes((b_{2}b_{1})b_{3})
+((d_{1}\dashv d_{2})\dashv d_{3})\otimes((b_{3}b_{2})b_{1})
+((d_{1}\dashv d_{2})\dashv d_{3})\otimes((b_{2}b_{3})b_{1})\\
=&\;((d_{1}\vdash(d_{2}\vdash d_{3}))\otimes((b_{1}b_{2})b_{3})
+((d_{1}\vdash(d_{2}\vdash d_{3}))\otimes((b_{2}b_{1})b_{3})
+((d_{1}\dashv(d_{2}\vdash d_{3}))\otimes((b_{2}b_{3})b_{1})\\[-1mm]
&\quad+((d_{1}\vdash(d_{2}\dashv d_{3}))\otimes((b_{1}b_{3})b_{2})
+((d_{1}\vdash(d_{2}\dashv d_{3}))\otimes((b_{3}b_{1})b_{2})
+((d_{1}\dashv(d_{2}\dashv d_{3}))\otimes((b_{3}b_{2})b_{1})\\
=&\;(d_{1}\otimes b_{1})\cdot((d_{2}\otimes b_{2})\cdot(d_{3}\otimes b_{3})).
\end{align*}
Hence $(D\otimes B, \cdot)$ is an associative algebra.
\end{proof}

Recall that a (left) {\bf Leibniz algebra} $(L, \ast)$ is a vector space $L$ together
with a bilinear map $\ast: L\times L\rightarrow L$ satisfying the following Leibniz identity:
$$
x_{1}\ast(x_{2}\ast x_{3})=(x_{1}\ast x_{2})\ast x_{3}+x_{2}\ast(x_{1}\ast x_{3}),
$$
for any $x_{1}, x_{2}, x_{3}\in L$. In a Leibniz algebra $(L, \ast)$, if $\ast$ is
anticommutative, i.e., $x_{1}\ast x_{2}=-x_{2}\ast x_{1}$, then $(L, \ast)$ is a
{\bf Lie algebra}.
In a Lie algebra, we usually denote the binary operation $\ast$ by bracket $[-,-]$.
It is well-known that an associative algebra forms a Lie algebra under the commutator.
Similar results also exist for diassociative algebras and Leibniz algebras.

\begin{pro}[\cite{Lod}]\label{pro:commtor}
Let $(D, \dashv, \vdash)$ be a diassociative algebra. Define a new binary operation $\ast$
on $D$ by
$$
d_{1}\ast d_{2}=d_{1}\vdash d_{2}-d_{2}\dashv d_{1}
$$
for any $d_{1}, d_{2}\in D$, then $(D, \ast)$ is a Leibniz algebra, which is called
the {\bf Leibniz algebra induced by $(D, \dashv, \vdash)$}.
\end{pro}

Recently, we consider the construction of Lie bialgebras from Leibniz bialgebras
and Zinbiel bialgebras in \cite{HL1}. The starting point for this construction is that
there exists a Lie algebra structure on the tensor product of a Leibniz algebra and
a Zinbiel algebra.

\begin{pro}[\cite{HL1}]\label{pro:L-Z-lie}
Let $(B, \diamond)$ be a Zinbiel algebra and $(L, \ast)$ be a Leibniz algebra.
Define a binary operation $[-,-]$ on $L\otimes B$ by
$$
[x_{1}\otimes b_{1},\; x_{2}\otimes b_{2}]=(x_{1}\ast x_{2})\otimes(b_{1}\diamond b_{2})
-(x_{2}\ast x_{1})\otimes(b_{2}\diamond b_{1}),
$$
for any $x_{1}, x_{2}\in L$ and $b_{1}, b_{2}\in B$. Then $(L\otimes B,\; [-, -])$
is a Lie algebra, called {\bf the induced Lie algebra} from $(L, \ast)$ and $(B, \diamond)$.
\end{pro}

Thus, for diassociative algebras, associative algebra, Lie algebras and Leibniz
algebras, we have a commutative diagram as given by the following proposition.

\begin{pro}\label{pro:comm-diag}
Let $(A, \dashv, \vdash)$ be a diassociative algebra and $(B, \diamond)$ be a
Zinbiel algebra. Then we have the following commutative diagram:
$$
\xymatrix@C=2cm@R=0.5cm{
\txt{$(D, \dashv, \vdash)$ \\ {\tiny a diassociative algebra}}
\ar[r]^-{{\rm Pro.}~\ref{pro:commtor}}
\ar[d]_-{{\rm Pro.}~\ref{pro:diass-ass}}
& \txt{$(D, \ast)$ \\ {\tiny a Leibniz algebra}}
\ar[d]^-{{\rm Pro.}~\ref{pro:L-Z-lie}}\\
\txt{$(D\otimes B, \cdot)$ \\ {\tiny an associative algebra}}
\ar[r]^-{{\rm\tiny commutator}}
&\txt{$(D\otimes B, [-,-])$\\ {\tiny a Lie algebra}}}
$$
\end{pro}

\begin{proof}
We denote $[-,-]_{1}$ the bracket on $D\otimes B$ induced by $\ast$, that is,
\begin{align*}
[d_{1}\otimes b_{1},\; d_{2}\otimes b_{2}]_{1}
&=(d_{1}\ast d_{2})\otimes(b_{1}\diamond b_{2})-(d_{2}\ast d_{1})\otimes(b_{2}\diamond b_{1})\\
&=(d_{1}\vdash d_{2})\otimes(b_{1}\diamond b_{2})
-(d_{1}\dashv d_{2})\otimes(b_{1}\diamond b_{2})\\[-1mm]
&\quad -(d_{2}\vdash d_{1})\otimes(b_{2}\diamond b_{1})
+(d_{2}\dashv d_{1})\otimes(b_{2}\diamond b_{1}),
\end{align*}
for any $d_{1}, d_{2}\in D$ and $b_{1}, b_{2}\in B$. On the other hand, if we
denote $[-,-]_{2}$ the bracket on $D\otimes B$ induced by $\cdot$, that is,
\begin{align*}
[d_{1}\otimes b_{1},\; d_{2}\otimes b_{2}]_{2}
&=(d_{1}\otimes b_{1})\cdot(d_{2}\otimes b_{2})-(d_{2}\otimes b_{2})\cdot(d_{1}\otimes b_{1})\\
&=(d_{1}\vdash d_{2})\otimes(b_{1}\diamond b_{2})
+(d_{2}\dashv d_{1})\otimes(b_{2}\diamond b_{1})\\[-1mm]
&\quad -(d_{2}\vdash d_{1})\otimes(b_{2}\diamond b_{1})
-(d_{1}\dashv d_{2})\otimes(b_{1}\diamond b_{2}).
\end{align*}
Thus, $[-,-]_{1}=[-,-]_{2}$, and so that the diagram is commutative.
\end{proof}

This commutative diagram provides us with two methods for constructing Lie algebras,
and both methods yield consistent Lie algebras. One of the main objectives of this paper
is to elevate this commutative diagram to the level of bialgebras and provide
some construction methods for Lie bialgebras.
Next, we consider bimodules over a diassociative algebra. Let $(A, \cdot)$ be an associative
algebra $V$ be a vector space and $\kl, \kr: A\rightarrow\gl(V)$ be two linear maps.
If for any $a_{1}, a_{2}\in A$, $\kl(a_{1}\cdot a_{2})=\kl(a_{1})\kl(a_{2})$,
$\kl(a_{1})\kr(a_{2})=\kl(a_{2})\kl(a_{1})$ and $\kr(a_{1}\cdot a_{2})
=\kr(a_{2})\kr(a_{1})$, then $(V, \kl, \kr)$ is called a {\bf bimodule} over
$(A, \cdot)$. In particular, $(A, \fl_{A}, \fr_{A})$ is a bimodule over $(A, \cdot)$,
which is celled the {\bf regular bimodule} over $(A, \cdot)$, where $\fl_{A}(a_{1})(a_{2})
=a_{1}\cdot a_{2}=\fr_{A}(a_{2})(a_{1})$ for any $a_{1}, a_{2}\in A$.

\begin{defi}\label{def:SL-mod}
Let $(D, \dashv, \vdash)$ be a diassociative algebra and $V$ be a $\Bbbk$-vector space.
If there exist four bilinear maps $\kl_{\dashv}, \kr_{\dashv}, \kl_{\vdash}, \kr_{\vdash}:
D\rightarrow\gl(V)$, such that $(V, \kl_{\dashv}, \kr_{\dashv})$ is a bimodule
over $(D, \dashv)$, $(V, \kl_{\vdash}, \kr_{\vdash})$ is a bimodule over $(D, \vdash)$
and for any $d_{1}, d_{2}\in D$,
\begin{align*}
\kr_{\vdash}(d_{1}\vdash d_{2})&=\kr_{\vdash}(d_{1}\dashv d_{2}),
&& \kl_{\vdash}(d_{1})\kr_{\vdash}(d_{2})=\kl_{\vdash}(d_{1})\kr_{\dashv}(d_{2}),
&& \kl_{\vdash}(d_{1})\kl_{\vdash}(d_{2})=\kl_{\vdash}(d_{1})\kl_{\dashv}(d_{2}),  \\
\kr_{\vdash}(d_{1})\kr_{\dashv}(d_{2})&=\kr_{\dashv}(d_{2}\vdash d_{1}),
&& \kr_{\vdash}(d_{1})\kl_{\dashv}(d_{2})=\kl_{\dashv}(d_{2})\kr_{\vdash}(d_{1}),
&&\; \kl_{\vdash}(d_{1}\dashv d_{2})=\kl_{\dashv}(d_{1})\kl_{\vdash}(d_{2}),  \\
\kr_{\dashv}(d_{1})\kr_{\vdash}(d_{2})&=\kr_{\dashv}(d_{1})\kr_{\dashv}(d_{2}),
&& \kr_{\dashv}(d_{1})\kl_{\vdash}(d_{2})=\kr_{\dashv}(d_{1})\kl_{\dashv}(d_{2}),
&&\; \kl_{\dashv}(d_{1}\vdash d_{2})=\kl_{\dashv}(d_{1}\dashv d_{2}),
\end{align*}
we call $(V, \kl_{\dashv}, \kr_{\dashv}, \kl_{\vdash}, \kr_{\vdash})$ is a
{\bf bimodule} over $(D, \dashv, \vdash)$.
\end{defi}

Let $(V, \kl_{\dashv}, \kr_{\dashv}, \kl_{\vdash}, \kr_{\vdash})$ and $(V', \kl'_{\dashv},
\kr'_{\dashv}, \kl'_{\vdash}, \kr'_{\vdash})$ be two bimodules over a diassociative algebra
$(D, \dashv, \vdash)$ and $f: V\rightarrow V'$ be a linear map. Then $f$ is called a
{\bf morphism of bimodule} if $f(\beta(d)(v))=\beta'(d)(f(v))$ for any
$\beta\in\{\kl_{\dashv}, \kr_{\dashv}, \kl_{\vdash}, \kr_{\vdash}\}$, $d\in A$ and $v\in V$.
This two bimodules $(V, \kl_{\dashv}, \kr_{\dashv}, \kl_{\vdash}, \kr_{\vdash})$ and
$(V', \kl'_{\dashv}, \kr'_{\dashv}, \kl'_{\vdash}, \kr'_{\vdash})$ are called {\bf
isomorphic as bimodules} if the morphism $f$ is a bijection. It is easy to see that
a diassociative algebra $(D, \dashv, \vdash)$ is a bimodule over itself under the actions
$\fl_{\dashv}, \fr_{\dashv}, \fl_{\vdash}, \fr_{\vdash}: D\rightarrow\gl(D)$,
$\fl_{\dashv}(d_{1})(d_{2})=d_{1}\dashv d_{2}$,
$\fr_{\dashv}(d_{1})(d_{2})=d_{2}\dashv d_{1}$,
$\fl_{\vdash}(d_{1})(d_{2})=d_{1}\vdash d_{2}$ and
$\fr_{\vdash}(d_{1})(d_{2})=d_{2}\vdash d_{1}$ for $d_{1}, d_{2}\in D$.
This bimodule is called the {\bf regular bimodule} over $(D, \dashv, \vdash)$.
Moreover, by direct calculations, we can give an equivalent condition as follows.

\begin{pro}\label{pro:bimodule}
Let $(D, \dashv, \vdash)$ be a diassociative algebra, $V$ be a $\Bbbk$-vector space and
$\kl_{\dashv}, \kr_{\dashv}, \kl_{\vdash}, \kr_{\vdash}: D\rightarrow\gl(V)$ be four
linear maps. Then $(V, \kl_{\dashv}, \kr_{\dashv}, \kl_{\vdash}, \kr_{\vdash})$ is a
bimodule over $(D, \dashv, \vdash)$ if and only if $D\oplus V$ is a diassociative algebra under
the following operations:
\begin{align*}
(d_{1}, v_{1})\dashv(d_{2}, v_{2})&:=\big(d_{1}\dashv d_{2},\ \
\kl_{\dashv}(d_{1})(v_{2})+\kr_{\dashv}(d_{2})(v_{1})\big),\\
(d_{1}, v_{1})\vdash(d_{2}, v_{2})&:=\big(d_{1}\vdash d_{2},\ \
\kl_{\vdash}(d_{1})(v_{2})+\kr_{\vdash}(d_{2})(v_{1})\big),
\end{align*}
for all $d_{1}, d_{2}\in D$ and $v_{1}, v_{2}\in V$. This diassociative algebra is called
a {\bf semidirect product} of $(D, \dashv, \vdash)$ by bimodule $(V, \kl_{\dashv},
\kr_{\dashv}, \kl_{\vdash}, \kr_{\vdash})$, denoted by $D\ltimes V$.
\end{pro}

\begin{proof}
It is straightforward.
\end{proof}

%Let $V$ be a vector space. Denote the standard pairing between the dual space
%$V^{\ast}$ and $V$ by
%\begin{align*}
%\langle-,-\rangle:\quad V^{\ast}\otimes V\rightarrow \Bbbk, \qquad\quad
%\langle \xi,\; v \rangle:=\xi(v),
%\end{align*}
%for any $\xi\in V^{\ast}$ and $v\in V$. Let $V$, $W$ be two vector spaces. For a linear
%map $\varphi: V\rightarrow W$, the transpose map $\varphi^{\ast}: W^{\ast}\rightarrow
%V^{\ast}$ is defined by
%\begin{align*}
%\langle \varphi^{\ast}(\xi),\; v \rangle:=\langle\xi,\; \varphi(v)\rangle,
%\end{align*}
%for any $v\in V$ and $\xi\in W^{\ast}$. Let $(D, \dashv, \vdash)$ be a diassociative algebra
%and $V$ be a vector space. For a linear map $\psi: A\rightarrow\gl(V)$, the linear map
%$\psi^{\ast}: D\rightarrow\gl(V^{\ast})$ is defined by
%\begin{align*}
%\langle\psi^{\ast}(d)(\xi),\; v\rangle:=-\langle\xi,\; \psi(d)(v)\rangle,
%\end{align*}
%for any $d\in D$, $v\in V$, $\xi\in V^{\ast}$. That is, $\psi^{\ast}(d)=\psi(d)^{\ast}$
%for all $d\in D$.
For any bimodule $(V, \kl_{\dashv}, \kr_{\dashv}, \kl_{\vdash},
\kr_{\vdash})$ over a diassociative algebra $(D, \dashv, \vdash)$, one can check that
$(V^{\ast}, \kr_{\vdash}^{\ast}-\kr_{\dashv}^{\ast}, -\kl_{\vdash}^{\ast},
-\kr_{\dashv}^{\ast}$, $\kl_{\dashv}^{\ast}-\kl_{\vdash}^{\ast})$ is again a
bimodule over $(D, \dashv, \vdash)$. In particular, we get $(D^{\ast}, \fr_{\vdash}^{\ast}
-\fr_{\dashv}^{\ast}, -\fl_{\vdash}^{\ast}, -\fr_{\dashv}^{\ast}$, $\fl_{\dashv}^{\ast}
-\fl_{\vdash}^{\ast})$ is a bimodule over $(D, \dashv, \vdash)$, which is called
the {\bf coregular bimodule}.
Let $V$ be a vector space and $\omega(-,-)$ be a bilinear form on $V$. Recall that
\begin{enumerate}\itemsep=0pt
\item[-] $\omega(-,-)$ is called {\bf nondegenerate} if $\omega(v_{1},\;
        v_{2})=0$ for any $v_{2}\in V$, then $v_{1}=0$;
\item[-] $\omega(-,-)$ is called {\bf symmetric} if $\omega(v_{1},\; v_{2})
        =\omega(v_{2},\; v_{1})$ for any $v_{1}, v_{2}\in V$;
\item[-] $\omega(-,-)$ is called {\bf skew-symmetric} if $\omega(v_{1},\; v_{2})
        =-\omega(v_{2},\; v_{1})$, for any $v_{1}, v_{2}\in V$.
\end{enumerate}
Let $\omega(-,-)$ be a bilinear form on a diassociative algebra $(D, \dashv, \vdash)$.
Then $\omega(-,-)$ is called {\bf invariant} if for any $d_{1}, d_{2}, d_{3}\in D$,
$$
\omega(d_{1}\vdash d_{2},\; d_{3})=\omega(d_{1},\; d_{2}\vdash d_{3}-d_{2}\dashv d_{3})
\qquad\mbox{and}\qquad \omega(d_{1}\dashv d_{2},\; d_{3})=\omega(d_{1},\; d_{2}\vdash d_{3}).
$$
A (skew-symmetric) {\bf quadratic diassociative algebra}
$(D, \dashv, \vdash, \omega)$ is a diassociative algebra $(D, \dashv, \vdash)$ with a
skew-symmetric, nondegenerate and invariant bilinear form $\omega(-,-): D\otimes
D\rightarrow\Bbbk$. By direct calculation, we have:

\begin{pro}\label{pro:dual}
Let $(D, \dashv, \vdash)$ be a diassociative algebra. Then the regular bimodule $(D,
\fl_{\dashv}, \fr_{\dashv}$, $\fl_{\vdash}, \fr_{\vdash})$ and the coregular bimodule
$(D^{\ast}, \fr_{\vdash}^{\ast}-\fr_{\dashv}^{\ast}, -\fl_{\vdash}^{\ast},
-\fr_{\dashv}^{\ast}, \fl_{\dashv}^{\ast}-\fl_{\vdash}^{\ast})$ are isomorphic as
bimodules over $(D, \dashv, \vdash)$ if there exists a skew-symmetric,
nondegenerate invariant bilinear form $\omega(-,-)$ on $(D, \dashv, \vdash)$.
\end{pro}

%%%%%%%%%%%%%%%%%%%%%%%%%%%%%%%%%%%%%%%%%%%%%%%%%%%%%%%%%%%%%%%%%%%%%%%%%%%%%%%%%
%    section  3   ASI
%%%%%%%%%%%%%%%%%%%%%%%%%%%%%%%%%%%%%%%%%%%%%%%%%%%%%%%%%%%%%%%%%%%%%%%%%%%%%%%%%%%%%%
\section{Antisymmetric infinitesimal bialgebras induced by diassociative bialgebras}
\label{sec:ASI}
In this section, we show that there is an ASI bialgebra structure
on the tensor product of a diassociative bialgebra and a quadratic Zinbiel algebra. In
particular, we show that a quasi-triangular diassociative bialgebra induces a
quasi-triangular antisymmetric infinitesimal bialgebra structure on the tensor product.

Recall that a {\bf coassociative coalgebra} is a pair $(A, \Delta)$, where $A$ is a vector
space and $\Delta: A\rightarrow A\otimes A$ is a linear map satisfying
$(\Delta\otimes\id)\Delta=(\id\otimes\Delta)\Delta$.
A {\bf diassociative coalgebra} is a triple $(D, \theta_{\dashv}, \theta_{\vdash})$,
where $(D, \theta_{\dashv})$ and $(D, \theta_{\vdash})$ are coassociative coalgebras,
and satisfying the following conditions:
\begin{align}
(\theta_{\dashv}\otimes\id)\theta_{\dashv}
=(\id\otimes\theta_{\vdash})\theta_{\dashv},\qquad\quad
(\theta_{\vdash}\otimes\id)\theta_{\dashv}
=(\id\otimes\theta_{\dashv})\theta_{\vdash},\qquad\quad
(\theta_{\dashv}\otimes\id)\theta_{\vdash}
=(\theta_{\vdash}\otimes\id)\theta_{\vdash}.    \label{codi}
\end{align}
One can check that $(D, \theta_{\dashv}, \theta_{\vdash})$ is a diassociative coalgebra
if and only if $(D, \theta_{\dashv}^{\ast}, \theta_{\vdash}^{\ast})$ is a diassociative
algebra. A {\bf Zinbiel coalgebra} $(B, \nu)$ is a vector space $B$ with a linear map
$\nu: B\rightarrow B\otimes B$ such that
$$
(\id\otimes\nu)\nu=(\nu\otimes\id)\nu+(\tau\otimes\id)(\nu\otimes\id)\nu,
$$
where $\tau: B\otimes B\rightarrow B\otimes B$ is the twist map defined by
$\tau(b_{1}\otimes b_{2}):=b_{2}\otimes b_{1}$ for all $b_{1}, b_{2}\in B$. First,
we show that there is a coassociative coalgebra on the tensor product of a
diassociative coalgebra and a Zinbiel coalgebra.

\begin{pro}\label{pro:dco-coas}
Let $(B, \nu)$ be a Zinbiel coalgebra and $(D, \theta_{\dashv}, \theta_{\vdash})$ be a
diassociative coalgebra. Define a linear map $\Delta: D\otimes B\rightarrow
(D\otimes B)\otimes(D\otimes B)$ by
\begin{align*}
\Delta(d\otimes b)&=\theta_{\vdash}(d)\bullet\nu(b)
+\theta_{\dashv}(d)\bullet\tau(\nu(b))\\
&:=\sum_{(b)}\sum_{(d)}(d_{(1)}\otimes b_{(1)})\otimes(d_{(2)}\otimes b_{(2)})
+\sum_{(b)}\sum_{[d]}((d_{[1]}\otimes b_{(2)})\otimes(d_{[2]}\otimes b_{(1)}),
\end{align*}
for any $d\in D$ and $b\in B$, where $\theta_{\vdash}(d)=\sum_{(d)}
d_{(1)}\otimes d_{(2)}$, $\theta_{\dashv}(d)=\sum_{[d]}d_{[1]}\otimes d_{[2]}$ and
$\nu(b)=\sum_{(b)}b_{(1)}\otimes b_{(2)}$ in the Sweedler notation. Then $(D\otimes B,
\Delta)$ is a coassociative coalgebra.
\end{pro}

\begin{proof}
For any $\sum_{l}b'_{l}\otimes b''_{l}\otimes b'''_{l}\in B\otimes B\otimes B$ and
$\sum_{k}d'_{k}\otimes d''_{k}\otimes d'''_{k}\in D\otimes D\otimes D$, we denote
$$
\Big(\sum_{k}d'_{k}\otimes d''_{k}\otimes d'''_{k}\Big)
\bullet\Big(\sum_{l}b'_{l}\otimes b''_{l}\otimes b'''_{l}\Big)
=\sum_{l, k}(d'_{k}\otimes b'_{l})\otimes
(d''_{k}\otimes b''_{l})\otimes(d'''_{k}\otimes b'''_{l}).
$$
Then, by using the above notations, since $(B, \nu)$ is a Zinbiel coalgebra
and $(D, \theta_{\dashv}, \theta_{\vdash})$ is a diassociative coalgebra,
for any $d\otimes b\in D\otimes B$, we have:
\begin{align*}
&\;(\id\otimes\Delta)(\Delta(d\otimes b))\\
=&\;(\id\otimes\theta_{\vdash})(\theta_{\vdash}(d))\bullet(\id\otimes\nu)(\nu(b))
+(\id\otimes\theta_{\dashv})(\theta_{\dashv}(d))\bullet(\id\otimes\tau)
((\tau\otimes\id)((\id\otimes\tau)((\nu\otimes\id)(\nu(b)))))\\[-1mm]
&\;+(\id\otimes\theta_{\dashv})(\theta_{\vdash}(d))
\bullet(\id\otimes\tau)((\id\otimes\nu)(\nu(b)))
+(\id\otimes\theta_{\vdash})(\theta_{\dashv}(d))\bullet(\tau\otimes\id)
((\id\otimes\tau)((\nu\otimes\id)(\nu(b))))
\end{align*}
\begin{align*}
=&\;(\id\otimes\theta_{\vdash})(\theta_{\vdash}(d))\bullet(\nu\otimes\id)(\nu(b))
+(\id\otimes\theta_{\dashv})(\theta_{\dashv}(d))\bullet(\id\otimes\tau)
((\tau\otimes\id)((\id\otimes\tau)((\nu\otimes\id)(\nu(b)))))\\[-1mm]
&\;+(\id\otimes\theta_{\vdash})(\theta_{\vdash}(d))\bullet(\tau\otimes\id)
((\nu\otimes\id)(\nu(b)))+(\id\otimes\theta_{\dashv})(\theta_{\vdash}(d))
\bullet(\id\otimes\tau)((\tau\otimes\id)((\nu\otimes\id)(\nu(b))))\\[-1mm]
&\;+(\id\otimes\theta_{\dashv})(\theta_{\vdash}(d))
\bullet(\id\otimes\tau)((\nu\otimes\id)(\nu(b)))
+(\id\otimes\theta_{\vdash})(\theta_{\dashv}(d))\bullet(\tau\otimes\id)
((\id\otimes\tau)((\nu\otimes\id)(\nu(b))))\\
=&\;(\theta_{\vdash}\otimes\id)(\theta_{\vdash}(d))\bullet(\nu\otimes\id)(\nu(b))
+(\theta_{\dashv}\otimes\id)(\theta_{\dashv}(d))\bullet(\tau\otimes\id)
((\id\otimes\tau)((\tau\otimes\id)((\nu\otimes\id)(\nu(b)))))\\
&\;+(\theta_{\dashv}\otimes\id)(\theta_{\dashv}(d))\bullet(\tau\otimes\id)
((\id\otimes\tau)((\nu\otimes\id)(\nu(b))))
+(\theta_{\vdash}\otimes\id)(\theta_{\dashv}(d))\bullet(\id\otimes\tau)
((\nu\otimes\id)(\nu(b)))\\[-1mm]
&\;+(\theta_{\dashv}\otimes\id)(\theta_{\vdash}(d))
\bullet(\tau\otimes\id)((\nu\otimes\id)(\nu(b)))
+(\theta_{\vdash}\otimes\id)(\theta_{\dashv}(d))\bullet(\id\otimes\tau)
((\tau\otimes\id)((\nu\otimes\id)(\nu(b))))\\
=&\;(\theta_{\vdash}\otimes\id)(\theta_{\vdash}(d))\bullet(\nu\otimes\id)(\nu(b))
+(\theta_{\dashv}\otimes\id)(\theta_{\dashv}(d))\bullet(\tau\otimes\id)
((\id\otimes\tau)((\tau\otimes\id)((\id\otimes\nu)(\nu(b)))))\\
&\;+(\theta_{\dashv}\otimes\id)(\theta_{\vdash}(d))
\bullet(\tau\otimes\id)((\nu\otimes\id)(\nu(b)))
+(\theta_{\vdash}\otimes\id)(\theta_{\dashv}(d))\bullet(\id\otimes\tau)
((\tau\otimes\id)((\id\otimes\nu)(\nu(b))))\\
=&\;(\Delta\otimes\id)(\Delta(d\otimes b)).
\end{align*}
Thus, $(D\otimes B, \Delta)$ is a coassociative coalgebra.
\end{proof}

Recall that a bilinear form $\varpi(-,-)$ on a Zinbiel algebra $(B, \diamond)$
is called {\bf invariant} if
$$
\varpi(b_{1}\diamond b_{2},\; b_{3})=\varpi(b_{2},\; b_{1}\diamond b_{3}+b_{3}\diamond b_{1}),
$$
for any $b_{1}, b_{2}, b_{3}\in B$. A Zinbiel algebra $(B, \diamond)$ with a nondegenerate
skew-symmetric invariant bilinear form $\varpi(-, -)$ is called a {\bf quadratic Zinbiel
algebra} and denoted by $(B, \diamond, \varpi)$. The bilinear form $\varpi(-,-)$ can
naturally expand to the tensor product $B\otimes B\otimes\cdots\otimes B$, i.e.,
$$
\varpi(-,-):\qquad (\underbrace{B\otimes\cdots\otimes B}_{\mbox{\tiny $k$-fold}})\otimes
(\underbrace{B\otimes\cdots\otimes B}_{\mbox{\tiny $k$-fold}})\longrightarrow\Bbbk,
$$
$\varpi(b_{1}\otimes b_{2}\otimes\cdots\otimes b_{k},\ \ b'_{1}\otimes b'_{2}
\otimes\cdots\otimes b'_{k})=\prod_{i=1}^{k}\varpi(b_{i}, b'_{i})$
for any $b_{1}, b_{2},\cdots, b_{k}, b'_{1}, b'_{2},\cdots, b'_{k}\in B$.
Then $\varpi(-,-)$ on $B\otimes B\otimes\cdots\otimes B$ is also a
nondegenerate bilinear form.

\begin{lem}[\cite{HL1}]\label{lem:qZ-dual}
Let $(B, \diamond, \varpi)$ be a quadratic Zinbiel algebra. Define a linear map
$\nu_{\varpi}: B\rightarrow B\otimes B$ by $\varpi(\nu_{\varpi}(b_{1}),\;
b_{2}\otimes b_{3})=\varpi(b_{1},\; b_{2}\diamond b_{3})$, for any
$b_{1}, b_{2}, b_{3}\in B$. Then $(B, \nu_{\varpi})$ is a Zinbiel coalgebra.
\end{lem}

\begin{ex}\label{ex:qu-zib}
Let $(B, \diamond)$ be a $4$-dimensional Zinbiel algebra, where the vector space
$B={\rm span}_{\Bbbk}\{e_{1}, e_{2}, e_{3}, e_{4}\}$, the products are given by
$e_{1}\diamond e_{1}=e_{2}$, $e_{4}\diamond e_{4}=e_{3}$, $e_{1}\diamond e_{4}
=2e_{3}-e_{2}$ and $e_{4}\diamond e_{1}=2e_{2}-e_{3}$. If we define a bilinear form
$\varpi(-,-)$ on $(B, \diamond)$ by $\varpi(e_{3}, e_{1})=\varpi(e_{4}, e_{2})=1
=-\varpi(e_{1}, e_{3})=-\varpi(e_{2}, e_{4})$ and all others are zero, then
$(B, \diamond, \varpi)$ is a quadratic Zinbiel algebra. Then we get a Zinbiel coalgebra
$(B, \nu_{\varpi})$, where $\nu_{\varpi}(e_{1})=-e_{2}\otimes e_{2}-e_{2}\otimes e_{3}
+2e_{3}\otimes e_{2}$, $\nu_{\varpi}(e_{4})=e_{3}\otimes e_{3}-2e_{2}\otimes e_{3}
+e_{3}\otimes e_{2}$ and $\nu_{\varpi}(e_{2})=\nu_{\varpi}(e_{3})=0$.
\end{ex}

\begin{defi}[\cite{HLLZ,Lu}]\label{def:bidi}
A \textbf{diassociative bialgebra} is a quintuple $(D, \dashv, \vdash, \theta_{\dashv},
\theta_{\vdash})$, where $(D, \dashv, \vdash)$ is a diassociative algebra, $(D,
\theta_{\dashv}, \theta_{\vdash})$ is a diassociative coalgebra, and satisfying the
following compatibility conditions:
\begin{align}
&\qquad\qquad (\id\otimes\fl_{\dashv}(d_{1}))(\theta_{\dashv}(d_{2}))
=(\fr_{\vdash}(d_{2})\otimes\id)(\theta_{\vdash}(d_{1})),         \label{bidi1}\\
&\qquad\qquad (\fl_{\dashv}(d_{1})\otimes\id)(\theta_{\vdash}(d_{2}))
=(\id\otimes\fl_{\dashv}(d_{2}))(\tau(\theta_{\vdash}(d_{1}))),    \label{bidi2}\\
&\qquad\qquad (\fr_{\vdash}(d_{1})\otimes\id)(\tau(\theta_{\dashv}(d_{2})))
=(\id\otimes\fr_{\vdash}(d_{2}))(\theta_{\dashv}(d_{1})),          \label{bidi3}\\
& \theta_{\vdash}(d_{1}\vdash d_{2})
=(\id\otimes\fl_{\vdash}(d_{1}))(\theta_{\vdash}(d_{2}))
-((\fr_{\vdash}-\fr_{\dashv})(d_{2})\otimes\id)
((\theta_{\vdash}-\theta_{\dashv})(d_{1})),                         \label{bidi4}\\
&\quad \theta_{\vdash}(d_{1}\dashv d_{2})=(\id\otimes\fl_{\dashv}(d_{1}))
((\theta_{\vdash}-\theta_{\dashv})(d_{2}))
+(\fr_{\dashv}(d_{2})\otimes\id)(\theta_{\vdash}(d_{1})),          \label{bidi5}\\
&\quad \theta_{\dashv}(d_{1}\vdash d_{2})=(\id\otimes(\fl_{\vdash}-\fl_{\dashv})(d_{1}))
(\theta_{\dashv}(d_{2}))+(\fr_{\vdash}(d_{2})\otimes\id)
(\theta_{\dashv}(d_{1})),                                          \label{bidi6}\\
& \theta_{\dashv}(d_{1}\dashv d_{2})=(\fr_{\dashv}(d_{2})\otimes\id)
(\theta_{\dashv}(d_{1}))-(\id\otimes(\fl_{\vdash}-\fl_{\dashv})(d_{1}))
((\theta_{\vdash}-\theta_{\dashv})(d_{2})),                       \label{bidi7}\\
& ((\fl_{\vdash}-\fl_{\dashv})(d_{1})\otimes\id)(\theta_{\vdash}(d_{2}))
+(\id\otimes\fl_{\dashv}(d_{2}))(\tau(\theta_{\dashv}(d_{1})))    \label{bidi8}\\[-1mm]
&\qquad\qquad\quad=((\fr_{\vdash}-\fr_{\dashv})(d_{2})\otimes\id)
(\tau((\theta_{\vdash}-\theta_{\dashv})(d_{1})))
+(\id\otimes\fr_{\dashv}(d_{1}))(\theta_{\vdash}(d_{2})),               \nonumber\\
& (\fr_{\dashv}(d_{1})\otimes\id)(\tau(\theta_{\dashv}(d_{2})))+(\id\otimes
\fr_{\vdash}(d_{2}))((\theta_{\vdash}-\theta_{\dashv})(d_{1}))        \label{bidi9}\\[-1mm]
&\qquad\qquad\quad=((\fl_{\vdash}-\fl_{\dashv})(d_{2})\otimes\id)
((\theta_{\vdash}-\theta_{\dashv})(d_{1}))
+(\id\otimes\fl_{\vdash}(d_{1}))(\tau(\theta_{\dashv}(d_{2}))),    \nonumber
\end{align}
for any $d_{1}, d_{2}\in D$.
\end{defi}

Recall that an {\bf antisymmetric infinitesimal bialgebra}, or simply an ASI bialgebra
is a triple $(A, \cdot, \Delta)$ consisting of a vector space $A$ and linear maps
$\cdot: A\otimes A\rightarrow A$ and $\Delta: A\rightarrow A\otimes A$ such that
$(A, \cdot)$ is an associative algebra, $(A, \Delta)$ is a coassociative coalgebra
and for any $a_{1}, a_{2}\in A$,
\begin{align*}
&\qquad\quad\Delta(a_{1}\cdot a_{2})=(\fr_{A}(a_{2})\otimes\id)(\Delta(a_{1}))
+(\id\otimes\fl_{A}(a_{1}))(\Delta(a_{2})), \\
&\big(\fl_{A}(a_{1})\otimes\id-\id\otimes\fr_{A}(a_{1})\big)(\Delta(a_{2}))
=\tau\big(\big(\id\otimes\fr_{A}(a_{2})-\fl_{A}(a_{2})\otimes\id\big)(\Delta(a_{1}))\big).
\end{align*}

\begin{thm}\label{thm:dias-asbia}
Let $(D, \dashv, \vdash, \theta_{\dashv}, \theta_{\vdash})$ be a diassociative bialgebra,
$(B, \diamond, \varpi)$ be a quadratic Zinbiel algebra and $(D\otimes B, \cdot)$ be the
induced associative algebra from $(D, \dashv, \vdash)$ and $(B, \diamond)$. Define a
linear map $\Delta: D\otimes B\rightarrow(D\otimes B)\otimes(D\otimes B)$ by
\begin{align}
\Delta(d\otimes b)&=\theta_{\vdash}(d)\bullet\nu_{\varpi}(b)
+\theta_{\dashv}(d)\bullet\tau(\nu_{\varpi}(b))        \label{thes-ass}\\
&:=\sum_{(b)}\sum_{(d)}(d_{(1)}\otimes b_{(1)})\otimes(d_{(2)}\otimes b_{(2)})
+\sum_{(b)}\sum_{[d]}(d_{[1]}\otimes b_{(2)})\otimes(d_{[2]}\otimes b_{(1)}), \nonumber
\end{align}
for any $d\in D$ and $b\in B$, where $\theta_{\vdash}(d)=\sum_{(d)}
d_{(1)}\otimes d_{(2)}$, $\theta_{\dashv}(d)=\sum_{[d]}d_{[1]}\otimes d_{[2]}$ and
$\nu_{\varpi}(b)=\sum_{(b)}b_{(1)}\otimes b_{(2)}$ in the Sweedler notation. Then
$(D\otimes B, \cdot, \Delta)$ is an ASI bialgebra, which is called the {\bf ASI
bialgebra induced from $(D, \dashv, \vdash, \theta_{\dashv}, \theta_{\vdash})$
by $(B, \diamond, \varpi)$}.
\end{thm}

\begin{proof}
First, by Lemma \ref{lem:qZ-dual} and Proposition \ref{pro:dco-coas}, we get that
$(D\otimes B, \Delta)$ is a coassociative coalgebra. Second, for any $b, b'\in B$
and $d, d'\in D$, we have
\begin{align*}
\Delta((d\otimes b)\cdot(d'\otimes b'))
&=\Delta((d\vdash d')\otimes(b\diamond b')+(d\dashv d')\otimes(b'\diamond b))\\
&=\theta_{\vdash}(d\vdash d')\bullet\nu_{\varpi}(b\diamond b')
+\theta_{\dashv}(d\vdash d')\bullet\tau(\nu_{\varpi}(b\diamond b'))\\[-1mm]
&\quad+\theta_{\vdash}(d\dashv d')\bullet\nu_{\varpi}(b'\diamond b)
+\theta_{\dashv}(d\dashv d')\bullet\tau(\nu_{\varpi}(b'\diamond b)),
\end{align*}
and
\begin{align*}
&\; (\fr_{D\otimes B}(d'\otimes b')\otimes\id)(\Delta(d\otimes b))
+(\id\otimes\fl_{D\otimes B}(d\otimes b))(\Delta(d'\otimes b')) \\
=&\; (\fr_{D\otimes B}(d'\otimes b')\otimes\id)\Big(\sum_{(b)}\sum_{(d)}
(d_{(1)}\otimes b_{(1)})\otimes(d_{(2)}\otimes b_{(2)})
+\sum_{(b)}\sum_{[d]}((d_{[1]}\otimes b_{(2)})\otimes(d_{[2]}\otimes b_{(1)})\Big)\\[-2mm]
&\; +(\id\otimes\fl_{D\otimes B}(d\otimes b))\Big(\sum_{(b')}\sum_{(d')}
(d'_{(1)}\otimes b'_{(1)})\otimes(d'_{(2)}\otimes b'_{(2)})
+\sum_{(b')}\sum_{[d']}((d'_{[1]}\otimes b'_{(2)})\otimes(d'_{[2]}\otimes b'_{(1)})\Big)\\
=&\; \sum_{(b)}\sum_{(d)}\Big(((d_{(1)}\vdash d')\otimes d_{(2)})\bullet
((b_{(1)}\diamond b')\otimes b_{(2)})+((d_{(1)}\dashv d')\otimes d_{(2)})\bullet
((b'\diamond b_{(1)})\otimes b_{(2)})\Big)\\[-2mm]
&\; +\sum_{(b)}\sum_{[d]}\Big(((d_{[1]}\vdash d')\otimes d_{[2]})\bullet
((b_{(2)}\diamond b')\otimes b_{(1)})+((d_{[1]}\dashv d')\otimes d_{[2]})\bullet
((b'\diamond b_{(2)})\otimes b_{(1)})\Big)\\[-2mm]
&\; +\sum_{(b')}\sum_{(d')}\Big((d'_{(1)}\otimes(d\vdash d'_{(2)}))\bullet
(b'_{(1)}\otimes(b\diamond b'_{(2)}))+(d'_{(1)}\otimes(d\dashv d'_{(2)}))\bullet
(b'_{(1)}\otimes(b'_{(2)}\diamond b))\Big)\\[-2mm]
&\; +\sum_{(b')}\sum_{[d']}\Big((d'_{[1]}\otimes(d\vdash d'_{[2]}))\bullet
(b'_{(2)}\otimes(b\diamond b'_{(1)}))+(d'_{[1]}\otimes(d\dashv d'_{[2]}))\bullet
(b'_{(2)}\otimes(b'_{(1)}\diamond b))\Big).
\end{align*}
For any $b, b'\in B$, we denote $\Phi^{b, b'}_{1}, \Phi^{b, b'}_{2}, \Phi^{b, b'}_{3},
\Phi^{b, b'}_{4}, \Phi^{b, b'}_{5}\in B\otimes B$ such that
\begin{align*}
&\varpi(\Phi^{b, b'}_{1},\; e\otimes f)=\varpi(b,\; (e\diamond f)\diamond b'),\qquad\quad
\varpi(\Phi^{b, b'}_{2},\; e\otimes f)=\varpi(b,\; (f\diamond e)\diamond b'),\\
&\varpi(\Phi^{b, b'}_{3},\; e\otimes f)=\varpi(b,\; (b'\diamond e)\diamond f),\qquad\quad
\varpi(\Phi^{b, b'}_{4},\; e\otimes f)=\varpi(b,\; (e\diamond b')\diamond f),\\
&\varpi(\Phi^{b, b'}_{5},\; e\otimes f)=\varpi(b,\; (b'\diamond f)\diamond e
+(f\diamond b')\diamond e),
\end{align*}
for any $e, f\in B$. Note that $\varpi(\nu_{\varpi}(b\diamond b'),\; e\otimes f)
=\varpi(b\diamond b',\; e\diamond f)=-\varpi(b,\; (e\diamond f)\diamond b')$. We get
$\nu_{\varpi}(b\diamond b')=-\Phi^{b, b'}_{1}$. Similarly, we have
$\tau(\nu_{\varpi}(b\diamond b'))=-\Phi^{b, b'}_{2}$, $\nu_{\varpi}(b'\diamond b)
=\Phi^{b, b'}_{1}+\Phi^{b, b'}_{3}+\Phi^{b, b'}_{4}$, $\tau(\nu_{\varpi}(b'\diamond b)
=\Phi^{b, b'}_{2}+\Phi^{b, b'}_{5}$, $(b_{(1)}\diamond b')\otimes b_{(2)}
=-\Phi^{b, b'}_{4}$, $(b'\diamond b_{(1)})\otimes b_{(2)}=\Phi^{b, b'}_{3}+\Phi^{b, b'}_{4}$,
$(b_{(2)}\diamond b')\otimes b_{(1)}=-\Phi^{b, b'}_{1}-\Phi^{b, b'}_{2}$,
$(b'\diamond b_{(2)})\otimes b_{(1)}=\Phi^{b, b'}_{1}+\Phi^{b, b'}_{2}+\Phi^{b, b'}_{5}
=-b'_{(1)}\otimes(b\diamond b'_{(2)})$, $b'_{(1)}\otimes(b'_{(2)}\diamond b)
=\Phi^{b, b'}_{1}+\Phi^{b, b'}_{2}+\Phi^{b, b'}_{3}+\Phi^{b, b'}_{4}+\Phi^{b, b'}_{5}$,
$b'_{(2)}\otimes(b\diamond b'_{(1)})=\Phi^{b, b'}_{5}$
and $b'_{(2)}\otimes(b'_{(1)}\diamond b)=-\Phi^{b, b'}_{3}-\Phi^{b, b'}_{5}$. Then we obtain
\begin{align*}
&\; \Delta((d\otimes b)\cdot(d'\otimes b'))-(\fr_{D\otimes B}(d'\otimes b')\otimes\id)
(\Delta(d\otimes b))-(\id\otimes\fl_{D\otimes B}(d\otimes b))(\Delta(d'\otimes b'))\\
%=&\;\Big(\theta_{\vdash}(d\dashv d')-\theta_{\vdash}(d\vdash d')
%+(d_{[1]}\vdash d')\otimes d_{[2]}-(d_{[1]}\dashv d')\otimes d_{[2]}
%+(d'_{(1)}\otimes(d\vdash d'_{(2)})-d'_{(1)}\otimes(d\dashv d'_{(2)})
%\Big)\bullet\Phi^{b, b'}_{1}\\
%&\; +\Big(\theta_{\dashv}(d\dashv d')-\theta_{\dashv}(d\vdash d')
%+((d_{[1]}\vdash d')\otimes d_{[2]}-(d_{[1]}\dashv d')\otimes d_{[2]}
%+(d'_{(1)}\otimes(d\vdash d'_{(2)})-d'_{(1)}\otimes(d\dashv d'_{(2)})
%\Big)\bullet\Phi^{b, b'}_{2}\\
%&\; +\Big(\theta_{\vdash}(d\dashv d')-(d_{(1)}\dashv d')\otimes d_{(2)}
%-d'_{(1)}\otimes(d\dashv d'_{(2)})+d'_{[1]}\otimes(d\dashv d'_{[2]})
%\Big)\bullet\Phi^{b, b'}_{3}\\
%&\; +\Big(\theta_{\vdash}(d\dashv d')+(d_{[1]}\vdash d')\otimes d_{[2]}
%-(d_{(1)}\dashv d')\otimes d_{(2)}-d'_{(1)}\otimes(d\dashv d'_{(2)})
%\Big)\bullet\Phi^{b, b'}_{4}\\
%&\; +\Big(\theta_{\dashv}(d\dashv d')-(d_{[1]}\dashv d')\otimes d_{[2]}
%+(d'_{(1)}\otimes(d\vdash d'_{(2)})-d'_{(1)}\otimes(d\dashv d'_{(2)})
%-(d'_{[1]}\otimes(d\vdash d'_{[2]})+d'_{[1]}\otimes(d\dashv d'_{[2]})
%\Big)\bullet\Phi^{b, b'}_{5}\\
%
=&\;\Big(\theta_{\vdash}(d\dashv d')-\theta_{\vdash}(d\vdash d')
+(\fr_{\vdash}(d')\otimes\id)(\theta_{\dashv}(d))
-(\fr_{\dashv}(d')\otimes\id)(\theta_{\dashv}(d))\\[-2mm]
&\qquad\qquad\qquad\qquad\qquad +(\id\otimes\fl_{\vdash}(d))(\theta_{\vdash}(d'))
-(\id\otimes\fl_{\dashv}(d))(\theta_{\vdash}(d'))\Big)\bullet\Phi^{b, b'}_{1}\\[-1mm]
&\;+\Big(\theta_{\dashv}(d\dashv d')-\theta_{\dashv}(d\vdash d')
+(\fr_{\vdash}(d')\otimes\id)(\theta_{\dashv}(d))
-(\fr_{\dashv}(d')\otimes\id)(\theta_{\dashv}(d))\\[-2mm]
&\qquad\qquad\qquad\qquad\qquad +(\id\otimes\fl_{\vdash}(d))(\theta_{\vdash}(d'))
-(\id\otimes\fl_{\dashv}(d))(\theta_{\vdash}(d'))\Big)\bullet\Phi^{b, b'}_{2}
\end{align*}
\begin{align*}
&\;+\Big(\theta_{\vdash}(d\dashv d')-(\fr_{\dashv}(d')\otimes\id)(\theta_{\vdash}(d))
-(\id\otimes\fl_{\dashv}(d))(\theta_{\vdash}(d'))
+(\id\otimes\fl_{\dashv}(d))(\theta_{\dashv}(d'))\Big)\bullet\Phi^{b, b'}_{3}\\[-1mm]
&\;+\Big(\theta_{\vdash}(d\dashv d')+(\fr_{\vdash}(d')\otimes\id)(\theta_{\dashv}(d))
-(\fr_{\dashv}(d')\otimes\id)(\theta_{\vdash}(d))
-(\id\otimes\fl_{\dashv}(d))(\theta_{\vdash}(d'))\Big)\bullet\Phi^{b, b'}_{4}\\[-1mm]
&\;+\Big(\theta_{\dashv}(d\dashv d')-(\fr_{\dashv}(d')\otimes\id)(\theta_{\dashv}(d))
+(\id\otimes\fl_{\vdash}(d))(\theta_{\vdash}(d'))
-(\id\otimes\fl_{\dashv}(d))(\theta_{\vdash}(d'))\\[-2mm]
&\qquad\qquad\qquad\qquad\qquad -(\id\otimes\fl_{\vdash}(d))(\theta_{\dashv}(d'))
+(\id\otimes\fl_{\dashv}(d))(\theta_{\dashv}(d'))\Big)\bullet\Phi^{b, b'}_{5}.
\end{align*}
Since $(D, \dashv, \vdash, \theta_{\dashv}, \theta_{\vdash})$ is a diassociative bialgebra,
by Eqs. \eqref{bidi4} and \eqref{bidi5}, we get that the term for $\Phi^{b, b'}_{1}$
in the above equation is zero. By Eqs. \eqref{bidi6} and \eqref{bidi7}, we get that
the term for $\Phi^{b, b'}_{2}$ in the above equation is zero. By Eq. \eqref{bidi5},
we get that the term for $\Phi^{b, b'}_{3}$ in the above equation is zero. By Eqs.
\eqref{bidi1} and \eqref{bidi5}, we get that the term for $\Phi^{b, b'}_{2}$ in the
above equation is zero. By Eq. \eqref{bidi7}, we get that the term for $\Phi^{b, b'}_{5}$
in the above equation is zero. Thus, $\Delta((d\otimes b)\cdot(d'\otimes b'))
-(\fr_{D\otimes B}(d'\otimes b')\otimes\id)(\Delta(d\otimes b))-(\id\otimes\fl_{D\otimes B}
(d\otimes b))(\Delta(d'\otimes b'))=0$. Similarly, we also have $\big(\fl_{D\otimes B}
(d\otimes b)\otimes\id-\id\otimes\fr_{D\otimes B}(d\otimes b)\big)(\Delta(d'\otimes b'))
-\tau\big(\big(\id\otimes\fr_{D\otimes B}(d'\otimes b')-\fl_{D\otimes B}(d'\otimes b')
\otimes\id\big)(\Delta(d\otimes b))=0$.
Thus, $(D\otimes B, \cdot, \Delta)$ is an ASI bialgebra.
\end{proof}

\begin{ex}\label{ex:ind-diabi}
We consider $4$-dimensional diassociative bialgebra $(D, \dashv, \vdash, \theta_{\dashv},
\theta_{\vdash})$, where vector space $D={\rm span}_{\Bbbk}\{x_{1}, x_{2}, x_{3}, x_{4}\}$
and the nonzero products and coporducts are given by
$$
x_{2}\dashv x_{2}=x_{1},\qquad x_{3}\vdash x_{2}=-x_{4},\qquad
\theta_{\vdash}(x_{2})=-x_{1}\otimes x_{4},\qquad \theta_{\dashv}(x_{3})=x_{4}\otimes x_{4}.
$$
Let $(B={\rm span}_{\Bbbk}\{x_{1}, x_{2}, x_{3}, x_{4}\}, \diamond, \varpi)$ be the
$4$-dimensional quadratic Zinbiel algebra given in Example \ref{ex:qu-zib}. By Theorem
\ref{thm:dias-asbia}, we get a $16$-dimensional ASI bialgebra $(D\otimes B, \cdot, \Delta)$,
where the nonzero products and coproducts are given by
\begin{align*}
& (x_{2}\otimes e_{1})\cdot(x_{2}\otimes e_{1})=x_{1}\otimes e_{1},\qquad\quad
(x_{2}\otimes e_{1})\cdot(x_{2}\otimes e_{4})=2x_{1}\otimes e_{2}-x_{1}\otimes e_{3},\\
& (x_{2}\otimes e_{4})\cdot(x_{2}\otimes e_{4})=x_{1}\otimes e_{3},\qquad\quad
(x_{2}\otimes e_{4})\cdot(x_{2}\otimes e_{1})=2x_{1}\otimes e_{3}-x_{1}\otimes e_{2},\\
& (x_{3}\otimes e_{1})\cdot(x_{2}\otimes e_{1})=-x_{4}\otimes e_{2},\qquad\quad
(x_{3}\otimes e_{1})\cdot(x_{2}\otimes e_{4})=x_{4}\otimes e_{2}-2x_{4}\otimes e_{3},\\
& (x_{3}\otimes e_{4})\cdot(x_{2}\otimes e_{4})=-x_{4}\otimes e_{3},\qquad\quad
(x_{3}\otimes e_{4})\cdot(x_{2}\otimes e_{4})=x_{4}\otimes e_{3}-2x_{4}\otimes e_{2},\\
& \Delta(x_{2}\otimes e_{1})=(x_{1}\otimes e_{2})\otimes(x_{4}\otimes e_{2})
+(x_{1}\otimes e_{2})\otimes(x_{4}\otimes e_{3})
-2(x_{1}\otimes e_{3})\otimes(x_{4}\otimes e_{2}),\\
& \Delta(x_{2}\otimes e_{4})=2(x_{1}\otimes e_{2})\otimes(x_{4}\otimes e_{3})
-(x_{1}\otimes e_{3})\otimes(x_{4}\otimes e_{3})
-(x_{1}\otimes e_{3})\otimes(x_{4}\otimes e_{2}),\\
& \Delta(x_{3}\otimes e_{1})=2(x_{4}\otimes e_{3})\otimes(x_{4}\otimes e_{2})
-(x_{4}\otimes e_{2})\otimes(x_{4}\otimes e_{2})
-(x_{4}\otimes e_{2})\otimes(x_{4}\otimes e_{3}),\\
& \Delta(x_{3}\otimes e_{4})=(x_{4}\otimes e_{3})\otimes(x_{4}\otimes e_{3})
+(x_{4}\otimes e_{3})\otimes(x_{4}\otimes e_{2})
-2(x_{4}\otimes e_{2})\otimes(x_{4}\otimes e_{3}).
\end{align*}
\end{ex}

Recall that an ASI bialgebra $(A, \cdot, \Delta)$ is called {\bf coboundary}
if there exists an element $r\in A\otimes A$ such that $\Delta=\Delta_{r}$, where
\begin{align}
\Delta_{r}(a)=(\id\otimes\fl_{A}(a)-\fr_{A}(a)\otimes\id)(r), \label{ass-cobo}
\end{align}
for any $a\in A$. Let $(A, \cdot)$ be an associative algebra. An element
$r=\sum_{i}x_{i}\otimes y_{i}\in A\otimes A$ is said to be {\bf skew-symmetric}
if $r=-\tau(r)$, and it is said to be {\bf ass-invariant}
if $\Delta_{r}(a)=0$ for all $a\in A$. The equation
$$
\mathbf{A}_{r}:=r_{12}\cdot r_{13}+r_{13}\cdot r_{23}-r_{23}\cdot r_{12}=0
$$
is called the {\bf associative Yang-Baxter equation} (or $\AYBE$) in $(A, \cdot)$,
where $r_{12}\cdot r_{13}=\sum_{i,j}(x_{i}\cdot x_{j})\otimes y_{i}\otimes y_{j}$,
$r_{13}\cdot r_{23}=\sum_{i,j}x_{i}\otimes x_{j}\otimes(y_{i}\cdot y_{j})$ and
$r_{23}\cdot r_{12}=\sum_{i,j}x_{j}\otimes(x_{i}\cdot y_{j})\otimes y_{i}$.

\begin{pro}[\cite{SW}]\label{pro:quasass-bia}
Let $(A, \cdot)$ be an associative algebra, $r\in A\otimes A$ and $\Delta_{r}:
A\rightarrow A\otimes A$ be the linear map defined by Eq. \eqref{ass-cobo}.
\begin{enumerate}\itemsep=0pt
\item[$(i)$] If $r$ is a skew-symmetric solution of the $\AYBE$ in $(A, \cdot)$, then
     $(A, \cdot, \Delta_{r})$ is an ASI bialgebra, which is called a {\bf triangular ASI
     bialgebra} associated with $r$.
\item[$(ii)$] If $r$ is a solution of the $\AYBE$ in $(A, \cdot)$ and $r+\tau(r)$ is
     ass-invariant, then $(A, \cdot, \Delta_{r})$ is an ASI bialgebra,
     which is called a {\bf quasi-triangular ASI bialgebra} associated with $r$.
\end{enumerate}
\end{pro}

A quasi-triangular ASI bialgebra $(A, \cdot, \Delta_{r})$ is called a {\bf factorizable
ASI bialgebra} if $\mathcal{I}=r^{\sharp}+\tau(r)^{\sharp}: A^{\ast}\rightarrow A$
is an isomorphism of vector spaces.
Let $(D, \dashv, \vdash)$ be a diassociative algebra. We define two linear maps
$\mathcal{E}, \mathcal{F}: D\rightarrow\gl(D\otimes D)$ by
\begin{align}
\mathcal{E}(d)&:=(\fr_{\vdash}-\fr_{\dashv})(d)\otimes\id
+\id\otimes\fl_{\dashv}(d),                  \label{cobdi1}\\
\mathcal{F}(d)&:=\id\otimes(\fl_{\vdash}-\fl_{\dashv})(d)
-\fr_{\vdash}(d)\otimes\id,                  \label{cobdi2}
\end{align}
for any $d\in D$. An element $r\in D\otimes D$ is called {\bf symmetric} if $r=\tau(r)$.
If there exists an element $r\in D\otimes D$ such that $(D, \dashv, \vdash, \theta_{\dashv,r},
\theta_{\vdash,r})$ is a diassociative bialgebra, where $\theta_{\dashv,r}, \theta_{\vdash,r}:
D\rightarrow D\otimes D$ are given by
\begin{align}
\theta_{\dashv,r}(d)=\mathcal{F}(d)(r), \qquad\qquad
\theta_{\vdash,r}(d)=-\mathcal{E}(d)(r),    \label{cobdi}
\end{align}
for any $d\in D$, then $(D, \dashv, \vdash, \theta_{\dashv,r}, \theta_{\vdash,r})$ is called a
{\bf coboundary diassociative bialgebra}. Let $(D, \dashv, \vdash)$
be a diassociative algebra and $r=\sum_{i}x_{i}\otimes y_{i}\in D\otimes D$. The equation
$$
\mathbf{D}_{r}=r_{13}\vdash r_{23}-r_{23}\dashv r_{12}
-r_{12}\vdash r_{13}+r_{12}\dashv r_{13}=0
$$
is called the {\bf diassociative Yang-Baxter equation} (or $\DAYBE$) in the diassociative
algebra $(D, \dashv, \vdash)$, where $r_{13}\vdash r_{23}=\sum_{i,j}x_{i}\otimes x_{j}
\otimes(y_{i}\vdash y_{j})$, $r_{23}\dashv r_{12}=\sum_{i,j}x_{j}\otimes(x_{i}\dashv
y_{j})\otimes y_{i}$, $r_{12}\vdash r_{13}=\sum_{i,j}(x_{i}\vdash x_{j})\otimes y_{i}
\otimes y_{j}$ and $r_{12}\dashv r_{13}=\sum_{i,j}(x_{i}\dashv x_{j})\otimes y_{i}
\otimes y_{j}$.

\begin{pro}[\cite{Lu,HLLZ}]\label{pro:tri-di}
Let $(D, \dashv, \vdash)$ be a diassociative algebra, $r\in D\otimes D$ be a symmetric
solution of the $\DAYBE$ in $(D, \dashv, \vdash)$ and $\theta_{\dashv,r}, \theta_{\vdash,r}:
D\rightarrow D\otimes D$ be the linear maps defined by Eq. \eqref{cobdi}. Then
$(D, \dashv, \vdash, \theta_{\dashv,r}, \theta_{\vdash,r})$ is a diassociative bialgebra,
which is called a {\bf triangular diassociative bialgebra} associated with $r$.
\end{pro}

Triangular ASI bialgebras and triangular diassociative bialgebras are related to the
solutions of the Yang-Baxter equation. Next, we consider the relation between
the solutions of the $\DAYBE$ in a diassociative algebra and the $\AYBE$ in the induced
associative algebra. Let $(B, \diamond, \varpi)$ be a quadratic Zinbiel algebra and
$\{e_{1}, e_{2},\cdots, e_{n}\}$ be a basis of $B$. Since $\varpi(-,-)$ is skew-symmetric
nondegenerate, we get a basis $\{f_{1}, f_{2},\cdots, f_{n}\}$ of $B$, which is called the
dual basis of $\{e_{1}, e_{2},\cdots, e_{n}\}$ with respect to $\varpi(-,-)$, by
$\varpi(f_{i}, e_{j})=\delta_{ij}$, where $\delta_{ij}$ is the Kronecker delta.

\begin{pro}\label{pro:DAYBE-AYBE}
Let $(D, \vdash, \dashv)$ be a diassociative algebra, $(B, \diamond, \varpi)$ be a
quadratic Zinbiel algebra and $(D\otimes B, \cdot)$ be the induced associative algebra.
Suppose that $r=\sum_{i}x_{i}\otimes y_{i}\in D\otimes D$ is a symmetric solution of
the $\DAYBE$ in $(D, \vdash, \dashv)$. Then
\begin{align}
\widetilde{r}=\sum_{i,j}(x_{i}\otimes e_{j})\otimes(y_{i}\otimes f_{j})
\in(D\otimes B)\otimes(D\otimes B)  \label{r-max}
\end{align}
is a skew-symmetric solution of the $\AYBE$ in $(D\otimes B, \cdot)$, where $\{e_{1}, e_{2},
\cdots, e_{n}\}$ is a basis of $B$ and $\{f_{1}, f_{2},\cdots, f_{n}\}$ is the dual basis
of $\{e_{1}, e_{2},\cdots, e_{n}\}$ with respect to $\varpi(-,-)$.
\end{pro}

\begin{proof}
First, by direct calculation, we have
\begin{align*}
&\; \widetilde{r}_{12}\cdot\widetilde{r}_{13}+\widetilde{r}_{13}\cdot\widetilde{r}_{23}
-\widetilde{r}_{23}\cdot\widetilde{r}_{12}\\
=&\;\sum_{i,j}\sum_{p,q}\Big(\big((x_{i}\vdash x_{j})\otimes y_{i}\otimes y_{j}\big)
\bullet\big((e_{p}\diamond e_{q})\otimes f_{p}\otimes f_{q}\big)
+\big((x_{i}\dashv x_{j})\otimes y_{i}\otimes y_{j}\big)
\bullet\big((e_{q}\diamond e_{p})\otimes f_{p}\otimes f_{q}\big)\\[-4mm]
&\qquad\qquad+\big(x_{i}\otimes x_{j}\otimes(y_{i}\vdash y_{j})\big)\bullet
\big(e_{p}\otimes e_{q}\otimes(f_{p}\diamond f_{q})\big)
+\big(x_{i}\otimes x_{j}\otimes(y_{i}\dashv y_{j})\big)\bullet
\big(e_{p}\otimes e_{q}\otimes(f_{q}\diamond f_{p})\big)\\[-1mm]
&\qquad\qquad-\big(x_{i}\otimes(x_{j}\vdash y_{i})\otimes y_{j}\big)
\bullet\big(e_{p}\otimes(e_{q}\diamond f_{p})\otimes f_{q}\big)
-\big(x_{i}\otimes(x_{j}\dashv y_{i})\otimes y_{j}\big)
\bullet\big(e_{p}\otimes(f_{p}\diamond e_{q})\otimes f_{q}\big)\Big).
\end{align*}
Moreover, for given $s, u, v\in\{1, 2,\cdots, n\}$, we have
\begin{align*}
\varpi\Big(e_{s}\otimes e_{u}\otimes e_{v},\;
\sum_{p,q}(e_{p}\diamond e_{q})\otimes f_{p}\otimes f_{q}\Big)
&=\varpi(e_{s},\; e_{u}\diamond e_{v}),\\[-2mm]
\varpi\Big(e_{s}\otimes e_{u}\otimes e_{v},\;
\sum_{p,q}(e_{q}\diamond e_{p})\otimes f_{p}\otimes f_{q}\Big)
&=\varpi(e_{s},\; e_{v}\diamond e_{u}),\\[-2mm]
\varpi\Big(e_{s}\otimes e_{u}\otimes e_{v},\;
\sum_{p,q} e_{p}\otimes(e_{q}\diamond f_{p})\otimes f_{q}\Big)
&=\varpi(e_{s},\; e_{u}\diamond e_{v}+e_{v}\diamond e_{u}).
\end{align*}
By the nondegeneracy of $\varpi(-,-)$, we get
$$
\sum_{p,q}e_{p}\otimes(e_{q}\diamond f_{p})\otimes f_{q}
=\sum_{p,q}\big((e_{q}\diamond e_{p})\otimes f_{p}\otimes f_{q}
+(e_{p}\diamond e_{q})\otimes f_{p}\otimes f_{q}\big).
$$
Similarly, we also have $\sum_{p,q}e_{p}\otimes e_{q}\otimes(f_{p}\diamond f_{q})
=\sum_{p,q}(e_{q}\diamond e_{p})\otimes f_{p}\otimes f_{q}$, $\sum_{p,q}e_{p}
\otimes e_{q}\otimes(f_{q}\diamond f_{p})=-\sum_{p,q}e_{p}\otimes(e_{q}\diamond f_{p})
\otimes f_{q}$ and $\sum_{p,q} e_{p}\otimes(f_{p}\diamond e_{q})\otimes f_{q}=
-\sum_{p,q}(e_{p}\diamond e_{q})\otimes f_{p}\otimes f_{q}$.
%\begin{align*}
%\varpi\Big(e_{s}\otimes e_{u}\otimes e_{v},\;
%\sum_{p,q}(e_{p}\diamond e_{q})\otimes f_{p}\otimes f_{q}\Big)
%&=\varpi(e_{s},\; e_{u}\diamond e_{v}),\\[-2mm]
%\varpi\Big(e_{s}\otimes e_{u}\otimes e_{v},\;
%\sum_{p,q}(e_{q}\diamond e_{p})\otimes f_{p}\otimes f_{q}\Big)
%&=\varpi(e_{s},\; e_{v}\diamond e_{u}),\\[-2mm]
%\varpi\Big(e_{s}\otimes e_{u}\otimes e_{v},\;
%\sum_{p,q}e_{p}\otimes e_{q}\otimes(f_{p}\diamond f_{q})\Big)
%&=\varpi(e_{s},\; e_{v}\cdot e_{u}),\\[-2mm]
%\varpi\Big(e_{s}\otimes e_{u}\otimes e_{v},\;
%\sum_{p,q}e_{p}\otimes e_{q}\otimes(f_{q}\diamond f_{p})\Big)
%&=-\varpi(e_{s},\; e_{u}\diamond e_{v}+e_{v}\cdot e_{u}),\\[-2mm]
%\varpi\Big(e_{s}\otimes e_{u}\otimes e_{v},\;
%\sum_{p,q} e_{p}\otimes(e_{q}\diamond f_{p})\otimes f_{q}\Big)
%&=\varpi(e_{s},\; e_{u}\diamond e_{v}+e_{v}\cdot e_{u}),\\[-2mm]
%\varpi\Big(e_{s}\otimes e_{u}\otimes e_{v},\;
%\sum_{p,q} e_{p}\otimes(f_{p}\diamond e_{q})\otimes f_{q}\Big)
%&=-\varpi(e_{s},\; e_{u}\diamond e_{v}).
%\end{align*}
Thus, we obtain
\begin{align*}
&\; \widetilde{r}_{12}\cdot\widetilde{r}_{13}+\widetilde{r}_{13}\cdot\widetilde{r}_{23}
-\widetilde{r}_{23}\cdot\widetilde{r}_{12}\\
=&\;\sum_{i,j}\sum_{p,q}\Big((x_{i}\vdash x_{j})\otimes y_{i}\otimes y_{j}
-x_{i}\otimes x_{j}\otimes(y_{i}\dashv y_{j})
-x_{i}\otimes(x_{j}\vdash y_{i})\otimes y_{j}\qquad
\end{align*}
\begin{align*}
&\qquad\qquad\qquad\qquad\qquad\qquad+x_{i}\otimes(x_{j}\dashv y_{i})\otimes y_{j}
\Big)\bullet\big((e_{p}\diamond e_{q})\otimes f_{p}\otimes f_{q}\big)\\[-2mm]
&\quad+\Big((x_{i}\dashv x_{j})\otimes y_{i}\otimes y_{j}
+x_{i}\otimes x_{j}\otimes(y_{i}\vdash y_{j})
-x_{i}\otimes x_{j}\otimes(y_{i}\dashv y_{j})\\[-2mm]
&\qquad\qquad\qquad\qquad\qquad\qquad-x_{i}\otimes(x_{j}\vdash y_{i})\otimes y_{j}
\Big)\bullet\big((e_{q}\diamond e_{p})\otimes f_{p}\otimes f_{q}\big).
\end{align*}
Since $r$ is symmetric, we have
\begin{align*}
\mathbf{D}_{r}&=\sum_{i,j}\Big((x_{i}\vdash x_{j})\otimes y_{i}\otimes y_{j}
-x_{i}\otimes x_{j}\otimes(y_{i}\dashv y_{j})-x_{i}\otimes(x_{j}\vdash y_{i})\otimes y_{j}
+x_{i}\otimes(x_{j}\dashv y_{i})\otimes y_{j}\Big)\\[-2mm]
&=\sum_{i,j}\Big((x_{i}\dashv x_{j})\otimes y_{i}\otimes y_{j}
+x_{i}\otimes x_{j}\otimes(y_{i}\vdash y_{j})-x_{i}\otimes x_{j}\otimes(y_{i}\dashv y_{j})
-x_{i}\otimes(x_{j}\vdash y_{i})\otimes y_{j}\Big).
\end{align*}
Thus, we get that $\widetilde{r}_{12}\cdot\widetilde{r}_{13}+\widetilde{r}_{13}\cdot
\widetilde{r}_{23}-\widetilde{r}_{23}\cdot\widetilde{r}_{12}=0$ if $\mathbf{D}_{r}=0$.
Finally, for any $s, t\in\{1, 2,\cdots, n\}$, we have
$$
\varpi\Big(e_{s}\otimes e_{t},\; \sum_{j}e_{j}\otimes f_{j}\Big)
=\varpi(e_{s}, e_{t})=-\omega(e_{t}, e_{s})
=-\varpi\Big(e_{s}\otimes e_{t},\; \sum_{j}f_{j}\otimes e_{j}\Big).
$$
The nondegeneracy of $\varpi(-,-)$ yields that $\sum_{j}e_{j}\otimes f_{j}
=-\sum_{j}f_{j}\otimes e_{j}$. Hence $\widetilde{r}$ is skew-symmetric if
$r$ is symmetric. That is, $\widetilde{r}$ is skew-symmetric solution
of the $\AYBE$ in $(D\otimes B, \cdot)$.
\end{proof}

We now can give another main conclusion of this section.

\begin{thm}\label{thm:indu-triASI}
Let $(D, \dashv, \vdash, \theta_{\dashv}, \theta_{\vdash})$ be a diassociative bialgebra,
$(B, \diamond, \varpi)$ be a quadratic Zinbiel algebra and $(D\otimes B, \cdot, \Delta)$
be the induced ASI bialgebra from $(D, \dashv, \vdash, \theta_{\dashv}, \theta_{\vdash})$
by $(B, \diamond, \varpi)$. If $r\in D\otimes D$ is a symmetric solution of the $\DAYBE$
in $(D, \dashv, \vdash)$ and $\theta_{\dashv}=\theta_{\dashv,r}$, $\theta_{\vdash}=
\theta_{\vdash,r}$ are defined by Eq. \eqref{cobdi}, then $(D\otimes B, \cdot, \Delta)
=(D\otimes B, \cdot, \Delta_{\widetilde{r}})$ as ASI bialgebras, where
$\Delta_{\widetilde{r}}$ is defined by Eq. \eqref{ass-cobo} and $\widetilde{r}$ is
defined by Eq. \eqref{r-max}. That is, the induced ASI bialgebra $(D\otimes B, \cdot,
\Delta)$ is triangular if $(D, \dashv, \vdash, \theta_{\dashv}, \theta_{\vdash})$
is triangular.
\end{thm}

\begin{proof}
Let $r=\sum_{i}x_{i}\otimes y_{i}$ be a symmetric solution of the $\DAYBE$
in $(D, \dashv, \vdash)$. First, for any $d\in D$ and $b\in B$, by direct calculation, we have
\begin{align*}
\Delta(d\otimes b)&=\sum_{i}\sum_{(b)}\Big(((x_{i}\dashv d)\otimes y_{i})\bullet
(b_{(1)}\otimes b_{(2)})-((x_{i}\vdash d)\otimes y_{i})\bullet(b_{(1)}\otimes b_{(2)})\\[-5mm]
&\qquad\qquad\quad -(x_{i}\otimes(d\dashv y_{i}))\bullet(b_{(1)}\otimes b_{(2)})
+(x_{i}\otimes(d\vdash y_{i}))\bullet(b_{(2)}\otimes b_{(1)})\\[-2mm]
&\qquad\qquad\quad -(x_{i}\otimes(d\dashv y_{i}))\bullet(b_{(2)}\otimes b_{(1)})
-((x_{i}\vdash d)\otimes y_{i})\bullet(b_{(2)}\otimes b_{(1)})\Big),
\end{align*}
%\begin{align}
%\Delta(d\otimes b)&=\theta_{\vdash}(d)\bullet\nu_{\varpi}(b)
%+\theta_{\dashv}(d)\bullet\tau(\nu_{\varpi}(b))        \label{thes-ass}\\
%&:=\sum_{(b)}\sum_{(d)}(d_{(1)}\otimes b_{(1)})\otimes(d_{(2)}\otimes b_{(2)})
%+\sum_{(b)}\sum_{[d]}(d_{[1]}\otimes b_{(2)})\otimes(d_{[2]}\otimes b_{(1)}), \nonumber
%\end{align}
where $\theta_{\dashv,r}(d)=\big(\id\otimes(\fl_{\vdash}-\fl_{\dashv})(d)
-\fr_{\vdash}(d)\otimes\id\big)(r)=\sum_{i}\big(x_{i}\otimes(d\vdash y_{i})
-x_{i}\otimes(d\dashv y_{i})-(x_{i}\vdash d)\otimes y_{i}\big)$, $\theta_{\vdash,r}(d)=
\big((\fr_{\dashv}-\fr_{\vdash})(d)\otimes\id-\id\otimes\fl_{\dashv}(d)\big)(r)
=\sum_{i}\big((x_{i}\dashv d)\otimes y_{i}-(x_{i}\vdash d)\otimes y_{i}
-x_{i}\otimes(d\dashv y_{i})\big)$ and $\nu_{\varpi}(b)=\sum_{(b)}b_{(1)}
\otimes b_{(2)}$. Suppose that $\{e_{1}, e_{2},\cdots, e_{n}\}$ is a basis
of $B$ and $\{f_{1}, f_{2},\cdots, f_{n}\}$ is the dual basis of $\{e_{1}, e_{2},
\cdots, e_{n}\}$ with respect to $\varpi(-,-)$, i.e., $\varpi(f_{j}, e_{i})=\delta_{ij}$,
where $\delta_{ij}$ is the Kronecker delta. Then
\begin{align*}
\Delta_{\widetilde{r}}(d\otimes b)&=(\id\otimes\fl_{D\otimes B}(d\otimes b)
-\fr_{D\otimes B}(d\otimes b)\otimes\id)(\widetilde{r})\\
&=\sum_{i, j}\Big((x_{i}\otimes(d\vdash y_{i}))\bullet(e_{j}\otimes(b\diamond f_{j}))
+(x_{i}\otimes(d\dashv y_{i}))\bullet(e_{j}\otimes(f_{j}\diamond b))\\[-5mm]
&\qquad\quad-((x_{i}\vdash d)\otimes y_{i})\bullet((e_{j}\diamond b)\otimes f_{j})
-((x_{i}\dashv d)\otimes y_{i})\bullet((b\diamond e_{j})\otimes f_{j})\Big),
\end{align*}
where $\widetilde{r}=\sum_{i, j}(x_{i}\otimes e_{j})\otimes(y_{i}\otimes f_{j})$.
For two basis elements $e_{s}, e_{t}\in B$, since
$$
\varpi\Big(e_{s}\otimes e_{t},\; \sum_{(b)}b_{(1)}\otimes b_{(2)}\Big)
=\varpi(e_{s},\; b\diamond e_{t}-e_{t}\diamond b)
=\varpi\Big(e_{s}\otimes e_{t},\; \sum_{j}e_{j}\otimes(f_{j}\diamond b)\Big),
$$
and $\varpi(-,-)$ is nondegenerate, we get $\sum_{(b)}b_{(1)}\otimes b_{(2)}
=\sum_{j}e_{j}\otimes(f_{j}\diamond b)$. Similarly, we also have
$\sum_{(b)}b_{(2)}\otimes b_{(1)}=-\sum_{j}(e_{j}\diamond b)\otimes f_{j}$ and
$\sum_{j}(b\diamond e_{j})\otimes f_{j}=\sum_{j}e_{j}\otimes (b\diamond f_{j})
=\sum_{(b)}\big(b_{(1)}\otimes b_{(2)}-b_{(2)}\otimes b_{(1)}\big)$.
%\begin{align*}
%\varpi(e_{s}\otimes e_{t},\; \sum_{(b)}b_{(1)}\otimes b_{(2)})
%&=\varpi(b,\; e_{s}\diamond e_{t})=\varpi(e_{s},\; b\diamond e_{t}),\\
%\varpi(e_{s}\otimes e_{t},\; \sum_{(b)}b_{(2)}\otimes b_{(1)})
%&=\varpi(b,\; e_{t}\cdot e_{s})=-\varpi(e_{s},\; e_{t}\diamond b+b\diamond e_{t}),\\
%%
%\varpi(e_{s}\otimes e_{t},\; \sum_{j}e_{j}\otimes(b\diamond f_{j}))
%&=-\sum_{j}\varpi(e_{s}, e_{j})\varpi(b\diamond f_{j},\; e_{t})
%=-\sum_{j}\varpi(e_{s}, e_{t}\diamond b+b\diamond e_{t}),\\
%\varpi(e_{s}\otimes e_{t},\; \sum_{j}e_{j}\otimes(f_{j}\diamond b))
%&=-\sum_{j}\varpi(e_{s}, e_{j})\varpi(f_{j}\diamond b,\; e_{t})
%=\sum_{j}\varpi(e_{s},\; e_{t}\diamond b),\\
%\varpi(e_{s}\otimes e_{t},\; \sum_{j}(e_{j}\diamond b)\otimes f_{j})
%&=\sum_{j}\varpi(e_{j}\diamond b,\; e_{s})\varpi(f_{j}, e_{t})
%=\varpi(e_{s}\diamond b,\; e_{t})=-\varpi(e_{s},\; e_{t}\diamond b),\\
%\varpi(e_{s}\otimes e_{t},\; \sum_{j}(b\diamond e_{j})\otimes f_{j})
%&=\sum_{j}\varpi(b\diamond e_{j},\; e_{s})\varpi(f_{j}, e_{t})
%=\varpi(e_{t},\; b\diamond e_{s}+e_{s}\diamond b)
%=-\varpi(e_{s},\; b\diamond e_{t}),\\
%\end{align*}
Thus, for any $d\in D$ and $b\in B$,
\begin{align*}
\Delta(d\otimes b)-\Delta_{\widetilde{r}}(d\otimes b)
=&\; \sum_{i}\sum_{(b)}\Big(\Big((x_{i}\dashv d)\otimes y_{i}-(x_{i}\vdash d)\otimes y_{i}
-x_{i}\otimes(d\dashv y_{i})-x_{i}\otimes(d\vdash y_{i})\\[-5mm]
&\qquad\quad +x_{i}\otimes(d\dashv y_{i})+(x_{i}\vdash d)\otimes y_{i}
+x_{i}\otimes(d\vdash y_{i})-(x_{i}\dashv d)\otimes y_{i}\Big)
\bullet(b_{(1)}\otimes b_{(2)})\\[-2mm]
&\qquad\quad -\Big(x_{i}\otimes(d\vdash y_{i})-x_{i}\otimes(d\dashv y_{i})
-(x_{i}\vdash d)\otimes y_{i}-x_{i}\otimes(d\vdash y_{i})\\[-2mm]
&\qquad\quad +x_{i}\otimes(d\dashv y_{i})+(x_{i}\vdash d)\otimes y_{i}
\Big)\bullet(b_{(1)}\otimes b_{(2)}+b_{(2)}\otimes b_{(1)})\Big)\\
=&\; 0.
\end{align*}
That is, $\Delta=\Delta_{\widetilde{r}}$, and so that $(D\otimes B, \cdot, \Delta)$ is a
triangular ASI bialgebra.
\end{proof}

Let $(D, \dashv, \vdash)$ be a diassociative algebra and $r\in D\otimes D$ is a symmetric
solution of the $\DAYBE$ in $(D, \dashv, \vdash)$. Theorem \ref{thm:indu-triASI} also
gives the following commutative diagram:
$$
\xymatrix@C=3cm@R=0.5cm{
\txt{$r$ \\ {\tiny a symmetric solution of}\\ {\tiny the $\DAYBE$ in $(D, \dashv, \vdash)$}}
\ar[d]_{{\rm Pro.}~\ref{pro:DAYBE-AYBE}}\ar[r]^{{\rm Pro.}~\ref{pro:tri-di}} &
\txt{$(D, \dashv, \vdash, \theta_{\dashv,r}, \theta_{\vdash,r})$ \\
{\tiny a triangular diassociative bialgebra}}
\ar[d]_{{\rm Thm.}~\ref{thm:dias-asbia}}^{{\rm Thm.}~\ref{thm:indu-triASI}} \\
\txt{$\widetilde{r}$ \\ {\tiny a skew-summetric solution }\\ {\tiny of the $\AYBE$
in $(D\otimes B, \cdot)$}}
\ar[r]^{{\rm Pro.}~\ref{pro:quasass-bia}\qquad} &
\txt{$(D\otimes B, \cdot, \Delta)=(D\otimes B, \cdot, \Delta_{\widetilde{r}})$ \\
{\tiny a triangular ASI bialgebra}}}
$$

The $\mathcal{O}$-operator is considered to be the operator form of the solution
of the Yang-Baxter equation. We now consider the relationship between the
$\mathcal{O}$-operators of a diassociative algebra and the induced associative algebra.
Let $(A, \cdot)$ be an associative algebra and $(V, \kl, \kr)$ be a bimodule over
$(A, \cdot)$. Recall that a linear map $T: V\rightarrow A$ is called an {\bf
$\mathcal{O}$-operator of $(A, \cdot)$ associated to $(V, \kl, \kr)$} if for
any $v_{1}, v_{2}\in V$,
$$
T(v_{1})\cdot T(v_{2})=T\big(\kl(T(v_{1}))(v_{2})+\kr(T(v_{2}))(v_{1})\big).
$$

\begin{pro}[\cite{Bai}]\label{pro:o-ass}
Let $(A, \cdot)$ be an associative algebra and $r\in A\otimes A$ be skew-symmetric.
Then $r$ is a solution of the $\AYBE$ in $(A, \cdot)$ if and only if
$r^{\sharp}: A^{\ast}\rightarrow A$ is an $\mathcal{O}$-operator of $(A, \cdot)$
associated to the coregular module $(A^{\ast}, -\fr_{A}^{\ast}, -\fl_{A}^{\ast})$.
\end{pro}

The $\mathcal{O}$-operator of a diassociative algebra was considered in \cite{HLLZ,Lu}.
Recall that an {\bf $\mathcal{O}$-operator of a diassociative algebra $(D, \dashv, \vdash)$
associated to a bimodule $(V, \kl_{\dashv}, \kr_{\dashv}, \kl_{\vdash}, \kr_{\vdash})$}
is a linear map $T: V\rightarrow D$ such that
\begin{align*}
T(v_{1})\dashv T(v_{2})&=T\big(\kl_{\dashv}(T(v_{1}))(v_{2})
+\kr_{\dashv}(T(v_{2}))(v_{1})\big),\\
T(v_{1})\vdash T(v_{2})&=T\big(\kl_{\vdash}(T(v_{1}))(v_{2})
+\kr_{\vdash}(T(v_{2}))(v_{1})\big),
\end{align*}
for any $v_{1}, v_{2}\in V$.

\begin{pro}[\cite{HLLZ,Lu}]\label{pro:o-dia}
Let $(D, \dashv, \vdash)$ be a diassociative algebra and $r\in D\otimes D$ be symmetric.
Then $r$ is a solution of the $\DAYBE$ in $(D, \dashv, \vdash)$ if and only if
$r^{\sharp}$ is an $\mathcal{O}$-operator of $(D, \dashv, \vdash)$ associated to the
coregular bimodule $(D^{\ast}, \fr_{\vdash}^{\ast}-\fr_{\dashv}^{\ast},
-\fl_{\vdash}^{\ast}, -\fr_{\dashv}^{\ast}, \fl_{\dashv}^{\ast}-\fl_{\vdash}^{\ast})$.
\end{pro}

For the $\mathcal{O}$-operators of a diassociative algebra and $\mathcal{O}$-operators of
the induced associative algebra, we have following conclusion.

\begin{pro}\label{pro:o-dia-ass}
Let $(D, \dashv, \vdash)$ be a diassociative algebra, $(B, \diamond)$ be a
Zinbiel algebra and $(D\otimes B, \cdot)$ be the induced associative algebra from
$(D, \dashv, \vdash)$ and $(B, \diamond)$. If $r$ is a symmetric solution
of the $\DAYBE$ in $(D, \dashv, \vdash)$, then we have the following commutative diagram:
$$
\xymatrix@C=3cm@R=0.5cm{
\txt{$r$ \\ {\tiny a symmetric solution} \\ {\tiny of the $\DAYBE$ in $(D, \dashv, \vdash)$}}
\ar[d]_-{{\rm Pro.}~\ref{pro:DAYBE-AYBE}}\ar[r]^-{{\rm Pro.}~\ref{pro:o-dia}} &
\txt{$r^{\sharp}$\\ {\tiny an $\mathcal{O}$-operator of $(D, \dashv, \vdash)$} \\
{\tiny associated to $(D^{\ast}, \fr_{\vdash}^{\ast}-\fr_{\dashv}^{\ast},
-\fl_{\vdash}^{\ast}, -\fr_{\dashv}^{\ast}, \fl_{\dashv}^{\ast}-\fl_{\vdash}^{\ast})$}}
\ar[d]^-{\mbox{$-\otimes\kappa^{\sharp}$}} \\
\txt{$\widetilde{r}$ \\ {\tiny a skew-symmetric solution} \\ {\tiny of the $\AYBE$ in
$(D\otimes B, \cdot)$}} \ar[r]^-{{\rm Pro.}~\ref{pro:o-ass}}
& \txt{$\widetilde{r}^{\sharp}=r^{\sharp}\otimes\kappa^{\sharp}$ \\
{\tiny an $\mathcal{O}$-operator of $(A\otimes B, [-,-])$ } \\
{\tiny associated to $((D\otimes B)^{\ast}, -\fr^{\ast}_{D\otimes B},
-\fl^{\ast}_{D\otimes B})$}}}
$$
where $\kappa:=\sum_{j}e_{j}\otimes f_{j}\in B\otimes B$, $\{e_{1}, e_{2},\cdots, e_{n}\}$
is a basis of $B$ and $\{f_{1}, f_{2},\cdots, f_{n}\}$ is the dual basis of $\{e_{1}, e_{2},
\cdots, e_{n}\}$ with respect to $\varpi(-,-)$.
\end{pro}

\begin{proof}
Through calculation, direct verification can confirm that $\kappa^{\sharp}: B^{\ast}
\rightarrow B$, $\langle\kappa^{\sharp}(\xi_{1}),\; \xi_{2}\rangle=\langle\kappa,\;
\xi_{1}\otimes\xi_{2}\rangle$ for any $\xi_{1}, \xi_{2}\in B^{\ast}$ is a linear
isomorphism, and $\langle\kappa^{\sharp}(\xi_{1}),\; \xi_{2}\rangle=\langle\kappa,\;
\xi_{1}\otimes\xi_{2}\rangle=-\langle\kappa,\; \xi_{2}\otimes\xi_{1}\rangle
=\langle\kappa^{\sharp}(\xi_{2}),\; \xi_{1}\rangle$ since
$\sum_{j}e_{j}\otimes f_{j}=-\sum_{j}f_{j}\otimes e_{j}$.
Therefore, for any $\xi_{1}, \xi_{2}\in B^{\ast}$ and $\eta_{1}, \eta_{2}\in D^{\ast}$,
\begin{align*}
\langle\widetilde{r}^{\sharp}(\eta_{1}\otimes\xi_{1}),\; \eta_{2}\otimes\xi_{2}\rangle
&=\sum_{i,j}\langle(\eta_{1}\otimes\xi_{1})\otimes(\eta_{2}\otimes\xi_{2}),\ \
(x_{i}\otimes e_{j})\otimes(y_{i}\otimes f_{j})\rangle\\[-2mm]
&=\Big(\sum_{j}\langle\eta_{1}, x_{i}\rangle\langle\eta_{2}, y_{i}\rangle\Big)
\Big(\sum_{i}\langle\xi_{1}, e_{j}\rangle\langle\xi_{2}, f_{j}\rangle\Big)\\[-2mm]
&=\langle r^{\sharp}(\eta_{1}),\; \eta_{2}\rangle
\langle\kappa^{\sharp}(\xi_{1}),\; \xi_{2}\rangle\\
&=\langle r^{\sharp}(\eta_{1})\otimes\kappa^{\sharp}(\xi_{1}),\;
\eta_{2}\otimes\xi_{2}\rangle.
\end{align*}
That is $\widetilde{r}^{\sharp}=r^{\sharp}\otimes\kappa^{\sharp}$.
Thus, we obtain the commutative diagram.
\end{proof}

\begin{ex}\label{ex:ind-triASI}
We consider the $4$-dimensional diassociative algebra $(D={\rm span}_{\Bbbk}\{x_{1}, x_{2},
x_{3}, x_{4}\}, \dashv, \vdash)$ given in Example \ref{ex:ind-diabi}. Then $r=x_{1}
\otimes x_{3}+x_{3}\otimes x_{1}+x_{2}\otimes x_{4}+x_{4}\otimes x_{2}$ is a
symmetric solution of the $\DAYBE$ in $(D, \dashv, \vdash)$, and the diassociative
bialgebra $(D, \dashv, \vdash, \theta_{\dashv,r}, \theta_{\vdash,r})$ is exactly the
$4$-dimensional diassociative bialgebra considered in Example \ref{ex:ind-diabi}.
Let $(B={\rm span}_{\Bbbk}\{e_{1}, e_{2}, e_{3}, e_{4}\}, \diamond, \varpi)$ be the
$4$-dimensional quadratic Zinbiel algebra given in Example \ref{ex:qu-zib}. Then
we obtain a $16$-dimensional associative algebra $(D\otimes B, \cdot)$ and we get that
\begin{align*}
\widetilde{r}&=(x_{1}\otimes e_{1})\otimes(x_{3}\otimes e_{3})
+(x_{1}\otimes e_{2})\otimes(x_{3}\otimes e_{4})
-(x_{1}\otimes e_{3})\otimes(x_{3}\otimes e_{1})
-(x_{1}\otimes e_{4})\otimes(x_{3}\otimes e_{2})\\
&\quad+(x_{3}\otimes e_{1})\otimes(x_{1}\otimes e_{3})
+(x_{3}\otimes e_{2})\otimes(x_{1}\otimes e_{4})
-(x_{3}\otimes e_{3})\otimes(x_{1}\otimes e_{1})
-(x_{3}\otimes e_{4})\otimes(x_{1}\otimes e_{2})\\
&\quad+(x_{2}\otimes e_{1})\otimes(x_{4}\otimes e_{3})
+(x_{2}\otimes e_{2})\otimes(x_{4}\otimes e_{4})
-(x_{2}\otimes e_{3})\otimes(x_{4}\otimes e_{1})
-(x_{2}\otimes e_{4})\otimes(x_{4}\otimes e_{2})\\
&\quad+(x_{4}\otimes e_{1})\otimes(x_{2}\otimes e_{3})
+(x_{4}\otimes e_{2})\otimes(x_{2}\otimes e_{4})
-(x_{4}\otimes e_{3})\otimes(x_{2}\otimes e_{1})
-(x_{4}\otimes e_{4})\otimes(x_{2}\otimes e_{2})
\end{align*}
is a skew-symmetric solution of the $\AYBE$ in $(D\otimes B, \cdot)$.
By Proposition \ref{pro:quasass-bia}, $\widetilde{r}$ induces a coproduct
$\Delta_{\widetilde{r}}$ on $D\otimes B$ such that $(D\otimes B, \cdot,
\Delta_{\widetilde{r}})$ is a triangular ASI bialgebra. By direct calculation,
it can be concluded that this triangular ASI bialgebra is exactly the
$16$-dimensional ASI bialgebra given in Example \ref{ex:ind-diabi}.
\end{ex}

%%%%%%%%%%%%%%%%%%%%%%%%%%%%%%%%%%%%%%%%%%%%%%%%%%%%%%%%%%%%%%%%%%%%%%%%%%%%%%%%%
%    section  4   From Zinbiel bialgebras to ASI
%%%%%%%%%%%%%%%%%%%%%%%%%%%%%%%%%%%%%%%%%%%%%%%%%%%%%%%%%%%%%%%%%%%%%%%%%%%%%%%%%%%%%%
\section{Infinite-dimensional ASI bialgebras via affinization of Zinbiel bialgebras}
\label{sec:zinb-ASI}
In this section, we recall the notion of a quadratic $\bz$-graded diassociative algebra, as
a $\bz$-graded diassociative algebra equipped with an invariant bilinear form. We show that
the tensor product of a finite-dimensional Zinbiel bialgebra and a quadratic
$\bz$-graded diassociative algebra can be naturally endowed with a completed ASI bialgebra.
The converse of this result also holds when the quadratic $\bz$-graded diassociative algebra
is special, giving the desired characterization of the Zinbiel bialgebra
that its affinization is a completed ASI bialgebra.

%%%%%%%%%%%%%%%%%%%%%%%%%%%%%%%%%%%%%%%%%%%%%%%%%%%%%%%%%%%%%%%%%%%%%%%%%%%%%%%%
\subsection{Affinization of Zinbiel algebras and Zinbiel coalgebras}\label{subsec:aff}
To provide the affinization characterization of Zinbiel algebra, we first recall
the notion of $\bz$-graded diassociative algebras.

\begin{defi}\label{def:zgrad-alg}
A {\bf $\bz$-graded associative algebra} (resp. {\bf $\bz$-graded diassociative algebra})
is an associative algebra $(A, \cdot)$ (resp. a diassociative algebra $(D, \dashv,
\vdash)$) with a linear decomposition $A=\oplus_{i\in\bz}A_{i}$ (resp. $D=\oplus_{i\in\bz}
D_{i}$) such that each $A_{i}$ (resp. $D_{i}$) is finite-dimensional and $A_{i}\cdot A_{j}
\subseteq A_{i+j}$ (resp. $D_{i}\dashv D_{j}\subseteq D_{i+j}\supseteq D_{i}\vdash D_{j}$)
for all $i, j\in\bz$.
\end{defi}

\begin{ex}[\cite{Lu}]\label{ex:grdia}
Let $D=\{f_{1}\partial_{1}+f_{2}\partial_{2}\mid f_{1}, f_{2}\in\Bbbk[\kt^{\pm}]\}$ and
define binary operations $\dashv, \vdash: D\otimes D\rightarrow D$ by
\begin{align*}
\kt^{i}\partial_{m}\dashv\kt^{j}\partial_{n}:=\kt^{i+j+\delta_{1,n}}\partial_{m},\qquad\qquad
\kt^{i}\partial_{m}\vdash\kt^{j}\partial_{n}:=\kt^{i+j+\delta_{1,m}}\partial_{n},
\end{align*}
for any $i, j\in\bz$ and $m, n\in\{1, 2\}$. Then $(D, \dashv, \vdash)$ is a $\bz$-graded
diassociative algebra with the linear decomposition $D=\oplus_{i\in\bz}D_{i}$, where
$$
D_{i}=\Big\{\sum_{k=1}^{2}f_{k}\partial_{k}\mid f_{k}\text{ is a homogeneous
polynomial with }\deg(f_{k})=i-1,\; k=1,2\Big\},
$$
for all $i\in\bz$.
\end{ex}

We now extend the conclusion of Proposition \ref{pro:diass-ass} to the $\bz$-graded algebras.

\begin{pro}\label{pro:aff-Zalg}
Let $(B, \diamond)$ be a finite-dimensional Zinbiel algebra and $(D=\oplus_{i\in\bz}D_{i},
\dashv, \vdash)$ be a $\bz$-graded diassociative algebra.
Define a binary operation on $D\otimes B$ by
$$
(d_{1}\otimes b_{1})\cdot(d_{2}\otimes b_{2})=(d_{1}\vdash d_{2})\otimes(b_{1}\diamond b_{2})
+(d_{1}\dashv d_{2})\otimes(b_{2}\diamond b_{1}),
$$
for any $d_{1}, d_{2}\in D$ and $b_{1}, b_{2}\in B$. Then $(D\otimes B, \cdot)$ is a
$\bz$-graded associative algebra, which is called an {\bf affine associative algebra}
from $(B, \diamond)$ by $(D=\oplus_{i\in\bz}D_{i}, \dashv, \vdash)$. Moreover, if
$(D=\oplus_{i\in\bz}D_{i}, \dashv, \vdash)$ is the $\bz$-graded diassociative algebra
given in Example \ref{ex:grdia}, then $(D\otimes B, \cdot)$ is a $\bz$-graded associative
algebra if and only if $(B, \diamond)$ is a Zinbiel algebra.
\end{pro}

\begin{proof}
First, by Proposition \ref{pro:diass-ass}, $(D\otimes B, \cdot)$ is an associative algebra.
Since $(D=\oplus_{i\in\bz}D_{i}, \dashv, \vdash)$ is $\bz$-graded, $(D\otimes B, \cdot)$
is a $\bz$-graded associative algebra. If $(D=\oplus_{i\in\bz}D_{i}, \dashv, \vdash)$ is the
$\bz$-graded diassociative algebra given in Example \ref{ex:grdia}, then we have

\begin{align*}
\Big((\partial_{1}\otimes b_{1})\cdot(\partial_{1}\otimes b_{2})\Big)
\cdot(\partial_{2}\otimes b_{3})&=\Big(\kt\partial_{1}\otimes(b_{1}\diamond b_{2}))
+\kt\partial_{1}\otimes(b_{2}\diamond b_{1})\Big)\cdot(\partial_{2}\otimes b_{3})\\
&=\kt^{2}\partial_{2}\otimes((b_{1}\diamond b_{2})\diamond b_{3})
+\kt^{2}\partial_{2}\otimes((b_{2}\diamond b_{1})\diamond b_{3})\\[-1mm]
&\quad+\kt\partial_{1}\otimes(b_{3}\diamond(b_{1}\diamond b_{2}))
+\kt\partial_{1}\otimes(b_{3}\diamond(b_{2}\diamond b_{1})),\\
(\partial_{1}\otimes b_{1})\cdot\Big((\partial_{1}\otimes b_{2})
\cdot(\partial_{2}\otimes b_{3})\Big)&=(\partial_{1}\otimes b_{1})\cdot
\Big(\kt\partial_{2}\otimes(b_{2}\diamond b_{3})
+\partial_{1}\otimes(b_{3}\diamond b_{2})\Big)\\
&=\kt^{2}\partial_{2}\otimes(b_{1}\diamond(b_{2}\diamond b_{3}))
+\kt\partial_{1}\otimes(b_{1}\diamond(b_{3}\diamond b_{2}))\\[-1mm]
&\quad+\kt\partial_{1}\otimes((b_{2}\diamond b_{3})\diamond b_{1})
+\kt\partial_{1}\otimes((b_{3}\diamond b_{2})\diamond b_{1}).
\end{align*}
Comparing the coefficients of $\kt^{2}\partial_{2}$ in equation $((\partial_{1}\otimes b_{1})
\cdot(\partial_{1}\otimes b_{2}))\cdot(\partial_{2}\otimes b_{3})=(\partial_{1}\otimes b_{1})
\cdot((\partial_{1}\otimes b_{2})\cdot(\partial_{2}\otimes b_{3}))$, we get $b_{1}\diamond
(b_{2}\diamond b_{3})=(b_{1}\diamond b_{2})\diamond b_{3}
+(b_{2}\diamond b_{1})\diamond b_{3}$. Thus, $(B, \diamond)$ is a Zinbiel algebra.
\end{proof}

Recall that a {\bf Zinbiel coalgebra} is a pair $(B, \nu)$, where $B$ is a vector space
and $\nu: B\rightarrow B\otimes B$ is linear maps such that $(\id\otimes\nu)\nu=
(\nu\otimes\id)\nu+(\tau\otimes\id)(\nu\otimes\id)\nu$. To carry out the Zinbiel
coalgebra affinization, we need to extend the codomain of the coproduct $\nu$ to allow
infinite sums. Let $U=\oplus_{i\in\bz}U_{i}$ and $V=\oplus_{j\in\bz}V_{j}$ be
$\bz$-graded vector spaces. We call the {\bf completed tensor product} of $U$ and
$V$ to be the vector space
$$
U\,\hat{\otimes}\,V:=\prod_{i,j\in\bz}U_{i}\otimes V_{j}.
$$
If $U$ and $V$ are finite-dimensional, then $U\,\hat{\otimes}\,V$ is just the usual
tensor product $U\otimes V$. In general, an element in $U\,\hat{\otimes}\,V$ is an
infinite formal sum $\sum_{i,j\in\bz}X_{ij} $ with $X_{ij}\in U_{i}\otimes V_{j}$.
So $X_{ij}=\sum_{\alpha} u_{i, \alpha}\otimes v_{j, \alpha}$ for pure tensors
$u_{i, \alpha}\otimes v_{j, \alpha}\in U_{i}\otimes V_{j}$ with $\alpha$ in a finite
index set. Thus a general term of $U\,\hat{\otimes}\,V$ is a possibly infinite sum
$\sum_{i,j,\alpha}u_{i\alpha}\otimes v_{j\alpha}$, where $i, j\in\bz$ and $\alpha$ is
in a finite index set (which might depend on $i, j$). With these notations, for linear
maps $f: U\rightarrow U'$ and $g: V\rightarrow V'$, define
$$
f\,\hat{\otimes}\,g: U\,\hat{\otimes}\,V\rightarrow U'\,\hat{\otimes}\,V',
\qquad \sum_{i,j,\alpha}u_{i,\alpha}\otimes v_{j, \alpha}\mapsto
\sum_{i,j,\alpha} f(u_{i, \alpha})\otimes g(v_{j, \alpha}).
$$
Also the twist map $\tau$ has its completion $\hat{\tau}: V\,\hat{\otimes}\,V\rightarrow
V\,\hat{\otimes}\,V$, $\sum_{i,j,\alpha}u_{i, \alpha}\otimes v_{j, \alpha}\mapsto
\sum_{i,j,\alpha}v_{j, \alpha}\otimes u_{i, \alpha}$.
Moreover, we define a (completed) coproduct to be a linear map
$\nu: V\rightarrow V\,\hat{\otimes}\,V$, $\nu(x):=\sum_{i, j, \alpha}
x_{1, i, \alpha}\otimes x_{2, j, \alpha}$. Then we have the well-defined map
$$
(\nu\,\hat{\otimes}\,\id)(\nu(x))=(\nu\,\hat{\otimes}\,\id)
\Big(\sum_{i,j,\alpha}x_{1, i, \alpha}\otimes x_{2, j, \alpha}\Big)
:=\sum_{i,j,\alpha}\nu(x_{1, i, \alpha})\otimes x_{2, j, \alpha}
\in V\,\hat{\otimes}\,V\,\hat{\otimes}\,V.
$$

\begin{defi}\label{def:ali-coa}
\begin{enumerate}
\item[$(i)$] A {\bf completed coassociative coalgebra} is a pair $(A, \Delta)$ where
    $A=\oplus_{i\in\bz}A_{i}$ is a $\bz$-graded vector space and
    $\Delta: A\rightarrow A\,\hat{\otimes}\, A$ is a linear map satisfying
    $(\Delta\,\hat{\otimes}\,\id)\Delta=(\id\,\hat{\otimes}\,\Delta)\Delta$.
\item[$(ii)$] A {\bf completed diassociative coalgebra} is a triple $(D, \theta_{\dashv},
    \theta_{\vdash})$, where $D=\oplus_{i\in\bz}D_{i}$ is a $\bz$-graded vector space and
    $\theta_{\dashv}, \theta_{\vdash}: D\rightarrow D\,\hat{\otimes}\,D$ are a linear maps
    satisfying $(\theta_{\dashv}\,\hat{\otimes}\,\id)\theta_{\dashv}
    =(\id\,\hat{\otimes}\,\theta_{\vdash})\theta_{\dashv}$,
    $(\theta_{\vdash}\,\hat{\otimes}\,\id)\theta_{\dashv}=
    (\id\,\hat{\otimes}\,\theta_{\dashv})\theta_{\vdash}$ and
    $(\theta_{\dashv}\,\hat{\otimes}\,\id)\theta_{\vdash}
    =(\theta_{\vdash}\,\hat{\otimes}\,\id)\theta_{\vdash}$.
\end{enumerate}
\end{defi}

\begin{ex}\label{ex:grdiaco}
Consider the $\bz$-graded vector space $D=\{f_{1}\partial_{1}+f_{2}\partial_{2}\mid
f_{1}, f_{2}\in\Bbbk[\kt^{\pm}]\}=\oplus_{i\in\bz}D_{i}$ given in Example \ref{ex:grdia}.
Define two linear maps $\theta_{\dashv}, \theta_{\vdash}: D\rightarrow D\,\hat{\otimes}\,D$ by
\begin{align*}
\theta_{\dashv}(\kt^{i}\partial_{1})&=\sum_{j\in\bz}
\Big(\kt^{-j}\partial_{1}\otimes\kt^{i+j}\partial_{1}
-\kt^{-j}\partial_{1}\otimes\kt^{i+j+1}\partial_{2}\Big), \qquad
\theta_{\dashv}(\kt^{i}\partial_{2})=\sum_{j\in\bz}
\Big(\kt^{-j}\partial_{2}\otimes\kt^{i+j}\partial_{1}
-\kt^{-j}\partial_{2}\otimes\kt^{i+j+1}\partial_{2}\Big), \\[-2mm]
\theta_{\vdash}(\kt^{i}\partial_{1})&=\sum_{j\in\bz}
\Big(\kt^{-j}\partial_{1}\otimes\kt^{i+j}\partial_{1}
-\kt^{-j}\partial_{2}\otimes\kt^{i+j+1}\partial_{1}\Big), \qquad
\theta_{\vdash}(\kt^{i}\partial_{2})=\sum_{j\in\bz}
\Big(\kt^{-j}\partial_{1}\otimes\kt^{i+j}\partial_{2}
-\kt^{-j}\partial_{2}\otimes\kt^{i+j+1}\partial_{2}\Big),
\end{align*}
for any $i\in\bz$. Then $(D=\oplus_{i\in\bz}D_{i}, \theta_{\dashv}, \theta_{\vdash})$
is a completed diassociative coalgebra.
\end{ex}

Now, we consider the dual version of the Zinbiel algebra affinization,
for Zinbiel coalgebras. We give the dual version of Proposition \ref{pro:aff-Zalg}.

\begin{pro}\label{pro:zco-coass}
Let $(B, \nu)$ be a finite-dimensional Zinbiel coalgebra and $(D=\oplus_{i\in\bz}D_{i},
\theta_{\dashv}, \theta_{\vdash})$ be a completed diassociative coalgebra. Define a linear
map $\Delta: D\otimes B\rightarrow(D\otimes B)\,\hat{\otimes}\,(D\otimes B)$ by
\begin{align*}
\Delta(d\otimes b)&=\theta_{\vdash}(d)\bullet\nu(b)
+\theta_{\dashv}(d)\bullet\tau(\nu(b))\\
&:=\sum_{(b)}\sum_{(i,j,\alpha)}(d_{(1,i,\alpha)}\otimes b_{(1)})
\otimes(d_{(2,j,\alpha)}\otimes b_{(2)})+\sum_{(b)}\sum_{[i,j,\alpha]}(
(d_{[1,i,\alpha]}\otimes b_{(2)})\otimes(d_{[2,j,\alpha]}\otimes b_{(1)}),
\end{align*}
for any $d\in D$ and $b\in B$, where $\theta_{\vdash}(d)=\sum_{(i,j,\alpha)}
d_{(1,i,\alpha)}\otimes d_{(2,j,\alpha)}$, $\theta_{\dashv}(d)=\sum_{[i,j,\alpha]}
d_{[1,i,\alpha]}\otimes d_{[2,j,\alpha]}$ and $\nu(b)=\sum_{(b)}b_{(1)}\otimes b_{(2)}$
in the Sweedler notation. Then $(D\otimes B, \Delta)$ is a completed coassociative coalgebra.
Moreover, if $(D=\oplus_{i\in\bz}D_{i}, \theta_{\dashv}, \theta_{\vdash})$ is the completed
diassociative coalgebra given in Example \ref{ex:grdiaco}, then $(D\otimes B, \Delta)$
is a completed coassociative coalgebra if and only if $(B, \nu)$ is a Zinbiel coalgebra.
\end{pro}

\begin{proof}
First, since $(B, \nu)$ is a Zinbiel coalgebra and $(D, \theta_{\dashv}, \theta_{\vdash})$
is a completed diassociative coalgebra, similar to the proof Proposition \ref{pro:dco-coas},
we can obtain that $(D\otimes B, \Delta)$ is a completed coassociative coalgebra.
Second, suppose $(D\otimes B, \Delta)$ is a completed coassociative coalgebra, i.e.,
\begin{align*}
%&\;(\id\,\hat{\otimes}\,\Delta)(\Delta(d\otimes b))\\
%=&\;
&(\id\,\hat{\otimes}\,\theta_{\vdash})(\theta_{\vdash}(d))\bullet(\id\otimes\nu)(\nu(b))
+(\id\,\hat{\otimes}\,\theta_{\vdash})(\theta_{\dashv}(d))
\bullet(\id\otimes\nu)((\tau(\nu(b)))\\[-1mm]
&\quad +(\id\,\hat{\otimes}\,\theta_{\dashv})(\theta_{\vdash}(d))\bullet(\id\otimes\tau)
((\id\otimes\nu)(\nu(b)))
+(\id\,\hat{\otimes}\,\theta_{\dashv})(\theta_{\dashv}(d))\bullet(\id\otimes\tau)
((\id\otimes\nu)(\tau(\nu(b))))\\
=& (\theta_{\vdash}\,\hat{\otimes}\,\id)(\theta_{\vdash}(d))\bullet(\nu\otimes\id)(\nu(b))
+(\theta_{\vdash}\,\hat{\otimes}\,\id)(\theta_{\dashv}(d))
\bullet(\nu\otimes\id)(\tau(\nu(b)))\\[-1mm]
&\quad +(\theta_{\dashv}\,\hat{\otimes}\,\id)(\theta_{\vdash}(d))\bullet(\tau\otimes\id)
((\nu\otimes\id)(\nu(b)))
+(\theta_{\dashv}\,\hat{\otimes}\,\id)(\theta_{\dashv}(d))\bullet(\tau\otimes\id)
((\nu\otimes\id)(\tau(\nu(b)))),
%=&\;(\Delta\,\hat{\otimes}\,\id)(\Delta(d\otimes b))\\
\end{align*}
for any $d\in D$ and $b\in B$. Let $(D=\oplus_{i\in\bz}D_{i}, \theta_{\dashv},
\theta_{\vdash})$ be the completed diassociative coalgebra given in Example
\ref{ex:grdiaco} and $d=\partial_{1}$. Direct calculation shows that only the
following three terms in the above equation contain $\partial_{2}\otimes\partial_{2}
\otimes\partial_{1}$:
\begin{align*}
(\id\,\hat{\otimes}\,\theta_{\vdash})(\theta_{\vdash}(\partial_{1}))
&=\sum_{i,j}\Big(\kt^{-i}\partial_{1}\otimes\kt^{-j}\partial_{1}\otimes\kt^{i+j}\partial_{1}
-\kt^{-i}\partial_{1}\otimes\kt^{-j}\partial_{2}\otimes\kt^{i+j+1}\partial_{1}\\[-5mm]
&\qquad\qquad -\kt^{-i}\partial_{1}\otimes\kt^{-j}\partial_{2}\otimes\kt^{i+j+1}\partial_{1}
+\kt^{-i}\partial_{2}\otimes\kt^{-j}\partial_{2}\otimes\kt^{i+j+2}\partial_{1}\Big),\\[-1mm]
(\theta_{\vdash}\,\hat{\otimes}\,\id)(\theta_{\vdash}(\partial_{1}))
&=\sum_{i,j}\Big(\kt^{-i}\partial_{1}\otimes\kt^{i-j}\partial_{1}\otimes\kt^{j}\partial_{1}
-\kt^{-i}\partial_{2}\otimes\kt^{i-j+1}\partial_{1}\otimes\kt^{j}\partial_{1}\\[-5mm]
&\qquad\qquad -\kt^{-i}\partial_{1}\otimes\kt^{i-j}\partial_{2}\otimes\kt^{j+1}\partial_{1}
+\kt^{-i}\partial_{2}\otimes\kt^{i-j+1}\partial_{2}\otimes\kt^{j+1}\partial_{1}\Big),\\[-1mm]
(\theta_{\dashv}\,\hat{\otimes}\,\id)(\theta_{\vdash}(\partial_{1}))
&=\sum_{i,j}\Big(\kt^{-i}\partial_{1}\otimes\kt^{i-j}\partial_{1}\otimes\kt^{j}\partial_{1}
-\kt^{-i}\partial_{1}\otimes\kt^{i-j+1}\partial_{2}\otimes\kt^{j}\partial_{1}\\[-5mm]
&\qquad\qquad -\kt^{-i}\partial_{2}\otimes\kt^{i-j}\partial_{1}\otimes\kt^{j+1}\partial_{1}
+\kt^{-i}\partial_{2}\otimes\kt^{i-j+1}\partial_{2}\otimes\kt^{j+1}\partial_{1}\Big).
\end{align*}
Comparing the coefficient of $\partial_{2}\otimes\partial_{2}\otimes\partial_{1}$
in equation $(\id\,\hat{\otimes}\,\Delta)(\Delta(\partial_{1}\otimes b))
=(\Delta\,\hat{\otimes}\,\id)(\Delta(\partial_{1}\otimes b))$, we get
$(\id\otimes\nu)\nu=(\nu\otimes\id)\nu+(\tau\otimes\id)(\nu\otimes\id)\nu$. Thus,
$(B, \nu)$ is a Zinbiel coalgebra.
\end{proof}

%%%%%%%%%%%%%%%%%%%%%%%%%%%%%%%%%%%%%%%%%%%%%%%%%%%%%%%%%%%%%%%%%%%%%%%%%%%%%%%%
\subsection{Completed ASI bialgebras from Zinbiel bialgebras}\label{subsec:affbia}
In this subsection, we will construct a completed ASI bialgebra by affinization of
Zinbiel bialgebra. First, we extend the notion of quadratic diassociative algebra to
the case of $\bz$-graded.

\begin{defi}\label{def:quad}
A bilinear form $\omega(-, -)$ on a $\bz$-graded diassociative algebra
$(D=\oplus_{i\in\bz}D_{i}, \dashv, \vdash)$ called {\bf graded}, if there exists
some $m\in\bz$ such that $\omega(D_{i}, D_{j})=0$ when $i+j+m\neq0$.
A {\bf quadratic $\bz$-graded diassociative algebra}, denoted by $(D=\oplus_{i\in\bz}D_{i},
\dashv, \vdash, \omega)$, is a $\bz$-graded diassociative algebra together with
a skew-symmetric invariant nondegenerate graded bilinear form.
\end{defi}

In particular, if $D=D_{0}$ is finite-dimensional, then the quadratic $\bz$-graded
diassociative algebra $(D=D_{0}, \dashv$, $\vdash, \omega)$ is exactly  called a
quadratic diassociative algebra.

\begin{ex}\label{ex:qu-dia}
$(i)$ {\rm (\cite{Lu})} Let $(D=\oplus_{i\in\bz}D_{i}, \dashv, \vdash)$ be the $\bz$-graded
diassociative algebra given in Example \ref{ex:grdia}, where $D=\{f_{1}\partial_{1}
+f_{2}\partial_{2}\mid f_{1}, f_{2}\in\Bbbk[\kt^{\pm}]\}$. Define a skew-symmetric
bilinear form $\omega(-,-)$ on $(D=\oplus_{i\in\bz}D_{i}, \dashv, \vdash)$ by
$$
\omega(\kt^{i}\partial_{1},\; \kt^{j}\partial_{2})=\delta_{i+j,0},\qquad\qquad
\omega(\kt^{i}\partial_{1},\; \kt^{j}\partial_{1})
=\omega(\kt^{i}\partial_{2},\; \kt^{j}\partial_{2})=0,
$$
for any $i, j\in\bz$. Then $(D=\oplus_{i\in\bz}D_{i}, \dashv, \vdash, \omega)$ is a
quadratic $\bz$-graded diassociative algebra.
Moreover, $\{\kt^{-i}\partial_{2},\; \kt^{-i}\partial_{1}\mid i\in\bz\}$ is the dual basis
of $\{\kt^{i}\partial_{1},\; \kt^{i}\partial_{1}\mid i\in\bz\}$ with respect to
$\omega(-,-)$, consisting of homogeneous elements.

$(ii)$ Let $(D, \dashv, \vdash)$ be the $4$-dimensional diassociative algebra
with a basis $\{x_{1}, x_{2}, x_{3}, x_{4}\}$ whose the nonzero products are given by
\begin{align*}
& x_{1}\vdash x_{2}=x_{1},\qquad x_{2}\vdash x_{2}=x_{2},\qquad
x_{2}\vdash x_{3}=x_{3},\qquad x_{2}\vdash x_{4}=x_{4},\\
& x_{1}\dashv x_{2}=x_{1},\qquad x_{2}\dashv x_{2}=x_{2},\qquad
x_{3}\dashv x_{1}=x_{4},\qquad x_{4}\dashv x_{2}=x_{4}.
\end{align*}
If we define a skew-symmetric bilinear form $\omega(-,-)$ on $(D, \dashv, \vdash)$
by $\omega(x_{3}, x_{1})=\omega(x_{4}, x_{2})=1$, then $(D, \dashv, \vdash, \omega)$
is a quadratic diassociative algebra.
\end{ex}

Let $(D=\oplus_{i\in\bz}D_{i}, \dashv, \vdash, \omega)$ be a quadratic $\bz$-graded
diassociative algebra. The bilinear form $\omega(-,-)$ induces bilinear forms
$$
(\underbrace{D\,\hat{\otimes}\,D\,\hat{\otimes}\,\cdots\,\hat{\otimes}\,
D}_{n\text{-fold}})\otimes(\underbrace{D\otimes D\otimes\cdots
\otimes D}_{n\text{-fold}})\longrightarrow\Bbbk,
$$
for all $n\geq2$, which are denoted by $\hat{\omega}(-,-)$, are defined by
$$
\hat{\omega}\Big(\sum_{i_{1},\cdots,i_{n},\alpha} x_{1, i_{1}, \alpha}
\otimes\cdots\otimes x_{n, i_{n}, \alpha},\ \ y_{1}\otimes\cdots\otimes y_{n}\Big)
=\sum_{i_{1},\cdots,i_{n},\alpha}\prod_{j=1}^{n}\omega(x_{j, i_{j}, \alpha},\; y_{j}).
$$
Then, one can check that $\hat{\omega}(-,-)$ is {\bf left nondegenerate}, i.e., if
$$
\hat{\omega}\Big(\sum_{i_{1}, \cdots, i_{n},\alpha}x_{1, i_{1}, \alpha}
\otimes\cdots\otimes x_{n, i_{n}, \alpha},\ \ y_{1}\otimes\cdots\otimes y_{n}\Big)
=\hat{\omega}\Big(\sum_{i_{1},\cdots,i_{n},\alpha} z_{1, i_{1}, \alpha}
\otimes\cdots\otimes z_{n, i_{n}, \alpha},\ \ y_{1}\otimes\cdots\otimes y_{n}\Big),
$$
for all homogeneous elements $y_{1}, y_{2},\cdots, y_{n}\in D$, then
$$
\sum_{i_{1},\cdots,i_{n},\alpha} x_{1, i_{1}, \alpha}\otimes\cdots\otimes x_{n, i_{n}, \alpha}
=\sum_{i_{1},\cdots,i_{n},\alpha} z_{1, i_{1}, \alpha}
\otimes\cdots\otimes z_{n, i_{n}, \alpha}.
$$

\begin{lem}\label{lem:comd-dual}
Let $(D=\oplus_{i\in\bz}D_{i}, \dashv, \vdash, \omega)$ be a quadratic $\bz$-graded
diassociative algebra. Define two linear maps $\theta_{\dashv,\omega},
\theta_{\vdash,\omega}: D\rightarrow D\,\hat{\otimes}\,D$ by
$$
\hat{\omega}(\theta_{\dashv,\omega}(d_{1}),\; d_{2}\otimes d_{3})
=\omega(d_{1},\; d_{2}\dashv d_{3})\qquad\quad \mbox{and}\qquad\quad
\hat{\omega}(\theta_{\vdash,\omega}(d_{1}),\; d_{2}\otimes d_{3})
=\omega(d_{1},\; d_{2}\vdash d_{3})
$$
for any $d_{1}, d_{2}, d_{3}\in D$. Then $(D, \theta_{\dashv,\omega},
\theta_{\vdash,\omega})$ is a completed diassociative coalgebra.
\end{lem}

\begin{proof}
For any $d_{1}, d_{2}, d_{3}, d_{4}\in D$, we have
\begin{align*}
\hat{\omega}((\theta_{\dashv,\omega}\,\hat{\otimes}\,\id)
(\theta_{\dashv,\omega}(d_{1})),\; d_{2}\otimes d_{3}\otimes d_{4})
&=\omega(d_{1},\; (d_{2}\dashv d_{3})\dashv d_{4})\\
&=\omega(d_{1},\; d_{2}\dashv(d_{3}\vdash d_{4}))
=\hat{\omega}((\id\,\hat{\otimes}\,\theta_{\vdash,\omega})
(\theta_{\dashv,\omega}(d_{1})),\; d_{2}\otimes d_{3}\otimes d_{4}).
\end{align*}
That is $(\theta_{\dashv,\omega}\,\hat{\otimes}\,\id)\theta_{\dashv,\omega}
=(\id\,\hat{\otimes}\,\theta_{\vdash,\omega})\theta_{\dashv,\omega}$ since
$\hat{\omega}(-,-)$ is left nondegenerate. Similarly,
$(\theta_{\vdash}\,\hat{\otimes}\,\id)\theta_{\dashv}=(\id\,\hat{\otimes}\,\theta_{\dashv})
\theta_{\vdash}$ and $(\theta_{\dashv}\,\hat{\otimes}\,\id)\theta_{\vdash}
=(\theta_{\vdash}\,\hat{\otimes}\,\id)\theta_{\vdash}$. Thus, $(D, \theta_{\dashv,\omega},
\theta_{\vdash,\omega})$ is a completed diassociative coalgebra.
\end{proof}

\begin{ex}\label{ex:ind-codia}
Consider the quadratic $\bz$-graded diassociative algebra $(D=\oplus_{i\in\bz}D_{i}, \dashv,
\vdash, \omega)$ given in Example \ref{ex:grdia}. Then the induced completed
diassociative coalgebra $(D=\oplus_{i\in\bz}D_{i}, \theta_{\dashv,\omega},
\theta_{\vdash,\omega})$ is precisely the completed diassociative coalgebra
 given in Example \ref{ex:grdiaco}.
\end{ex}

Next, we extend the Zinbiel algebra affinization and the Zinbiel coalgebra
affinization to Zinbiel bialgebras. We first present the notions of Zinbiel bialgebra
and completed ASI bialgebra. Let $(B, \diamond)$ be a Zinbiel algebra. Recall that
a {\bf bimodule} over $(B, \diamond)$ is a triple $(V, \bar{\kl}, \bar{\kr})$,
where $V$ is a vector space, $\bar{\kl}, \bar{\kr}: B\rightarrow\gl(V)$ are linear maps such
that the following equalities hold for all $b_{1}, b_{2}\in B$,
\begin{align*}
&\qquad\qquad \bar{\kl}_{B}(b_{2})\bar{\kl}_{B}(b_{1})
=\bar{\kl}_{B}(b_{2}\diamond b_{1})+\bar{\kl}_{B}(b_{1}\diamond b_{2}),\\
&\bar{\kr}_{B}(b_{1}\diamond b_{2})=\bar{\kr}_{B}(b_{2})\bar{\kr}_{B}(b_{1})
+\bar{\kr}_{B}(b_{2})\bar{\kl}_{B}(b_{1})=\bar{\kl}_{B}(b_{1})\bar{\kr}_{B}(b_{2}).
\end{align*}
Define the left multiplication map $\bar{\fl}_{B}: B\rightarrow\gl(B)$ and
the right multiplication map $\bar{\fr}_{B}: B\rightarrow\gl(B)$ by
$\bar{\fl}_{B}(b_{1})(b_{2})=b_{1}\diamond b_{2}=\bar{\fr}_{B}(b_{2})(b_{1})$
for any $b_{1}, b_{2}\in B$. Then $(B, \bar{\fl}_{B}, \bar{\fr}_{B})$ is a bimodule
over $(B, \diamond)$, which is called the {\bf regular bimodule}.

\begin{defi}[\cite{Wan}]\label{def:Zinbiel-bi}
Let $(B, \diamond)$ be a Zinbiel algebra and $\nu: B\rightarrow B\otimes B$ be a linear map.
If $(B, \nu)$ is a Zinbiel coalgebra and for any $b_{1}, b_{2}\in B$,
\begin{align}
&\quad \nu(b_{1}\diamond b_{2})+\nu(b_{2}\diamond b_{1})
=(\id\otimes(\bar{\fl}_{B}+\bar{\fr}_{B})(b_{2}))(\nu(b_{1}))
+(\bar{\fl}_{B}(b_{1})\otimes\id)(\nu(b_{2})),                 \label{zbialg1}\\
& \nu(b_{1}\diamond b_{2})+\tau(\nu(b_{1}\diamond b_{2}))
=(\id\otimes\bar{\fr}_{B}(b_{2}))(\nu(b_{1}))
+(\bar{\fl}_{B}(b_{1})\otimes\id)(\nu(b_{2})+\tau(\nu(b_{2}))),  \label{zbialg2}
\end{align}
then $(B, \diamond, \nu)$ is called a {\bf Zinbiel bialgebra}.
\end{defi}

\begin{ex}\label{ex:bialg}
Let $B=\Bbbk\{e_{1}, e_{2}\}$ be a $2$-dimensional Zinbiel algebra with nonzero products:
$e_{1}\diamond e_{1}=e_{2}$. By direct computations, we can obtain that a linear map
$\nu: B\rightarrow B\otimes B$ such that $(B, \diamond, \nu)$ is a Zinbiel bialgebra
if and only if $\nu(e_{1})=ke_{2}\otimes e_{2}$ for some $k\in\Bbbk$.
\end{ex}

\begin{defi}[\cite{Hou1}]\label{def:CASIbia}
A {\bf completed ASI bialgebra} is a triple $(A, \cdot, \Delta)$ consisting of a
vector space $A$ and linear maps $\cdot: A\otimes A\rightarrow A$ and
$\Delta: A\rightarrow A\otimes A$ such that
\begin{enumerate}\itemsep=0pt
\item[$(i)$] $(A, \cdot)$ is a $\bz$-graded associative algebra;
\item[$(ii)$] $(A, \Delta)$ is a completed coassociative coalgebra;
\item[$(iii)$] for any $a_{1}, a_{2}\in A$,
\begin{align}
&\qquad\qquad\Delta(a_{1}\cdot a_{2})=(\fr_{A}(a_{2})\,\hat{\otimes}\,\id)
(\Delta(a_{1}))+(\id\,\hat{\otimes}\,\fl_{A}(a_{1}))(\Delta(a_{2})),   \label{CASI1} \\
&\big(\fl_{A}(a_{1})\,\hat{\otimes}\,\id-\id\,\hat{\otimes}\,\fr_{A}(a_{1})\big)
(\Delta(a_{2}))=\hat{\tau}\big(\big(\id\,\hat{\otimes}\,\fr_{A}(a_{2})
-\fl_{A}(a_{2})\,\hat{\otimes}\,\id\big)(\Delta(a_{1}))\big).          \label{CASI2}
\end{align}
\end{enumerate}
\end{defi}

Now we can give an affinization of finite-dimensional Zinbiel bialgebras.

\begin{thm}\label{thm:den-perm-ass}
Let $(B, \diamond, \nu)$ be a finite-dimensional Zinbiel bialgebra, $(D=\oplus_{i\in\bz}
D_{i}, \dashv, \vdash, \omega)$ be a quadratic $\bz$-graded diassociative algebra and
$(D\otimes B, \cdot)$ be the induced $\bz$-graded associative algebra from $(B, \diamond)$
by $(D=\oplus_{i\in\bz}D_{i}, \dashv, \vdash)$. Define a linear
map $\Delta: D\otimes B\rightarrow(D\otimes B)\,\hat{\otimes}\,(D\otimes B)$ by
\begin{align*}
\Delta(d\otimes b)&=\theta_{\vdash,\omega}(d)\bullet\nu(b)
+\theta_{\dashv,\omega}(d)\bullet\tau(\nu(b))\\
&=\sum_{(b)}\sum_{(i,j,\alpha)}(d_{(1,i,\alpha)}\otimes b_{(1)})
\otimes(d_{(2,j,\alpha)}\otimes b_{(2)})+\sum_{(b)}\sum_{[i,j,\alpha]}(
(d_{[1,i,\alpha]}\otimes b_{(2)})\otimes(d_{[2,j,\alpha]}\otimes b_{(1)}),
\end{align*}
for any $d\in D$ and $b\in B$, where $\theta_{\vdash,\omega}(d)=\sum_{(i,j,\alpha)}
d_{(1,i,\alpha)}\otimes d_{(2,j,\alpha)}$, $\theta_{\dashv,\omega}(d)=\sum_{[i,j,\alpha]}
d_{[1,i,\alpha]}\otimes d_{[2,j,\alpha]}$ and $\nu(b)=\sum_{(b)}b_{(1)}\otimes b_{(2)}$
in the Sweedler notation.  Then $(D\otimes B, \cdot, \Delta)$ is a completed ASI bialgebra,
which is called the {\bf completed ASI bialgebra induced from $(B, \diamond, \nu)$ by
$(D=\oplus_{i\in\bz}D_{i}, \dashv, \vdash, \omega)$}.

Moreover, if $(D=\oplus_{i\in\bz}D_{i}, \dashv, \vdash, \omega)$ is the quadratic
$\bz$-graded diassociative algebra given in Example \ref{ex:grdiaco}, then $(D\otimes B,
\cdot, \Delta)$ is a completed ASI bialgebra if and only if $(B, \diamond, \nu)$ is
a Zinbiel bialgebra.
\end{thm}

\begin{proof}
By Proposition \ref{pro:zco-coass} and Lemma \ref{lem:comd-dual}, we get
$(D\otimes B, \Delta)$ is a completed coassociative coalgebra. Thus, we only
need to show that Eqs. \eqref{CASI1} and \eqref{CASI2} hold.
For any $d, d'\in D$ and $b, b'\in B$, we have
\begin{align*}
&\;\Delta((d\otimes b)\cdot(d'\otimes b'))
-(\fr_{D\otimes B}(d'\otimes b')\,\hat{\otimes}\,\id)(\Delta(d\otimes b))
-(\id\,\hat{\otimes}\,\fl_{D\otimes B}(d\otimes b))(\Delta(d'\otimes b'))\\
=&\;\Delta((d\vdash d')\otimes(b\diamond b')+(d\dashv d')\otimes(b'\diamond b))
-(\fr_{D\otimes B}(d'\otimes b')\,\hat{\otimes}\,\id)
\big(\theta_{\vdash,\omega}(d)\bullet\nu(b)
+\theta_{\dashv,\omega}(d)\bullet\tau(\nu(b))\big)\\[-1mm]
&\quad-(\id\,\hat{\otimes}\,\fl_{D\otimes B}(d\otimes b))\big(\theta_{\vdash,\omega}(d')
\bullet\nu(b')+\theta_{\dashv,\omega}(d')\bullet\tau(\nu(b'))\big)\\
=&\;\theta_{\vdash,\omega}(d\vdash d')\bullet\nu(b\diamond b')
+\theta_{\dashv,\omega}(d\vdash d')\bullet\tau(\nu(b\diamond b'))
+\theta_{\vdash,\omega}(d\dashv d')\bullet\nu(b'\diamond b)
+\theta_{\dashv,\omega}(d\dashv d')\bullet\tau(\nu(b'\diamond b))\\
&\;-\sum_{(b)}\sum_{(i,j,\alpha)}\Big(
\big((d_{(1,i,\alpha)}\vdash d')\otimes d_{(2,j,\alpha)}\big)\bullet
\big((b_{(1)}\diamond b')\otimes b_{(2)}\big)
+\big((d_{(1,i,\alpha)}\dashv d')\otimes d_{(2,j,\alpha)}\big)\bullet
\big((b'\diamond b_{(1)})\otimes b_{(2)}\big)\Big)\\[-2mm]
&\;-\sum_{(b)}\sum_{[i,j,\alpha]}\Big(
\big((d_{[1,i,\alpha]}\vdash d')\otimes d_{[2,j,\alpha]}\big)\bullet
\big((b_{(2)}\diamond b')\otimes b_{(1)}\big)
+\big((d_{[1,i,\alpha]}\dashv d')\otimes d_{[2,j,\alpha]}\big)\bullet
\big((b'\diamond b_{(2)})\otimes b_{(1)}\big)\Big)\\[-2mm]
&\;-\sum_{(b')}\sum_{(i,j,\alpha)}\Big(
\big(d'_{(1,i,\alpha)}\otimes(d\vdash d'_{(2,j,\alpha)})\big)\bullet
\big(b'_{(1)}\otimes(b\diamond b'_{(2)})\big)
+\big(d'_{(1,i,\alpha)}\otimes(d\dashv d'_{(2,j,\alpha)})\big)\bullet
\big(b'_{(1)}\otimes(b'_{(2)}\diamond b)\big)\Big)\\[-2mm]
&\;-\sum_{(b')}\sum_{[i,j,\alpha]}\Big(
\big(d'_{[1,i,\alpha]}\otimes(d\vdash d'_{[2,j,\alpha]})\big)\bullet
\big(b'_{(2)}\otimes(b\diamond b'_{(1)})\big)
+\big(d'_{[1,i,\alpha]})\otimes(d\dashv d'_{[2,j,\alpha]})\big)\bullet
\big(b'_{(2)}\otimes(b'_{(1)}\diamond b)\big)\Big).
\end{align*}
Note that
$$
\hat{\omega}(\theta_{\vdash,\omega}(d\dashv d'),\; e\otimes f)
=\omega(d,\; (d'\vdash e)\vdash f)=\hat{\omega}\Big(\sum_{(i,j,\alpha)}
(d_{(1,i,\alpha)}\dashv d')\otimes d_{(2,j,\alpha)},\; e\otimes f\Big)
$$
for any $b, b', e, f\in B$ and $\hat{\omega}(-,-)$ is left nondegenerate.
We get $\theta_{\vdash,\omega}(d\dashv d')=\sum_{(i,j,\alpha)}(d_{(1,i,\alpha)}\dashv d')
\otimes d_{(2,j,\alpha)}$. Similarly, $\theta_{\vdash,\omega}(d\dashv d')=
\sum_{(i,j,\alpha)}d'_{(1,i,\alpha)}\otimes(d\dashv d'_{(2,j,\alpha)})$,
$\theta_{\dashv,\omega}(d\dashv d')=\sum_{[i,j,\alpha]}(d_{[1,i,\alpha]}\dashv d')
\otimes d_{[2,j,\alpha]}$, $\theta_{\vdash,\omega}(d\vdash d')=\sum_{(i,j,\alpha)}
d'_{(1,i,\alpha)}\otimes(d\vdash d'_{(2,j,\alpha)})$ and $\theta_{\dashv,\omega}(d\vdash d')
=\sum_{[i,j,\alpha]}(d_{[1,i,\alpha]}\vdash d')\otimes d_{[2,j,\alpha]}=
\sum_{[i,j,\alpha]}d'_{[1,i,\alpha]}\otimes(d\vdash d'_{[2,j,\alpha]})$.
If $(B, \diamond, \nu)$ is a Zinbiel bialgebra, i.e., Eqs. \ref{zbialg1} and
\ref{zbialg2} hold, then we obtain
\begin{align*}
&\;\Delta((d\otimes b)\cdot(d'\otimes b'))
-(\fr_{D\otimes B}(d'\otimes b')\,\hat{\otimes}\,\id)(\Delta(d\otimes b))
-(\id\,\hat{\otimes}\,\fl_{D\otimes B}(d\otimes b))(\Delta(d'\otimes b'))\\
=&\;\theta_{\vdash,\omega}(d\dashv d')\bullet
\Big(\nu(b\diamond b')+\nu(b'\diamond b)-(b'\diamond b_{(1)})\otimes b_{(2)}
-b'_{(1)}\otimes(b\diamond b'_{(2)})-b'_{(1)}\otimes(b'_{(2)}\diamond b)\Big)\\[-1mm]
&\ \ +\theta_{\dashv,\omega}(d\dashv d')\bullet
\Big(\tau(\nu(b\diamond b'))+\tau(\nu(b'\diamond b))-(b_{(2)}\diamond b')\otimes b_{(1)}
-(b'\diamond b_{(2)})\otimes b_{(1)}-b'_{(2)}\otimes(b\diamond b'_{(1)})\Big)\\[-1mm]
&\ \ +\Big(\theta_{\vdash,\omega}(d\dashv d')-\theta_{\vdash,\omega}(d\vdash d')\Big)\bullet
\Big(\nu(b\diamond b')+\tau(\nu(b\diamond b'))-(b_{(2)}\diamond b')\otimes b_{(1)}\\[-2mm]
&\qquad\qquad\qquad\qquad\qquad\qquad\qquad\qquad\qquad
-b'_{(1)}\otimes(b\diamond b'_{(2)})-b'_{(2)}\otimes(b\diamond b'_{(1)})\Big)\\
=&\; 0.
\end{align*}
Similarly, we also have $\big(\fl_{D\otimes B}(d\otimes b)\,\hat{\otimes}\,\id
-\id\,\hat{\otimes}\,\fr_{D\otimes B}(d\otimes b)\big)(\Delta(d'\otimes b'))
=\tau\big(\big(\id\,\hat{\otimes}\,\fr_{D\otimes B}(d'\otimes b')
-\fl_{D\otimes B}(d'\otimes b')\,\hat{\otimes}\,\id\big)(\Delta(d\otimes b))\big)$,
for any $d, d'\in D$ and $b, b'\in B$. Thus, $(D\otimes B, \cdot, \Delta)$ is a
completed ASI bialgebra.

Conversely, if $(D=\oplus_{i\in\bz}D_{i}, \dashv, \vdash, \omega)$ is the quadratic
$\bz$-graded diassociative algebra given in Example \ref{ex:grdiaco} and
$(D\otimes B, \cdot, \Delta)$ is a completed ASI bialgebra, then $(B, \diamond)$ is a
Zinbiel algebra and $(B, \nu)$ is a Zinbiel coalgebra by Propositions \ref{pro:aff-Zalg}
and \ref{pro:zco-coass} respectively. Now we only need to prove that $(B, \diamond, \nu)$
is a Zinbiel bialgebra. Since $(D\otimes B, \ast, \Delta)$ is a completed ASI
bialgebra, we have
\begin{align*}
0=&\;\Delta((\partial_{1}\otimes b)\cdot(\partial_{1}\otimes b'))
-(\fr_{D\otimes B}(\partial_{1}\otimes b')\,\hat{\otimes}\,\id)(\Delta(\partial_{1}\otimes b))
-(\id\,\hat{\otimes}\,\fl_{D\otimes B}(\partial_{1}\otimes b))
(\Delta(\partial_{1}\otimes b'))\\
%=&\;\Delta(\kt\partial_{1}\otimes(b\diamond b')+\kt\partial_{1}\otimes(b'\diamond b))
%-(\fr_{D\otimes B}(\partial_{1}\otimes b')\,\hat{\otimes}\,\id)
%\big(\theta_{\vdash,\omega}(\partial_{1})\bullet\nu(b)
%+\theta_{\dashv,\omega}(\partial_{1})\bullet\tau(\nu(b))\big)\\[-1mm]
%&\quad-(\id\,\hat{\otimes}\,\fl_{D\otimes B}(\partial_{1}\otimes b))
%\big(\theta_{\vdash,\omega}(\partial_{1})
%\bullet\nu(b')+\theta_{\dashv,\omega}(\partial_{1})\bullet\tau(\nu(b'))\big)\\
%
=&\;\Big(\sum_{j\in\bz}\big(\kt^{-j}\partial_{1}\otimes\kt^{j+1}\partial_{1}
-\kt^{-j}\partial_{2}\otimes\kt^{j+2}\partial_{1}\big)\Big)
\bullet\big(\nu(b\diamond b')+\nu(b'\diamond b)\big)\\[-2mm]
&\ \ +\Big(\sum_{j\in\bz}\big(\kt^{-j}\partial_{1}\otimes\kt^{j+1}\partial_{1}
-\kt^{-j}\partial_{1}\otimes\kt^{j+2}\partial_{2}\big)\Big)
\bullet\big(\tau(\nu(b\diamond b'))+\tau(\nu(b'\diamond b))\big)\\[-2mm]
&\ \ -\Big(\sum_{j\in\bz}\big(\kt^{-j+1}\partial_{1}\otimes\kt^{j}\partial_{1}
-\kt^{-j}\partial_{1}\otimes\kt^{j+1}\partial_{1}\big)\Big)
\bullet\big((\bar{\fr}_{B}(b')\otimes\id)(\nu(b))\big)\\[-2mm]
&\ \ -\Big(\sum_{j\in\bz}\big(\kt^{-j+1}\partial_{1}\otimes\kt^{j}\partial_{1}
-\kt^{-j+1}\partial_{2}\otimes\kt^{j+1}\partial_{1}\big)\Big)
\bullet\big((\bar{\fl}_{B}(b')\otimes\id)(\nu(b))\big)\\[-2mm]
&\ \ -\Big(\sum_{j\in\bz}\big(\kt^{-j+1}\partial_{1}\otimes\kt^{j}\partial_{1}
-\kt^{-j+1}\partial_{1}\otimes\kt^{j+1}\partial_{2}\big)\Big)
\bullet\big(((\bar{\fl}_{B}+\bar{\fr}_{B})(b')\otimes\id)(\tau(\nu(b)))\big)\\[-2mm]
&\ \ -\Big(\sum_{j\in\bz}\big(\kt^{-j}\partial_{1}\otimes\kt^{j+1}\partial_{1}
-\kt^{-j}\partial_{2}\otimes\kt^{j+2}\partial_{1}\big)\Big)
\bullet\big((\id\otimes(\bar{\fl}_{B}+\bar{\fr}_{B})(b))(\nu(b'))\big)\\[-2mm]
&\ \ -\Big(\sum_{j\in\bz}\big(\kt^{-j}\partial_{1}\otimes\kt^{j+1}\partial_{1}
-\kt^{-j}\partial_{1}\otimes\kt^{j+1}\partial_{1}\big)\Big)
\bullet\big((\id\otimes\bar{\fl}_{B}(b))(\nu(b'))\big)\\[-2mm]
&\ \ -\Big(\sum_{j\in\bz}\big(\kt^{-j}\partial_{1}\otimes\kt^{j+1}\partial_{1}
-\kt^{-j}\partial_{1}\otimes\kt^{j+2}\partial_{2}\big)\Big)
\bullet\big((\id\otimes\bar{\fr}_{B}(b))(\nu(b'))\big).
\end{align*}
Comparing the coefficients of $\partial_{2}\otimes\partial_{1}$ in the equation above,
we get Eq. \eqref{zbialg1} holds.
Similarly, comparing the coefficients of $\partial_{1}\otimes\partial_{1}$ and
in the equation $\Delta((\partial_{1}\otimes d)\cdot(\partial_{2}\otimes d'))
-(\fr_{D\otimes B}(\partial_{2}\otimes d')\,\hat{\otimes}\,\id)(\Delta(\partial_{1}
\otimes d))-(\id\,\hat{\otimes}\,\fl_{D\otimes B}(\partial_{1}\otimes d))
(\Delta(\partial_{2}\otimes d'))=0$, we get Eq. \eqref{zbialg2} holds.
Thus, $(B, \diamond, \nu)$ is a Zinbiel bialgebra in this case.
\end{proof}

We have constructed an infinite-dimensional ASI bialgebra using the affinization
of a Zinbiel bialgebra. Now let us return to finite-dimensional ASI bialgebras.
First, for a special case of Theorem \ref{thm:den-perm-ass}, where $(D=D_{0}, \dashv$,
$\vdash, \omega)$ is a finite-dimensional quadratic diassociative algebra, we have:

\begin{cor}\label{cor:indassbia}
Let $(B, \diamond, \nu)$ be a Zinbiel bialgebra, $(D, \dashv, \vdash, \omega)$ be a
quadratic diassociative algebra and $(D\otimes B, \cdot)$ be the induced associative
algebra from $(B, \diamond)$ and $(D, \dashv, \vdash)$. Define a linear map
$\Delta: D\otimes B\rightarrow(D\otimes B)\otimes(D\otimes B)$ by
\begin{align}
\Delta(d\otimes b)&=\theta_{\vdash,\omega}(d)\bullet\nu(b)
+\theta_{\dashv,\omega}(d)\bullet\tau(\nu(b))   \label{ind-coass}\\
&:=\sum_{(d)}\sum_{(b)}(d_{(1)}\otimes b_{(1)})\otimes(d_{(2)}\otimes b_{(2)})
+\sum_{[d]}\sum_{(b)}(d_{[1]}\otimes b_{(2)})\otimes(d_{[2]}\otimes b_{(1)}), \nonumber
\end{align}
for any $d\in D$ and $b\in B$, where $\theta_{\vdash,\omega}(d)=\sum_{(d)}d_{(1)}
\otimes d_{(2)}$, $\theta_{\dashv,\omega}(d)=\sum_{[d]}d_{[1]}\otimes d_{[2]}$ and
$\nu(b)=\sum_{(b)}b_{(1)}\otimes b_{(2)}$ in the Sweedler notation.
Then $(D\otimes B, \cdot, \Delta)$ is an ASI bialgebra, which is called the {\bf
ASI bialgebra induced from $(B, \diamond, \nu)$ by $(D, \dashv, \vdash, \omega)$}.
\end{cor}

%\begin{ex}\label{ex:D-ind-ASI}
%Let $(D=\Bbbk\{e_{1}, e_{2}\}, \prec, \succ, \theta_{\prec}, \theta_{\succ})$ be a
%dendriform $\md$-bialgebra, where the nonzero products and coproducts are given by
%$e_{1}\succ e_{1}=e_{1}$, $e_{2}\prec e_{1}=e_{2}$, $\theta_{\succ}(e_{1})=
%-e_{1}\otimes e_{1}$ and $\theta_{\succ}(e_{2})=-e_{1}\otimes e_{2}$. Let
%$(B=\Bbbk\{x_{1}, x_{2}\}, \cdot, \omega)$ be the quadratic perm algebra given in
%Example \ref{ex:prelie-indlie}. Then by Corollary \ref{cor:indassbia}, we obtain an
%ASI bialgebra $(D\otimes B, \ast, \Delta)$, where the nonzero products and coproducts
%are given by
%\begin{align*}
%&\;y_{2}y_{1}=y_{1},\qquad\qquad\; y_{2}y_{2}=y_{2},\qquad\qquad\;
%y_{3}y_{2}=y_{3},\qquad\qquad\;\; y_{4}y_{2}=y_{4},\\
%&\Delta(y_{1})=y_{1}\otimes y_{1},\qquad \Delta(y_{2})=y_{1}\otimes y_{2},\qquad
%\Delta(y_{3})=y_{1}\otimes y_{3},\qquad \Delta(y_{4})=y_{1}\otimes y_{4},
%\end{align*}
%$y_{1}:=e_{1}\otimes x_{1}$, $y_{2}:=e_{1}\otimes x_{2}$, $y_{3}:=e_{2}\otimes x_{1}$
%and $y_{4}:=e_{2}\otimes x_{2}$.
%\end{ex}

%%%%%%%%%%%%%%%%%%%%%%%%%%%%%%%%%%%%%%%%%%%%%%%%%%%%%%%%%%%%%%%%%%%%%%%%%%%%%%%%
\subsection{Quasi-triangular ASI bialgebras from quasi-triangular Zinbiel bialgebras}
\label{subsec:triASI}
In this subsection, we show that the ASI bialgebras induced from a quasi-triangular
(resp. triangular, factorizable) Zinbiel bialgebras is also quasi-triangular (resp.
triangular, factorizable). Recall that a Zinbiel bialgebra $(B, \diamond, \nu)$
is called {\bf coboundary} if there exists an element $r\in B\otimes B$ such
that the coproduct $\nu$ is given by
\begin{align}
\nu(b)=\nu_{r}(b)=\big(\bar{\fl}_{B}(b)
-\id\otimes(\bar{\fl}_{B}+\bar{\fr}_{B})(b)\otimes\id\big)(r),   \label{coboZ}
\end{align}
for any $b\in B$.

\begin{defi}[\cite{Wan}]\label{def:ZYBE}
Let $(B, \diamond)$ be a Zinbiel algebra and $r\in B\otimes B$. The equation
$$
\mathbf{Z}_{r}:=r_{13}\diamond r_{21}+r_{21}\diamond r_{13}+r_{12}\diamond r_{23}
+r_{23}\diamond r_{12}-r_{13}\diamond r_{23}-r_{23}\diamond r_{13}
-r_{13}\diamond r_{12}-r_{23}\diamond r_{21}=0
$$
is called the (classical) {\bf Zinbiel Yang-Baxter equation ($\ZYBE$)} in $(B, \diamond)$,
where $r_{13}\diamond r_{21}=\sum_{i,j}(x_{i}\diamond y_{j})\otimes x_{j}\otimes y_{i}$,
$r_{21}\diamond r_{13}=\sum_{i,j}(y_{i}\diamond x_{j})\otimes x_{i}\otimes y_{j}$,
$r_{12}\diamond r_{23}=\sum_{i,j}x_{i}\otimes(y_{i}\diamond x_{j})\otimes y_{j}$,
$r_{23}\diamond r_{12}=\sum_{i,j}x_{j}\otimes(x_{i}\diamond y_{j})\otimes y_{i}$,
$r_{13}\diamond r_{23}=\sum_{i,j}x_{i}\otimes x_{j}\otimes(y_{i}\diamond y_{j})$,
$r_{23}\diamond r_{13}=\sum_{i,j}x_{j}\otimes x_{i}\otimes(y_{i}\diamond y_{j})$,
$r_{13}\diamond r_{12}=\sum_{i,j}(x_{i}\diamond x_{j})\otimes y_{j}\otimes y_{i}$,
$r_{23}\diamond r_{21}=\sum_{i,j}y_{j}\otimes(x_{i}\diamond x_{j})\otimes y_{i}$.
\end{defi}

Let $(B, \diamond)$ be a Zinbiel algebra. An element $r\in B\otimes B$ is called {\bf
Zinb-invariant} if for any $b\in B$,
$$
\big(\bar{\fl}_{B}(b)\otimes\id-\id\otimes(\bar{\fl}_{B}+\bar{\fr}_{B})(b)\big)(r)=0.
$$
%Recall that an element $r\in B\otimes B$ is called {\bf symmetric}, if $r=\tau(r)$.
Clearly, $r$ is Zinb-invariant if $r$ is symmetric.

\begin{pro}[\cite{Wan}]\label{pro:qtr-zib}
Let $(B, \diamond)$ be a Zinbiel algebra, $r\in B\otimes B$ and $\nu_{r}: B\rightarrow
B\otimes B$ be a linear map defined by Eq. \eqref{coboZ}.
\begin{itemize}
\item[$(i)$] If $r$ is a solution of $\ZYBE$ in $(B, \diamond)$ and $r-\tau(r)$ is
     Zinb-invariant, then $(B, \diamond, \nu_{r})$ is a Zinbiel bialgebra, which is
     called a {\bf quasi-triangular Zinbiel bialgebra} associated with $r$.
\item[$(ii)$] If $r$ is a symmetric solution of $\ZYBE$ in $(B, \diamond)$, then
     $(B, \diamond, \nu_{r})$ is a Zinbiel bialgebra, which is called a {\bf triangular
     Zinbiel bialgebra} associated with $r$.
\end{itemize}
\end{pro}

\begin{ex}\label{ex:ZYBE}
Let $B=\Bbbk\{e_{1}, e_{2}\}$ be a $2$-dimensional Zinbiel algebra with nonzero products:
$e_{1}\diamond e_{1}=e_{2}$. Consider the symmetric element $r=e_{1}\otimes e_{2}
+e_{2}\otimes e_{1}\in A\otimes A$. It is easy to see that $r$ is a
symmetric solution of the $\ZYBE$ in Zinbiel algebra $(B, \diamond)$.
Therefore, the Eq. (\ref{coboZ}) give a Zinbiel bialgebra structure on $(B, \diamond)$ by
$\nu(e_{1})=e_{2}\otimes e_{2}$ and $\nu(e_{2})=0$.
\end{ex}

Triangular ASI bialgebra (resp. triangular Zinbiel bialgebra) is a special
type of ASI bialgebra (resp. Zinbiel bialgebra) that is related to the
solution of the Yang-Baxter equation. Next, we consider the relation between the
solutions of the $\ZYBE$ in a Zinbiel algebra and the solutions of the $\AYBE$
in the induced associative algebra.

\begin{pro}\label{pro:ZYBE-AYBE}
Let $(B, \diamond)$ be a Zinbiel algebra, $(D, \dashv, \vdash, \omega)$ be a quadratic
diassociative algebra and $(D\otimes B, \cdot)$ be the induced associative algebra.
Suppose that $r=\sum_{i}x_{i}\otimes y_{i}\in B\otimes B$ is a solution
of the $\ZYBE$ in $(B, \diamond)$. If $r-\tau(r)$ is Zinb-invariant, then
\begin{align}
\widehat{r}:=\sum_{i, j}(e_{j}\otimes x_{i})\otimes(f_{j}\otimes y_{i})
\in(D\otimes B)\otimes(D\otimes B)  \label{assr-max}
\end{align}
is a solution of the $\AYBE$ in $(D\otimes B, \cdot)$ and $\widehat{r}+\tau(\widehat{r})$
is ass-invariant, where $\{e_{1}, e_{2},\cdots, e_{n}\}$ is a basis of $D$ and $\{f_{1},
f_{2},\cdots, f_{n}\}$ is the dual basis of $\{e_{1}, e_{2},\cdots, e_{n}\}$ with
respect to $\omega(-,-)$.

In particular, $\widehat{r}$ is a skew-symmetric solution of the $\AYBE$ in
$(D\otimes B, \cdot)$ if $r$ is a symmetric solution of the $\ZYBE$ in $(B, \diamond)$.
\end{pro}

\begin{proof}
First, by direct calculation, we have
\begin{align*}
&\; \widehat{r}_{12}\cdot\widehat{r}_{13}+\widehat{r}_{13}\cdot\widehat{r}_{23}
-\widehat{r}_{23}\cdot\widehat{r}_{12}\\
%=&\;\sum_{i,j}\sum_{p,q}\Big(((e_{p}\otimes x_{i})\cdot(e_{q}\otimes x_{j}))
%\otimes(f_{p}\otimes y_{i})\otimes(f_{q}\otimes y_{j})+(e_{p}\otimes x_{i},)\otimes
%(e_{q}\otimes x_{j})\otimes((f_{p}\otimes y_{i})\cdot(f_{q}\otimes y_{j}))\\[-5mm]
%&\qquad\qquad-(e_{q}\otimes x_{j})\otimes((e_{p}\otimes x_{i})
%\cdot(f_{q}\otimes y_{j}))\otimes(f_{p}\otimes y_{i})\Big)\\
%=&\;\sum_{i,j}\sum_{p,q}\Big(((e_{p}\vdash e_{q})\otimes(x_{i}\diamond x_{j}))
%\otimes(f_{p}\otimes y_{i})\otimes(f_{q}\otimes y_{j})
%+((e_{p}\dashv e_{q})\otimes(x_{j}\diamond x_{i}))\otimes(f_{p}\otimes y_{i})
%\otimes(f_{q}\otimes y_{j})\\[-4mm]
%&\qquad\qquad+(e_{p}\otimes x_{i})\otimes(e_{q}\otimes x_{j})\otimes
%((f_{p}\vdash f_{q})\otimes(y_{i}\diamond y_{j}))
%+(e_{p}\otimes x_{i})\otimes(e_{q}\otimes x_{j})\otimes
%((f_{p}\dashv f_{q})\otimes(y_{j}\diamond y_{i}))\\[-1mm]
%&\qquad\qquad-(e_{q}\otimes x_{j})\otimes((e_{p}\vdash f_{q})\otimes
%(x_{i}\diamond y_{j}))\otimes(f_{p}\otimes y_{i})
%-(e_{q}\otimes x_{j})\otimes((e_{p}\dashv f_{q})\otimes
%(y_{j}\diamond x_{i}))\otimes(f_{p}\otimes y_{i})\Big).\\
=&\;\sum_{i,j}\sum_{p,q}\Big(\big((e_{p}\vdash e_{q})\otimes f_{p}\otimes f_{q}\big)
\bullet\big((x_{i}\diamond x_{j})\otimes y_{i}\otimes y_{j}\big)
+\big((e_{p}\dashv e_{q})\otimes f_{p}\otimes f_{p}\big)\bullet
\big((x_{j}\diamond x_{i})\otimes y_{i}\otimes y_{j}\big)\\[-4mm]
&\qquad\qquad+\big(e_{p}\otimes e_{q}\otimes(f_{p}\vdash f_{q})\big)
\bullet\big(x_{i}\otimes x_{j}\otimes(y_{i}\diamond y_{j})\big)
+\big(e_{p}\otimes e_{q}\otimes(f_{p}\dashv f_{p})\big)
\bullet\big(x_{i}\otimes x_{j}\otimes(y_{j}\diamond y_{i})\big)\\[-1mm]
&\qquad\qquad-\big(e_{q}\otimes(e_{p}\vdash f_{q})\otimes f_{p}\big)
\bullet\big(x_{j}\otimes(x_{i}\diamond y_{j})\otimes y_{i}\big)
-\big(e_{q}\otimes(e_{p}\dashv f_{q})\otimes f_{p}\big)
\bullet\big(x_{j}\otimes(y_{j}\diamond x_{i})\otimes y_{i}\big)\Big).
\end{align*}
%\begin{align*}
%& \hat{\omega}\Big(\sum_{p,q}(e_{p}\vdash e_{q})\otimes f_{p}\otimes f_{q},\ \
%e_{s}\otimes e_{u}\otimes e_{v}\Big)
%%=\omega(e_{u},\; e_{v}\vdash e_{s}-e_{v}\dashv e_{s}),
%=-\omega(e_{s},\; e_{u}\vdash e_{v}),\\
%& \hat{\omega}\Big(\sum_{p,q}(e_{p}\dashv e_{q})\otimes f_{p}\otimes f_{q},\ \
%e_{s}\otimes e_{u}\otimes e_{v}\Big)
%%=\omega(e_{u},\; e_{v}\vdash e_{s}),
%=-\omega(e_{s},\; e_{u}\dashv e_{v}),\\
%& \hat{\omega}\Big(\sum_{p,q}e_{p}\otimes e_{q}\otimes(f_{p}\vdash f_{q}),\ \
%e_{s}\otimes e_{u}\otimes e_{v}\Big)
%%=\omega(e_{u},\; e_{v}\dashv e_{s}),
%=\omega(e_{s},\; e_{u}\vdash e_{v}-e_{u}\dashv e_{v}),\\
%& \hat{\omega}\Big(\sum_{p,q}e_{p}\otimes e_{q}\otimes(f_{p}\dashv f_{q}),\ \
%e_{s}\otimes e_{u}\otimes e_{v}\Big)
%%=\omega(e_{u},\;e_{v}\dashv e_{s}-e_{v}\vdash e_{s}),
%=\omega(e_{s},\; e_{u}\vdash e_{v}),\\
%& \hat{\omega}\Big(\sum_{p,q}e_{q}\otimes(e_{p}\vdash f_{q})\otimes f_{p},\ \
%e_{s}\otimes e_{u}\otimes e_{v}\Big)
%%=\omega(e_{u},\; e_{v}\vdash e_{s}),
%=-\omega(e_{s},\; e_{u}\dashv e_{v}),\\
%& \hat{\omega}\Big(\sum_{p,q}e_{q}\otimes(e_{p}\dashv f_{q})\otimes f_{p},\ \
%e_{s}\otimes e_{u}\otimes e_{v}\Big)
%%=\omega(e_{u},\; e_{v}\dashv e_{s})
%=\omega(e_{s},\; e_{u}\vdash e_{v}-e_{u}\dashv e_{v}).\\
%\end{align*}
Moreover, for given $s, u, v\in\{1, 2,\cdots, n\}$, we have
$$
\hat{\omega}\Big(\sum_{p,q}(e_{p}\vdash e_{q})\otimes f_{p}\otimes f_{q},\ \
e_{s}\otimes e_{u}\otimes e_{v}\Big)
=-\omega(e_{s},\; e_{u}\vdash e_{v})
=-\hat{\omega}\Big(\sum_{p,q}e_{p}\otimes e_{q}\otimes(f_{p}\dashv f_{q}),\ \
e_{s}\otimes e_{u}\otimes e_{v}\Big).
$$
By the nondegeneracy of $\hat{\omega}(-,-)$, we get $\sum_{p,q}(e_{p}\vdash e_{q})\otimes f_{p}
\otimes f_{q}=-\sum_{p,q}e_{p}\otimes e_{q}\otimes(f_{p}\dashv f_{q})$. Similarly, we also
have $\sum_{p,q}(e_{p}\dashv e_{q})\otimes f_{p}\otimes f_{q}=\sum_{p,q}e_{q}\otimes
(e_{p}\vdash f_{q})\otimes f_{p}$ and $\sum_{p,q}e_{p}\otimes e_{q}\otimes(f_{p}\vdash
f_{q})=\sum_{p,q}e_{q}\otimes(e_{p}\dashv f_{q})\otimes f_{p}=\sum_{p,q}e_{p}\otimes e_{q}
\otimes(f_{p}\dashv f_{q})+\sum_{p,q}e_{q}\otimes(e_{p}\vdash f_{q})\otimes f_{p}$.
Thus, we obtain
\begin{align*}
&\; \widehat{r}_{12}\cdot\widehat{r}_{13}+\widehat{r}_{13}\cdot\widehat{r}_{23}
-\widehat{r}_{23}\cdot\widehat{r}_{12}\\
=&\;\sum_{i,j}\sum_{p,q}\Big(\big((e_{p}\vdash e_{q})\otimes f_{p}\otimes f_{q}\big)\bullet
\big((x_{i}\diamond x_{j})\otimes y_{i}\otimes y_{j}
-x_{i}\otimes x_{j}\otimes(y_{i}\diamond y_{j})\\[-5mm]
&\qquad\qquad\qquad\qquad\qquad\qquad\qquad\qquad
-x_{i}\otimes x_{j}\otimes(y_{j}\diamond y_{i})
+x_{j}\otimes(y_{j}\diamond x_{i})\otimes y_{i}\big)\\[-1mm]
&\qquad\qquad+\big((e_{p}\dashv e_{q})\otimes f_{p}\otimes f_{q}\big)\bullet
\big((x_{j}\diamond x_{i})\otimes y_{i}\otimes y_{j}
+x_{i}\otimes x_{j}\otimes(y_{i}\diamond y_{i})\\[-2mm]
&\qquad\qquad\qquad\qquad\qquad\qquad\qquad\qquad
-x_{j}\otimes(x_{i}\diamond y_{j})\otimes y_{i}
+x_{j}\otimes(y_{j}\diamond x_{i})\otimes y_{i}\big)\Big).
\end{align*}
Since $r-\tau(r)$ is Zinb-invariant, we can simplify $\mathbf{Z}_{r}$ into the
following forms:
\begin{align*}
\mathbf{Z}_{r}&=\sum_{i,j}\big((x_{i}\diamond x_{j})\otimes y_{i}\otimes y_{j}
-x_{i}\otimes x_{j}\otimes(y_{i}\diamond y_{j})
-x_{i}\otimes x_{j}\otimes(y_{j}\diamond y_{i})
+x_{j}\otimes(y_{j}\diamond x_{i})\otimes y_{i}\big)\\[-2mm]
&=\sum_{i,j}\big(y_{i}\otimes(x_{i}\diamond x_{j})\otimes y_{j}
-y_{i}\otimes x_{j}\otimes (y_{j}\diamond x_{i})
-y_{i}\otimes x_{j}\otimes (x_{i}\diamond y_{j})
+(y_{j}\diamond y_{i})\otimes x_{j}\otimes x_{i}\big).
\end{align*}
Therefore, we obtain
\begin{align*}
&\; \widehat{r}_{12}\cdot\widehat{r}_{13}+\widehat{r}_{13}\cdot\widehat{r}_{23}
-\widehat{r}_{23}\cdot\widehat{r}_{12}\\[-1mm]
=&\; \Big(\sum_{p,q}(e_{p}\vdash e_{q})\otimes f_{p}
\otimes f_{q}\Big)\bullet\mathbf{Z}_{r}+\Big(\sum_{p,q}(e_{p}\dashv e_{q})\otimes f_{p}
\otimes f_{q}\Big)\bullet(\tau\otimes\id)((\id\otimes\tau)(\mathbf{Z}_{r})).
\end{align*}
That is, $\widehat{r}$ is a solution of the $\AYBE$ in $(D\otimes B, \cdot)$ if $r$
is solution of the $\ZYBE$ in $(B, \diamond)$ and $r-\tau(r)$ is Zinb-invariant.

Second, we show that $\widehat{r}+\tau(\widehat{r})$ is ass-invariant in this case.
For any $b\in B$, $e_{p}\in D$, $p=1, 2,\cdots, n$, we have
\begin{align*}
&\;\big(\id\otimes\fl_{D\otimes B}(e_{p}\otimes b)-\fr_{D\otimes B}(e_{p}\otimes b)
\otimes\id\big)(\widehat{r}+\tau(\widehat{r}))\\
=&\;\sum_{i,j}\Big(
\big(e_{j}\otimes(e_{p}\vdash f_{j})\big)\bullet\big(x_{i}\otimes(b\diamond y_{i})\big)
+\big(e_{j}\otimes(e_{p}\dashv f_{j})\big)\bullet\big(x_{i}\otimes(y_{i}\diamond b)\big)
-\big((e_{j}\vdash e_{p})\otimes f_{j}\big)\bullet\big((x_{i}\diamond b)\otimes y_{i}\big)
\\[-4mm]
&\qquad
-\big((e_{j}\dashv e_{p})\otimes f_{j}\big)\bullet\big((b\diamond x_{i})\otimes y_{i}\big)
+\big(f_{j}\otimes(e_{p}\vdash e_{j})\big)\bullet\big(y_{i}\otimes(b\diamond x_{i})\big)
+\big(f_{j}\otimes(e_{p}\dashv e_{j})\big)\bullet\big(y_{i}\otimes(x_{i}\diamond b)\big)
\\[-1mm]
&\qquad
-\big((f_{j}\vdash e_{p})\otimes e_{j}\big)\bullet\big((y_{i}\diamond b)\otimes x_{i}\big)
-\big((f_{j}\dashv e_{p})\otimes e_{j}\big)\bullet\big((b\diamond y_{i})\otimes x_{i}\big)
\Big).
\end{align*}
For any $s, t\in\{1, 2,\cdots, n\}$, since
$$
\hat{\omega}\Big(\sum_{j}e_{j}\otimes(e_{p}\dashv f_{j}),\; e_{s}\otimes e_{t}\Big)
=-\omega(e_{p},\; e_{s}\vdash e_{t})
=-\hat{\omega}\Big(\sum_{j}f_{j}\otimes(e_{p}\dashv e_{j}),\; e_{s}\otimes e_{t}\Big)
$$
and $\hat{\omega}(-,-)$ is nondegenerate, we get $\sum_{j}e_{j}\otimes(e_{p}\dashv f_{j})
=-\sum_{j}f_{j}\otimes(e_{p}\dashv e_{j})$.
%\begin{align*}
%&\omega\Big(\sum_{j}e_{j}\otimes(e_{p}\vdash f_{j}),\; e_{s}\otimes e_{t}\Big)
%%=-\omega(e_{s},\; e_{t}\dashv e_{p})
%=\omega(e_{p},\; e_{s}\dashv e_{t}-e_{s}\vdash e_{t}),\\[-2mm]
%&\omega\Big(\sum_{j}e_{j}\otimes(e_{p}\dashv f_{j}),\; e_{s}\otimes e_{t}\Big)
%%=\omega(e_{s},\; e_{t}\vdash e_{p}-e_{t}\dashv e_{p})
%=-\omega(e_{p},\; e_{s}\vdash e_{t}),\\[-2mm]
%&\omega\Big(\sum_{j}(e_{j}\vdash e_{p})\otimes f_{j},\; e_{s}\otimes e_{t}\Big)
%%=-\omega(e_{s},\; e_{t}\vdash e_{p})
%=\omega(e_{p},\; e_{s}\dashv e_{t}),\\[-2mm]
%&\omega\Big(\sum_{j}(e_{j}\dashv e_{p})\otimes f_{j},\; e_{s}\otimes e_{t}\Big)
%%=-\omega(e_{s},\; e_{t}\dashv e_{p})
%=\omega(e_{p},\; e_{s}\dashv e_{t}-e_{s}\vdash e_{t}),\\[-2mm]
%%
%&\omega\Big(\sum_{j}f_{j}\otimes(e_{p}\vdash e_{j}),\; e_{s}\otimes e_{t}\Big)
%%=\omega(e_{s},\; e_{t}\dashv e_{p})
%=\omega(e_{p},\; e_{s}\vdash e_{t}-e_{s}\dashv e_{t}),\\[-2mm]
%&\omega\Big(\sum_{j}f_{j}\otimes(e_{p}\dashv e_{j}),\; e_{s}\otimes e_{t}\Big)
%%=\omega(e_{s},\; e_{t}\dashv e_{p}-e_{t}\vdash e_{p})
%=\omega(e_{p},\; e_{s}\vdash e_{t}),\\[-2mm]
%&\omega\Big(\sum_{j}(f_{j}\vdash e_{p})\otimes e_{j},\; e_{s}\otimes e_{t}\Big)
%%=\omega(e_{s},\; e_{t}\vdash e_{p})
%=-\omega(e_{p},\; e_{s}\dashv e_{t}),\\[-2mm]
%&\omega\Big(\sum_{j}(f_{j}\dashv e_{p})\otimes e_{j},\; e_{s}\otimes e_{t}\Big)
%%=\omega(e_{s},\; e_{t}\dashv e_{p})
%=\omega(e_{p},\; e_{s}\vdash e_{t}-e_{s}\dashv e_{t}).
%\end{align*}
Similarly, we also have $\sum_{j}(e_{j}\vdash e_{p})\otimes f_{j}=-\sum_{j}(f_{j}\vdash
e_{p})\otimes e_{j}$ and $\sum_{j}e_{j}\otimes(e_{p}\vdash f_{j})=\sum_{j}(e_{j}\dashv
e_{p})\otimes f_{j}=-\sum_{j}f_{j}\otimes(e_{p}\vdash e_{j})=-\sum_{j}(f_{j}\dashv e_{p})
\otimes e_{j}=\sum_{j}e_{j}\otimes(e_{p}\dashv f_{j})+\sum_{j}(e_{j}\vdash e_{p})\otimes
f_{j}$.
Therefore, we obtain
\begin{align*}
&\; \big(\id\otimes\fl_{D\otimes B}(e_{p}\otimes b)-\fr_{D\otimes B}(e_{p}\otimes b)
\otimes\id\big)(\widehat{r}+\tau(\widehat{r}))\\
=&\;\sum_{i,j}\Big(\big(e_{j}\otimes(e_{p}\dashv f_{j})\big)\bullet
\big(x_{i}\otimes(b\diamond y_{i})+x_{i}\otimes(y_{i}\diamond b)
-(b\diamond x_{i})\otimes y_{i}\\[-4mm]
&\qquad\qquad\qquad\qquad\qquad\qquad
-y_{i}\otimes(b\diamond x_{i})-y_{i}\otimes(x_{i}\diamond b)
+(b\diamond y_{i})\otimes x_{i}\big)\\
&\qquad+\big((e_{j}\vdash e_{p})\otimes f_{j}\big)\bullet
\big((b\diamond y_{i})\otimes x_{i}+(y_{i}\diamond b)\otimes x_{i}
-y_{i}\otimes(b\diamond x_{i})\\[-1mm]
&\qquad\qquad\qquad\qquad\qquad\qquad
-(b\diamond x_{i})\otimes y_{i}-(x_{i}\diamond b)\otimes y_{i}
+x_{i}\otimes(b\diamond y_{i})\big)\Big)\\
=&\; 0,
\end{align*}
since $r-\tau(r)$ is Zinb-invariant.
%\begin{align*}
%&\;\big(\id\otimes(\bar{\fl}_{B}+\bar{\fr}_{B})(b)-\bar{\fl}_{B}(b)\otimes\id\big)(r-\tau(r))\\
%=&\;\sum_{i}x_{i}\otimes(b\diamond y_{i})+x_{i}\otimes(y_{i}\diamond b)
%-(b\diamond x_{i})\otimes y_{i}-y_{i}\otimes(b\diamond x_{i})-y_{i}\otimes(x_{i}\diamond b)
%+(b\diamond y_{i})\otimes x_{i}=0
%\end{align*}
Thus, we get that $\widehat{r}+\tau(\widehat{r})$ is ass-invariant if
$r-\tau(r)$ is Zinb-invariant.

Finally, if $r$ is symmetric, then $r-\tau(r)=0$ is Zinb-invariant, and so that
$\widehat{r}$ is a solution of the $\AYBE$ in $(D\otimes B, \cdot)$.
Moreover, for any $s, t\in\{1, 2,\cdots, n\}$, we have
$$
\hat{\omega}\Big(\sum_{j}e_{j}\otimes f_{j},\; e_{s}\otimes e_{t}\Big)
=\omega(e_{t}, e_{s})
=-\hat{\omega}\Big(\sum_{j}f_{j}\otimes e_{j},\; e_{s}\otimes e_{t}\Big).
$$
The nondegeneracy of $\hat{\omega}(-,-)$ yields that $\sum_{j}e_{j}\otimes f_{j}
=-\sum_{j}f_{j}\otimes e_{j}$. Since $r$ is symmetric, we get $\widehat{r}$
is skew-symmetric. Hence, $\widehat{r}$ is a skew-symmetric solution of the
$\AYBE$ in $(D\otimes B, \cdot)$. The proof is finished.
\end{proof}

Now, by using the close relationship between solutions of the $\ZYBE$ in a
Zinbiel algebra and solutions of the $\AYBE$ in the induced associative algebra given above,
we can another main conclusion of this section.

\begin{thm}\label{thm:indu-qutriass}
Let $(B, \diamond, \nu)$ be a Zinbiel bialgebra, $(D, \dashv, \vdash, \omega)$
be a quadratic diassociative algebra and $(D\otimes B, \cdot)$ be the induced associative
algebra. If $r\in B\otimes B$, $r-\tau(r)$ is Zinb-invariant and $\nu=\nu_{r}$ is defined
by Eq. \eqref{coboZ}, then $(D\otimes B, \cdot, \Delta)=(D\otimes B, \cdot,
\Delta_{\widehat{r}})$ as ASI bialgebras, where $\Delta$ and $\Delta_{\widehat{r}}$
are defined by Eqs. \eqref{ind-coass} and \eqref{ass-cobo} respectively and
$\widehat{r}$ is defined by Eq. \eqref{assr-max}. Therefore, we have
\begin{enumerate}\itemsep=0pt
\item[$(i)$] $(D\otimes B, \cdot, \Delta)$ is quasi-triangular if
     $(B, \diamond, \nu)$ is quasi-triangular;
\item[$(ii)$] $(D\otimes B, \cdot, \Delta)$ is triangular if
     $(B, \diamond, \nu)$ is triangular;
\item[$(iii)$] $(D\otimes B, \cdot, \Delta)$ is factorizable if
     $(B, \diamond, \nu)$ is factorizable.
\end{enumerate}
\end{thm}

\begin{proof}
Let $r=\sum_{i}x_{i}\otimes y_{i}$. For any $b\in B$ and $d\in D$, we have
\begin{align*}
\Delta(d\otimes b)&=\theta_{\vdash,\omega}(d)\bullet\nu_{r}(b)
+\theta_{\dashv,\omega}(d)\bullet\tau(\nu_{r}(b)) \\
&=\sum_{i}\Big(\theta_{\vdash,\omega}(d)\bullet\big((b\diamond x_{i})\otimes y_{i}
-x_{i}\otimes(b\diamond y_{i})-x_{i}\otimes(y_{i}\diamond b)\big)\\[-4mm]
&\qquad\quad+\theta_{\dashv,\omega}(d)\bullet\big(y_{i}\otimes(b\diamond x_{i})
-(b\diamond y_{i})\otimes x_{i}-(y_{i}\diamond b)\otimes x_{i}\big)\Big)
\end{align*}
and
\begin{align*}
\Delta_{\widehat{r}}(d\otimes b)
&=(\id\otimes\fl_{A}(d\otimes b)-\fr_{A}(d\otimes b)\otimes\id)\Big(
\sum_{i,j}(e_{j}\otimes x_{i})\otimes(f_{j}\otimes y_{i})\Big)\\[-2mm]
%&=\sum_{i,j}\Big((e_{j}\otimes x_{i})\otimes((d\vdash f_{j})\otimes(b\diamond y_{i}))
%+(e_{j}\otimes x_{i})\otimes((d\dashv f_{j})\otimes(y_{i}\diamond b))\\[-5mm]
%&\qquad\quad-((e_{j}\vdash d)\otimes(x_{i}\diamond b))\otimes(f_{j}\otimes y_{i})
%-((e_{j}\dashv d)\otimes(b\diamond x_{i}))\otimes(f_{j}\otimes y_{i})\Big),\\
&=\sum_{i,j}\Big(\big(e_{j}\otimes(d\vdash f_{j})\big)\bullet\big(x_{i}\otimes
(b\diamond y_{i})\big)+\big(e_{j}\otimes(d\dashv f_{j})\big)\bullet
\big(x_{i}\otimes(y_{i}\diamond b)\big)\\[-5mm]
&\qquad\quad-\big((e_{j}\vdash d)\otimes f_{j}\big)\bullet\big((x_{i}\diamond b)
\otimes y_{i}\big)-\big((e_{j}\dashv d)\otimes f_{j}\big)\bullet
\big((b\diamond x_{i})\otimes y_{i}\big)\Big),
\end{align*}
where $\widehat{r}=\sum_{i, j}(e_{j}\otimes x_{i})\otimes(f_{j}\otimes y_{i})$,
$\{e_{1}, e_{2},\cdots, e_{n}\}$ is a basis of $B$ and $\{f_{1}, f_{2},\cdots, f_{n}\}$
is the dual basis of $\{e_{1}, e_{2},\cdots, e_{n}\}$ with respect to $\omega(-,-)$.
For any $p, q\in\{1,2,\cdots,n\}$, since
$$
\hat{\omega}(\theta_{\dashv,\omega}(d),\; e_{p}\otimes e_{q})
=\omega(d,\; e_{p}\dashv e_{q})
=\hat{\omega}\Big(\sum_{j}(e_{j}\vdash d)\otimes f_{j},\; e_{p}\otimes e_{q}\Big)
$$
and $\omega(-,-)$ is nondegenerate, we get $\theta_{\dashv,\omega}(d)=\sum_{j}(e_{j}
\vdash d)\otimes f_{j}$. Similarly, $\theta_{\vdash,\omega}(d)=-\sum_{j}e_{j}\otimes
(d\dashv f_{j})$ and $\sum_{j}e_{j}\otimes(d\vdash f_{j})=\sum_{j}(e_{j}\dashv d)\otimes
f_{j}=\theta_{\dashv,\omega}(d)-\theta_{\vdash,\omega}(d)$.
%\begin{align*}
%&\hat{\omega}(\theta_{\vdash,\omega}(d),\; e_{p}\otimes e_{q})
%=\omega(d,\; e_{p}\vdash e_{q}),\\
%&\hat{\omega}(\theta_{\dashv,\omega}(d),\; e_{p}\otimes e_{q})
%=\omega(d,\; e_{p}\dashv e_{q}),\\
%&\hat{\omega}(e_{j}\otimes(d\vdash f_{j}),\; e_{p}\otimes e_{q})
%=\omega(d,\; e_{p}\dashv e_{q}-e_{p}\vdash e_{q}),\\
%&\hat{\omega}(e_{j}\otimes(d\dashv f_{j}),\; e_{p}\otimes e_{q})
%=-\omega(d,\; e_{p}\vdash e_{q}),\\
%&\hat{\omega}((e_{j}\vdash d)\otimes f_{j},\; e_{p}\otimes e_{q})
%=\omega(d,\; e_{p}\dashv e_{q}),\\
%&\hat{\omega}((e_{j}\dashv d)\otimes f_{j},\; e_{p}\otimes e_{q})
%=\omega(d,\; e_{p}\dashv e_{q}-e_{p}\vdash e_{q}).
%\end{align*}
Thus, we obtain
\begin{align*}
\Delta_{\widehat{r}}(d\otimes b)
&=\sum_{i}\Big(\theta_{\vdash,\omega}(d)\bullet\big((b\diamond x_{i})\otimes y_{i}
-x_{i}\otimes(b\diamond y_{i})-x_{i}\otimes(y_{i}\diamond b)\big)\\[-4mm]
&\qquad\quad+\theta_{\dashv,\omega}(d)\bullet\big(x_{i}\otimes(b\diamond y_{i})
-(b\diamond x_{i})\otimes y_{i}-(x_{i}\diamond b)\otimes y_{i}\big)\Big)\\
&=\Delta(d\otimes b),
\end{align*}
since $r-\tau(r)$ is Zinb-invariant. Therefore, we get that $(D\otimes B, \cdot, \Delta)
=(D\otimes B, \cdot, \Delta_{\widehat{r}})$ as ASI bialgebras.

Second, $(i)$ and $(ii)$ follow from Proposition \ref{pro:qtr-zib}.
Finally, if $(B, \diamond, \nu_{r})$ is factorizable, i.e., $r$ is a solution of
$\ZYBE$ in $(B, \diamond)$, $r-\tau(r)$ is Zinb-invariant and the map $\mathcal{I}=
r^{\sharp}-\tau(r)^{\sharp}: B^{\ast}\rightarrow B$ is an isomorphism of vector spaces.
We need show that $\widehat{\mathcal{I}}=\widehat{r}^{\sharp}+\tau(\widehat{r})^{\sharp}:
(D\otimes B)^{\ast}\rightarrow D\otimes B$ is an isomorphism of vector spaces. Denote
$\hat{\kappa}:=\sum_{j}e_{j}\otimes f_{j}\in D\otimes D$. Define $\hat{\kappa}^{\sharp}:
D^{\ast}\rightarrow D$ by $\langle\hat{\kappa}^{\sharp}(\eta_{1}),\; \eta_{2}\rangle
=\langle\eta_{1}\otimes\eta_{2},\; \hat{\kappa}\rangle$ for any $\eta_{1}, \eta_{2}
\in D^{\ast}$. Then, $\hat{\kappa}^{\sharp}$ is a linear isomorphism and
$\langle\hat{\kappa}^{\sharp}(\eta_{1}),\; \eta_{2}\rangle=
-\langle\hat{\kappa}^{\sharp}(\eta_{2}),\; \eta_{1}\rangle$.
Therefore, for any $\eta_{1}, \eta_{2}\in D^{\ast}$ and $\xi_{1}, \xi_{2}\in B^{\ast}$,
\begin{align*}
\langle\tau(\widehat{r})^{\sharp}(\eta_{1}\otimes\xi_{1}),\; \eta_{2}\otimes\xi_{2}\rangle
&=\sum_{i,j}\langle(\eta_{1}\otimes\xi_{1})\otimes(\eta_{2}\otimes\xi_{2}),\ \
(f_{j}\otimes y_{i})\otimes(e_{j}\otimes x_{i})\rangle\\[-2mm]
&=\Big(\sum_{j}\langle\eta_{1}, f_{j}\rangle\langle\eta_{2}, e_{j}\rangle\Big)
\Big(\sum_{i}\langle\xi_{1}, y_{i}\rangle\langle\xi_{2}, x_{i}\rangle\Big)\\[-2mm]
&=-\langle\hat{\kappa}^{\sharp}(\eta_{1}),\; \eta_{2}\rangle
\langle\tau(r)^{\sharp}(\xi_{1}),\; \xi_{2}\rangle\\
&=-\langle\hat{\kappa}^{\sharp}(\eta_{1})\otimes\tau(r)^{\sharp}(\xi_{1}),\;
\xi_{2}\otimes\eta_{2}\rangle.
\end{align*}
That is, $\tau(\widehat{r})^{\sharp}=-\hat{\kappa}^{\sharp}\otimes\tau(r)^{\sharp}$,
Similarly, we have $\widehat{r}^{\sharp}=\hat{\kappa}^{\sharp}\otimes r^{\sharp}$. Thus,
$\widehat{\mathcal{I}}=\widehat{r}^{\sharp}+\tau(\widehat{r})^{\sharp}=
\hat{\kappa}^{\sharp}\otimes\mathcal{I}$ is an isomorphism of vector spaces.
The proof is completed.
\end{proof}

Let $(B, \diamond)$ be a Zinbiel algebra and $(V, \bar{\kl}, \bar{\kr})$ be a bimodule over
$(B, \diamond)$. Recall that a linear map $T: V\rightarrow B$ is called an {\bf
$\mathcal{O}$-operator of $(B, \diamond)$ associated to $(V, \bar{\kl}, \bar{\kr})$} if for
any $v_{1}, v_{2}\in V$,
$$
T(v_{1})\diamond T(v_{2})=T\big(\bar{\kl}(T(v_{1}))(v_{2})+\bar{\kr}(T(v_{2}))(v_{1})\big).
$$

\begin{pro}[\cite{Wan}]\label{pro:o-Zinb}
Let $(B, \diamond)$ be a Zinbiel algebra, $r\in B\otimes B$ be symmetric. Then $r$ is a
solution of the $\ZYBE$ in $(B, \diamond)$ if and only if $r^{\sharp}: B^{\ast}
\rightarrow B$ is an $\mathcal{O}$-operator of $(B, \diamond)$ associated to the coregular
bimodule $(B^{\ast}, -\bar{\fl}_{B}^{\ast}-\bar{\fr}_{B}^{\ast}, \bar{\fr}_{B}^{\ast})$.
\end{pro}

Let $(B, \diamond)$ be a Zinbiel algebra, $r\in B\otimes B$ and $(D, \dashv, \vdash,
\omega)$ be a quadratic diassociative algebra. Then we have the following commutative diagram:

$$
\xymatrix@C=1.4cm@R=0.5cm{
\txt{$(B, \diamond, \nu_{r})$ \\ {\tiny a triangular Zinbiel bialgebra}}
\ar[d]^{{\rm Cro.}~\ref{cor:indassbia}}_{{\rm Thm.}~\ref{thm:indu-qutriass}} &
\txt{$r$ \\ {\tiny a symmetric solution} \\ {\tiny of the $\ZYBE$ in $(B, \diamond)$}}
\ar[d]_-{{\rm Pro.}~\ref{pro:ZYBE-AYBE}}\ar[r]^-{{\rm Pro.}~\ref{pro:o-Zinb}}
\ar[l]_-{{\rm Pro.}~\ref{pro:qtr-zib}}    &
\txt{$r^{\sharp}$\\ {\tiny an $\mathcal{O}$-operator of $(B, \diamond)$} \\
{\tiny associated to coregular bimodule}}
\ar[d]^-{\mbox{$\hat{\kappa}^{\sharp}\otimes-$}} \\
\txt{$(D\otimes B, \cdot, \Delta_{\widehat{r}})=(D\otimes B, \cdot, \Delta)$ \\
{\tiny a triangular ASI bialgebra}}   &
\txt{$\widehat{r}$ \\ {\tiny a skew-symmetric solution} \\ {\tiny of the $\AYBE$ in
$(D\otimes B, \cdot)$}} \ar[r]^-{{\rm Pro.}~\ref{pro:o-ass}}
\ar[l]_-{{\rm Pro.}~\ref{pro:quasass-bia}}   &
\txt{$\widehat{r}^{\sharp}=\hat{\kappa}^{\sharp}\otimes r^{\sharp}$ \\
{\tiny an $\mathcal{O}$-operator of $(D\otimes B, \cdot)$ } \\
{\tiny associated to coregular bimodule}}}
$$
More precisely, from proof of Theorem \ref{thm:indu-qutriass}, we can conclude that the
square on the left and the square on the right are commutative.

%\begin{ex}\label{ex:ZYBE}
%Let $A=\Bbbk\{e_{1}, e_{2}\}$ be a $2$-dimensional Zinbiel algebra with non-zero products:
%$e_{1}\cdot e_{1}=e_{2}$. Consider the symmetric element $r=e_{1}\otimes e_{2}
%+e_{2}\otimes e_{1}\in A\otimes A$. It is easy to see that $r$ is a
%symmetric solution of the $\ZYBE$ in Zinbiel algebra $(A, \cdot)$.
%Therefore, the Eq. (\ref{cobou}) give a Zinbiel bialgebra structure on $(A, \cdot)$ by
%$\Delta(e_{1})=e_{2}\otimes e_{2}$ and $\Delta(e_{2})=0$.
%\end{ex}

%%%%%%%%%%%%%%%%%%%%%%%%%%%%%%%%%%%%%%%%%%%%%%%%%%%%%%%%%%%%%%%%%%%%%%%%%%%%%%%%%
%    section  5   Some constructions of Lie bialgebras
%%%%%%%%%%%%%%%%%%%%%%%%%%%%%%%%%%%%%%%%%%%%%%%%%%%%%%%%%%%%%%%%%%%%%%%%%%%%%%%%%%%%%%
\section{Some constructions of Lie bialgebras} \label{sec:diass-lie}
In this section, we consider the relationship between diassociative bialgebras,
ASI bialgebras, Leibniz bialgebras and Lie bialgebras, and show that the commutative
diagram in Proposition \ref{pro:comm-diag} is correct at the level of bialgebras.
Moreover, we can construct a Lie bialgebra form a Zinbiel bialgebra and the constructed
Lie bialgebra is quasi-triangular (resp. triangular, factorizable) if the original
Zinbiel bialgebra is quasi-triangular (resp. triangular, factorizable).

%%%%%%%%%%%%%%%%%%%%%%%%%%%%%%%%%%%%%%%%%%%%%%%%%%%%%%%%%%%%%%%%%%%%%%%%%%%%%%%%
\subsection{Constructing Lie bialgebra from Zinbiel bialgebras}\label{subsec:Z-L}
First, let us review some basic facts of Lie bialgebras.
Let $(\g, [-,-])$ be a Lie algebra. A {\bf representation} of $(\g, [-,-])$ is a pair
$(V, \rho)$, where $V$ is a vector space and $\rho: A\rightarrow \gl(V)$ is linear map
such that $\rho([g_{1}, g_{2}])=\rho(g_{1})\rho(g_{2})-\rho(g_{2})\rho(g_{1})$ for any
$g_{1}, g_{2}\in\g$. In particular, if we define $\ad_{\g}: \g\rightarrow\gl(\g)$ by
$\ad_{\g}(g_{1})(g_{2})=[g_{1}, g_{2}]$, then $(\g, \ad_{\g})$ is a representation of
$(\g, [-,-])$, which is called the {\bf regular representation}. The {\bf coregular
representation} of $(\g, [-,-])$ is given by $(\g^{\ast}, \ad_{\g}^{\ast})$.
Recall that {\bf Lie coalgebra} $(\g, \delta)$ is a vector space $\g$ with a linear map
$\delta: \g\rightarrow\g\otimes\g$ such that
$$
\tau\delta=-\delta \qquad\quad \mbox{and}\qquad\quad
(\id\otimes\delta)\delta-(\tau\otimes\id)(\id\otimes\delta)\delta=(\delta\otimes\id)\delta.
$$
A {\bf Lie bialgebra} is a triple $(\g, [-,-], \delta)$ such that
$(\g, [-,-])$ is a Lie algebra, $(\g, \delta)$ is a Lie coalgebra, and the following
compatibility condition holds:
$$
\delta([g_{1}, g_{2}])=(\ad_{\g}(g_{1})\otimes\id+\id\otimes\ad_{\g}(g_{1}))
(\delta(g_{2}))-(\ad_{\g}(g_{2})\otimes\id+\id\otimes\ad_{\g}(g_{2}))(\delta(g_{1})).
$$
A Lie bialgebra $(\g, [-,-], \delta)$ is called {\bf coboundary} if there exists an element
$r\in\g\otimes\g$ such that $\delta=\delta_{r}$, where
\begin{align}
\delta_{r}(g)=(\id\otimes\ad_{\g}(g)+\ad_{\g}(g)\otimes\id)(r), \label{lie-cobo}
\end{align}
for any $g\in\g$. Let $(\g, [-,-])$ be a Lie algebra. An element
$r=\sum_{i}x_{i}\otimes y_{i}\in\g\otimes\g$ is said to be {\bf Lie-invariant}
if $(\id\otimes\ad_{\g}(g)+\ad_{\g}(g)\otimes\id)(r)=0$. The equation
$$
\mathbf{C}_{r}:=[r_{12}, r_{13}]+[r_{13}, r_{23}]+[r_{12}, r_{23}]=0
$$
is called the {\bf classical Yang-Baxter equation} (or $\CYBE$) in $(\g, [-,-])$,
where $[r_{12}, r_{13}]=\sum_{i,j}[x_{i}, x_{j}]\otimes y_{i}\otimes y_{j}$,
$[r_{13}, r_{23}]=\sum_{i,j}x_{i}\otimes x_{j}\otimes[y_{i}, y_{j}]$ and
$[r_{12}, r_{23}]=\sum_{i,j}x_{i}\otimes[y_{i}, x_{j}]\otimes y_{j}$.

\begin{pro}[\cite{RS,LS}]\label{pro:splie-bia}
Let $(\g, [-,-])$ be a Lie algebra, $r\in\g\otimes\g$ and $\delta_{r}:
\g\rightarrow\g\otimes\g$ be the linear map defined by Eq. \eqref{lie-cobo}.
\begin{enumerate}
\item[$(i)$] If $r$ is a solution of the $\CYBE$ in $(\g, [-,-])$ and $r+\tau(r)$ is
     Lie-invariant, then $(\g, [-,-], \delta_{r})$ is a Lie bialgebra,
     which is called a {\bf quasi-triangular Lie bialgebra} associated with $r$.
\item[$(ii)$] If $r$ is a skew-symmetric solution of the $\CYBE$ in $(\g, [-,-])$, then
     $(\g, [-,-], \delta_{r})$ is a Lie bialgebra, which is called a {\bf triangular Lie
     bialgebra} associated with $r$.
\item[$(iii)$] If $(\g, [-,-], \delta_{r})$ is a quasi-triangular Lie bialgebra and
     $\mathcal{I}=r^{\sharp}+\tau(r)^{\sharp}: \g^{\ast}\rightarrow\g$ is an isomorphism
     of vector spaces, then $(\g, [-,-], \delta_{r})$ is called a {\bf factorizable Lie
     bialgebra} associated with $r$.
\end{enumerate}
\end{pro}

Given an associative algebra, we can get a Lie algebra by the commutator.
Considering the dual case, for a coassociative coalgebra $(A, \Delta)$,
we also have a Lie coalgebra structure on $A$ by setting a coproduct
$\delta=\Delta-\tau\Delta: A\rightarrow A\otimes A$.
Moreover, each ASI bialgebra induces a Lie bialgebra as following proposition.

\begin{pro}[\cite{Bai}]\label{pro:ASI-Liebia}
Let $(A, \cdot, \Delta)$ be an ASI bialgebra. Define $[-,-]: A\otimes A\rightarrow A$
by $[a_{1}, a_{2}]=a_{1}\cdot a_{2}-a_{2}\cdot a_{1}$ for any $a_{1}, a_{2}\in A$, and
$\delta=\Delta-\tau\Delta: A\rightarrow A\otimes A$. Then $(A, [-,-], \delta)$ is a Lie
bialgebra, which is called the {\bf Lie bialgebra induced from $(A, \cdot, \Delta)$}.
\end{pro}

In \cite{Hou1}, we further analyzed this conclusion and found that it maintains many
characteristics of ASI bialgebra, and provide that some solutions of the $\AYBE$ in a
associative algebra are also solutions of the $\CYBE$ in the induced Lie algebra.

\begin{pro}[\cite{Hou1}]\label{pro:ass-Lie-YBE}
Let $(A, \cdot)$ be an associative algebra and $(A, [-,-])$ be the induced Lie algebra
from $(A, \cdot)$. Suppose $r=\sum_{i}x_{i}\otimes y_{i}\in A\otimes A$ is a solution of the
$\AYBE$ in $(A, \cdot)$. If $r+\tau(r)$ is ass-invariant, then $r$ is a solution of the
$\CYBE$ in $(A, [-,-])$ and $r+\tau(r)$ is Lie-invariant.

In particular, each skew-symmetric solution of the $\AYBE$ in $(A, \cdot)$ is also a
skew-symmetric solution of the $\CYBE$ in $(A, [-,-])$.
\end{pro}

Based on the conclusion of Proposition \ref{pro:ass-Lie-YBE} and the relationship
between the bialgebra structures and the solutions of the Yang-Baxter equation,
we obtain the following conclusion.

\begin{pro}[\cite{Hou1}]\label{pro:qtAss-qtLie}
Let $(A, \cdot)$ be an associative algebra and $r\in A\otimes A$.
Suppose $(A, \cdot, \Delta_{r})$ is an ASI bialgebra and $(A, [-,-], \delta)$ is the
induced Lie bialgebra from $(A, \ast, \Delta_{r})$, where $\Delta_{r}$ is given by
Eq. \eqref{ass-cobo}. If $r+\tau(r)$ is ass-invariant, then $(A, [-,-], \delta)=
(A, [-,-], \delta_{r})$ as Lie bialgebras, where $\delta_{r}$ is defined by Eq.
\eqref{lie-cobo}. Thus, we have $(A, [-,-], \delta)$ is quasi-triangular
(resp. triangular, factorizable) if $(A, \cdot, \Delta)$ is quasi-triangular
(resp. triangular, factorizable).
\end{pro}

Thus, using the conclusions presented in Section \ref{subsec:triASI}, we can propose a
method for constructing Lie bialgebras using Zinbiel bialgebras. We first provide the
construction of solutions of the $\CYBE$ in Lie algebras.

\begin{pro}\label{pro:Z-Lie-YBE}
Let $(B, \diamond)$ be a Zinbiel algebra and $(D, \dashv, \vdash)$ be a diassociative
algebra. If we define a bracket $[-,-]$ on $D\otimes B$ by
\begin{align}
[d_{1}\otimes b_{1},\; (d_{2}\otimes b_{2}]&=(d_{1}\vdash d_{2})\otimes(b_{1}\diamond b_{2})
+(d_{1}\dashv d_{2})\otimes(b_{2}\diamond b_{1}) \label{Z-lie}\\
&\qquad-(d_{2}\vdash d_{1})\otimes(b_{2}\diamond b_{1})
-(d_{2}\dashv d_{1})\otimes(b_{1}\diamond b_{2}),   \nonumber
\end{align}
for any $b_{1}, b_{2}\in B$ and $d_{1}, d_{2}\in D$, then $(D\otimes B, [-,-])$ is a Lie
algebra. Suppose $r\in B\otimes B$ is a solution of the $\ZYBE$ in $(B, \diamond)$ and
$r-\tau(r)$ is Zinb-invariant. Then $\widehat{r}$ is a solution of the
$\CYBE$ in $(D\otimes B, [-,-])$ and $\widehat{r}+\tau(\widehat{r})$ is Lie-invariant,
where $\widehat{r}$ is given by Eq. \eqref{assr-max}.

In particular, each symmetric solution of the $\ZYBE$ in $(B, \diamond)$ is also a
skew-symmetric solution of the $\CYBE$ in $(A, [-,-])$.
\end{pro}

Therefore, by Corollary \ref{cor:indassbia}, Theorem \ref{thm:indu-qutriass},
Propositions \ref{pro:ASI-Liebia}, \ref{pro:qtAss-qtLie} and \ref{pro:Z-Lie-YBE},
we can give a construction of Lie bialgebras from Zinbiel bialgebras.

\begin{thm}\label{thm:Z-Liebia}
Let $(B, \diamond, \nu)$ be a Zinbiel bialgebra, $(D, \dashv, \vdash, \omega)$ be a
quadratic diassociative algebra and $(D\otimes B, [-,-])$ is the induced Lie algebra, where
the bracket is given by Eq. \eqref{Z-lie}. We define a linear map $\delta:
D\otimes B\rightarrow(D\otimes B)\otimes(D\otimes B)$ by
$$
\delta(d\otimes b)=\theta_{\vdash,\omega}(d)\bullet\nu(b)
+\theta_{\dashv,\omega}(d)\bullet\tau(\nu(b))
-\tau(\theta_{\vdash,\omega}(d))\bullet\tau(\nu(b))
-\tau(\theta_{\dashv,\omega}(d))\bullet\nu(b),
$$
for any $d\in D$ and $b\in B$. Then $(D\otimes B, [-,-], \delta)$ is a Lie bialgebra.
In particular, we have $(D\otimes B, [-,-], \delta)$ is quasi-triangular
(resp. triangular, factorizable) if $(B, \diamond, \nu)$ is quasi-triangular
(resp. triangular, factorizable).
\end{thm}

Let $(\g, [-,-])$ be a Lie algebra and $(V, \rho)$ be a representation of $(\g, [-,-])$.
Recall that a linear map $T: V\rightarrow\g$ is called an {\bf $\mathcal{O}$-operator
of $(\g, [-,-])$ associated to $(V, \rho)$} if for any $v_{1}, v_{2}\in V$,
$$
[T(v_{1}),\; T(v_{2})]=T\big(\rho(T(v_{1}))(v_{2})-\rho(T(v_{2}))(v_{1})\big).
$$
The notion of $\mathcal{O}$-operator of Lie algebras was introduced by Kupershmidt in
\cite{Kup}, which is considered to be the operator form of solution of $\CYBE$ in
$(\g, [-,-])$.

\begin{pro}[\cite{Kup,CP}]\label{pro:o-lie}
Let $(\g, [-,-])$ be a Lie algebra, $r\in\g\otimes\g$ be skew-symmetric.
Then $r$ is a solution of the $\CYBE$ in $(\g, [-,-])$ if and only if $r^{\sharp}$
is an $\mathcal{O}$-operator of $(\g, [-,-])$ associated to the coregular
representation $(\g^{\ast}, \ad_{\g}^{\ast})$.
\end{pro}

For the $\mathcal{O}$-operator of an associative algebra and the $\mathcal{O}$-operator of
the induced Lie algebra, we have the following conclusion.

\begin{pro}[\cite{HL}]\label{pro:o-ass-lie}
Let $(A, \cdot)$ be an associative algebra and $(A, [-,-])$ be the induced Lie algebra.
If $T: A^{\ast}\rightarrow A$ is an $\mathcal{O}$-operator of $(A, \cdot)$ associated to
coregular bimodule $(A^{\ast}, -\fr_{A}^{\ast}, -\fl_{A}^{\ast})$, then $T$ is also an
$\mathcal{O}$-operator of $(A, [-,-])$ associated to coregular representation
$(A^{\ast}, \ad^{\ast}_{A})$.
\end{pro}

%\begin{proof}
%Let $T: A^{\ast}\rightarrow A$ be an $\mathcal{O}$-operator of $(A, \cdot)$ associated
%to $(A^{\ast}, -\fr_{A}^{\ast}, -\fl_{A}^{\ast})$. That is, $T(\xi_{1})\cdot T(\xi_{2})
%=-T\big(\fr_{A}^{\ast}(T(\xi_{1}))(\xi_{2})+\fl_{A}^{\ast}(T(\xi_{2}))(\xi_{1})\big)$
%for any $\xi_{1}, \xi_{2}\in A^{\ast}$. Note that
%$$
%\langle\ad^{\ast}_{A}(T(\xi_{1}))(\xi_{2}),\; a\rangle=-\langle\xi_{2},\;
%[T(\xi_{1}), a]\rangle=\langle(\fl_{A}^{\ast}-\fr_{A}^{\ast})
%(T(\xi_{1}))(\xi_{2}),\; a\rangle,
%$$
%for any $a\in A$. That is, $\ad^{\ast}_{A}=\fl_{A}^{\ast}-\fr_{A}^{\ast}$. Thus, we have
%\begin{align*}
%&\; [T(\xi_{1}),\; T(\xi_{2})]-T\big(\ad^{\ast}_{A}(T(\xi_{1}))(\xi_{2})
%-\ad^{\ast}_{A}(T(\xi_{2}))(\xi_{1})\big)\\
%=&\; T(\xi_{1})\cdot T(\xi_{2})-T(\xi_{2})\cdot T(\xi_{1})
%-T\big(\fl_{A}^{\ast}(T(\xi_{1}))(\xi_{2})-\fr_{A}^{\ast}(T(\xi_{1}))(\xi_{2})
%-\fl_{A}^{\ast}(T(\xi_{2}))(\xi_{1})+\fr_{A}^{\ast}(T(\xi_{2}))(\xi_{1})\big)\\
%=&\; 0.
%\end{align*}
%This means that $T$ is an $\mathcal{O}$-operator of $(A, [-,-])$ associated to
%$(A^{\ast}, \ad^{\ast}_{A})$.
%\end{proof}

Thus, let $(B, \diamond)$ be a Zinbiel algebra, $(D, \dashv,
\vdash, \omega)$ be a quadratic diassociative algebra and $r\in B\otimes B$.
Then we have the following commutative diagram:

$$
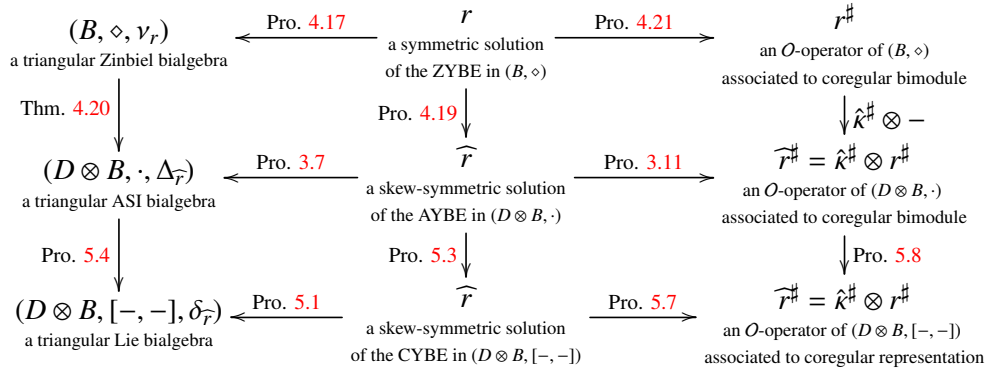
\xymatrix@C=1.4cm@R=0.5cm{
\txt{$(B, \diamond, \nu_{r})$ \\ {\tiny a triangular Zinbiel bialgebra}}
\ar[d]_{{\rm Thm.}~\ref{thm:indu-qutriass}} &
\txt{$r$ \\ {\tiny a symmetric solution} \\ {\tiny of the $\ZYBE$ in $(B, \diamond)$}}
\ar[d]_-{{\rm Pro.}~\ref{pro:ZYBE-AYBE}}\ar[r]^-{{\rm Pro.}~\ref{pro:o-Zinb}}
\ar[l]_-{{\rm Pro.}~\ref{pro:qtr-zib}}    &
\txt{$r^{\sharp}$\\ {\tiny an $\mathcal{O}$-operator of $(B, \diamond)$} \\
{\tiny associated to coregular bimodule}} \ar[d]^-{\mbox{$\hat{\kappa}^{\sharp}\otimes-$}} \\
\txt{$(D\otimes B, \cdot, \Delta_{\widehat{r}})$ \\
{\tiny a triangular ASI bialgebra}}  \ar[d]_-{{\rm Pro.}~\ref{pro:qtAss-qtLie}} &
\txt{$\widehat{r}$ \\ {\tiny a skew-symmetric solution} \\ {\tiny of the $\AYBE$ in
$(D\otimes B, \cdot)$}} \ar[r]^-{{\rm Pro.}~\ref{pro:o-ass}}
\ar[l]_-{{\rm Pro.}~\ref{pro:quasass-bia}}  \ar[d]_-{{\rm Pro.}~\ref{pro:ass-Lie-YBE}}   &
\txt{$\widehat{r}^{\sharp}=\hat{\kappa}^{\sharp}\otimes r^{\sharp}$ \\
{\tiny an $\mathcal{O}$-operator of $(D\otimes B, \cdot)$ } \\
{\tiny associated to coregular bimodule}} \ar[d]^-{{\rm Pro.}~\ref{pro:o-ass-lie}}\\
\txt{$(D\otimes B, [-,-], \delta_{\widehat{r}})$ \\ {\tiny a triangular Lie bialgebra}} &
\txt{$\widehat{r}$ \\ {\tiny a skew-symmetric solution} \\ {\tiny of the $\CYBE$ in
$(D\otimes B, [-,-])$}} \ar[r]^-{{\rm Pro.}~\ref{pro:o-lie}}
\ar[l]_-{{\rm Pro.}~\ref{pro:splie-bia}}   &
\txt{$\widehat{r}^{\sharp}=\hat{\kappa}^{\sharp}\otimes r^{\sharp}$ \\
{\tiny an $\mathcal{O}$-operator of $(D\otimes B, [-,-])$ } \\
{\tiny associated to coregular representation}}}
$$

Here, we show that there is a Lie bialgebra structure on the tensor product of a Zinbiel
bialgebra and a quadratic diassociative algebra. In \cite{HL1}, we have shown that
there is a Lie bialgebra structure on the tensor product of a Zinbiel bialgebra and a
quadratic Leibniz algebra. In fact, we can prove that these two methods are consistent.
Recall that a bilinear form $\omega(-,-)$ on a Leibniz algebra $(L, \ast)$ is called
{\bf invariant} if $\omega(x_{1}\ast x_{2},\; x_{3})=\omega(x_{1},\; x_{2}\ast x_{3}
+x_{3}\ast x_{2})$ for any $x_{1}, x_{2}, x_{3}\in L$. A {\bf quadratic Leibniz algebra},
denoted by $(L \ast, \omega)$, is a Leibniz algebra $(L, \ast)$ together with a
skew-symmetric invariant nondegenerate bilinear form $\omega(-,-)$.
A {\bf Leibniz coalgebra} $(L, \vartheta)$ is a vector space $L$ with a linear map
$\vartheta: L\rightarrow L\otimes L$ such that $(\id\otimes\vartheta)\vartheta-
(\tau\otimes\id)(\id\otimes\vartheta)\vartheta=(\vartheta\otimes\id)\vartheta$.

\begin{pro}[\cite{HL1}]\label{pro:dua-qleib}
Let $(L, \ast, \omega)$ be a quadratic Leibniz algebra. If we define a linear map
$\vartheta_{\omega}: L\rightarrow L\otimes L$ by $\omega(\vartheta_{\omega}(x_{1}),\;
x_{2}\otimes x_{3})=\omega(x_{1},\; x_{2}\diamond x_{3})$ for any
$x_{1}, x_{2}, x_{3}\in L$, then $(L, \vartheta_{\omega})$ is a Leibniz coalgebra.
\end{pro}

Let $(B, \diamond)$ be a Zinbiel algebra and $(L, \ast)$ be a Leibniz algebra.
If we define a bracket $[-,-]$ on $L\otimes B$ by $[x_{1}\otimes b_{1},\; x_{2}
\otimes b_{2}]=(x_{1}\ast x_{2})\otimes(b_{1}\diamond b_{2})-(x_{2}\ast x_{1})\otimes
(b_{2}\diamond b_{1})$ for any $x_{1}, x_{2}\in L$ and $b_{1}, b_{2}\in B$, then
$(L\otimes B, [-,-])$ is a Lie algebra, which is called {\bf the Lie algebra induced
from $(B, \diamond)$ and $(L, \ast)$}.
At the level of bialgebra, we have the following conclusion.

\begin{thm}[\cite{HL1}]\label{thm:Z-L-liebi}
Let $(B, \diamond, \nu)$ be a Zinbiel bialgebra, $(L, \ast, \omega)$ be a quadratic
Leibniz algebra and $(L\otimes B, [-,-])$ be the induced Lie algebra. Define a
linear map $\delta: L\otimes B\rightarrow(L\otimes B)\otimes(L\otimes B)$ by
$$
\delta(x\otimes b)=\vartheta_{\omega}(x)\bullet\nu(b)
-\tau(\vartheta_{\omega}(x))\bullet\tau(\nu(b)),
$$
for any $x\in L$ and $b\in B$, where $\vartheta_{\omega}$ is given in Proposition
\ref{pro:dua-qleib}. Then $(L\otimes B, [-, -], \delta)$ is a Lie bialgebra.
Moreover, we have $(L\otimes B, [-, -], \delta)$ is quasi-triangular
(resp. triangular, factorizable) if $(B, \diamond, \nu)$ is quasi-triangular
(resp. triangular, factorizable).
\end{thm}

Every diassociative algebra is naturally a Leibniz algebra (see Proposition
\ref{pro:commtor}). Direct verification shows that this conclusion is also correct at
the level of quadratic algebras.

\begin{pro}\label{pro:qdiass-qleib}
Let $(D, \dashv, \vdash, \omega)$ be a quadratic diassociative algebra and $(D, \ast)$
is the induced Leibniz algebra from $(D, \dashv, \vdash)$. Then $(D, \ast, \omega)$
is a quadratic Leibniz algebra.
\end{pro}

In fact, the Lie bialgebra obtained by the tensor product of a Zinbiel bialgebra and a
quadratic diassociative algebra is the same as the Lie bialgebra obtained by the tensor
product of this Zinbiel bialgebra and the induced quadratic Leibniz algebra.

\begin{thm}\label{pro:Z-leibcomm}
Let $(B, \diamond, \nu)$ be a Zinbiel bialgebra, $(D, \dashv, \vdash, \omega)$ be a
quadratic diassociative algebra and $(D, \ast, \omega)$ be a quadratic Leibniz algebra
induced from $(D, \dashv, \vdash, \omega)$. Denote by $(D\otimes B, [-,-]_{1}, \delta_{1})$
the induced Lie bialgebra form $(B, \diamond, \nu)$ and $(D, \dashv, \vdash, \omega)$,
and by $(D\otimes B, [-,-]_{2}, \delta_{2})$ the induced Lie bialgebra form $(B, \diamond,
\nu)$ and $(D, \ast, \omega)$. Then $(D\otimes B, [-,-]_{1}, \delta_{1})
=(D\otimes B, [-,-]_{2}, \delta_{2})$ as Lie bialgebras. That is, we have the following
commutative diagram:
$$
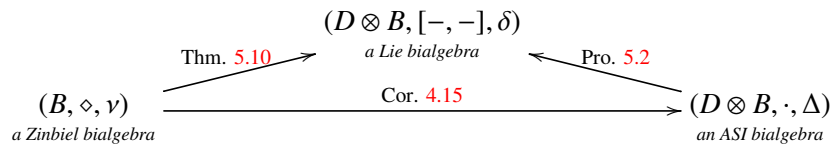
\xymatrix@C=2cm@R=0.2cm{
&\txt{$(D\otimes B, [-,-], \delta)$ \\ {\tiny a Lie bialgebra}} &  \\
\txt{$(B, \diamond, \nu)$ \\ {\tiny a Zinbiel bialgebra}}
\ar[rr]^-{{\rm Cor.}~\ref{cor:indassbia}}\ar[ru]^-{{\rm Thm.}~\ref{thm:Z-L-liebi}}
&& \txt{$(D\otimes B, \cdot, \Delta)$ \\ {\tiny an ASI bialgebra}}
\ar[lu]_-{{\rm Pro.}~\ref{pro:ASI-Liebia}}}
$$
\end{thm}

\begin{proof}
It can be directly verified by the results of Corollary \ref{cor:indassbia}, Proposition
\ref{pro:ASI-Liebia} and Theorem \ref{thm:Z-L-liebi}.
\end{proof}

\begin{ex}\label{ex:comdiagm}
Let $(B=\Bbbk\{e_{1}, e_{2}\}, \diamond, \nu)$ be the $2$-dimensional Zinbiel bialgebra
given in Example \ref{ex:ZYBE} and $(D=\Bbbk\{x_{1}, x_{2}, x_{3}, x_{4}\}, \dashv,
\vdash, \omega)$ be the $4$-dimensional quadratic diassociative algebra given in
Example \ref{ex:qu-dia}. Then we get a $8$-dimensional Lie bialgebra $(D\otimes B, [-,-],
\delta)$, where the nonzero brackets and nonzero coproducts are given by
\begin{align*}
&[x_{1}\otimes e_{1},\; x_{2}\otimes e_{1}]=2x_{1}\otimes e_{2},\qquad
\delta(x_{1}\otimes e_{1})=(x_{4}\otimes e_{2})\otimes(x_{1}\otimes e_{2})
-(x_{1}\otimes e_{2})\otimes(x_{4}\otimes e_{2}),\\
&[x_{1}\otimes e_{1},\; x_{3}\otimes e_{1}]=-x_{4}\otimes e_{2},\qquad
\delta(x_{2}\otimes e_{1})=(x_{1}\otimes e_{2})\otimes(x_{3}\otimes e_{2})
-(x_{3}\otimes e_{2})\otimes(x_{1}\otimes e_{2}),\\
&[x_{2}\otimes e_{1},\; x_{3}\otimes e_{1}]=x_{3}\otimes e_{2},\qquad
\delta(x_{3}\otimes e_{1})=2(x_{3}\otimes e_{2})\otimes(x_{4}\otimes e_{2})
-2(x_{4}\otimes e_{2})\otimes(x_{3}\otimes e_{2}).
\end{align*}
On the other hand, form the quadratic diassociative algebra $(D, x_{4}\}, \dashv,
\vdash, \omega)$, we obtain a quadratic Leibniz algebra $(D, \ast, \omega)$, where
$x_{1}\ast x_{2}=x_{1}=-x_{2}\ast x_{1}$, $x_{1}\ast x_{3}=-x_{4}$ and $x_{2}\ast
x_{3}=x_{3}$. Direct calculation shows that the Lie bialgebra induced by $(B,
\diamond, \nu)$ and $(D, \ast, \omega)$ is exactly given as above.
\end{ex}

%%%%%%%%%%%%%%%%%%%%%%%%%%%%%%%%%%%%%%%%%%%%%%%%%%%%%%%%%%%%%%%%%%%%%%%%%%%%%%%%
\subsection{Two approaches to constructing Lie bialgebras from diassociative bialgebras}
\label{subsec:diass-Lie}
In this subsection, we present two methods for constructing Lie bialgebras starting from
diassociative bialgebras, and the resulting Lie bialgebras from these two methods are
consistent. First, let $(D, \dashv, \vdash, \theta_{\dashv}, \theta_{\vdash})$ be a
diassociative bialgebra and $(B, \diamond, \varpi)$ be a quadratic Zinbiel algebra.
We have the following path to obtain a Lie bialgebra:
$$
{\bf (I)}: \qquad\qquad\quad
\xymatrix@C=2cm@R=0.6cm{
\txt{$(D, \dashv, \vdash, \theta_{\dashv}, \theta_{\vdash})$ \\
{\tiny a diassociative bialgebra}}
\ar[r]^-{\mbox{\tiny Thm. \ref{thm:dias-asbia}}}
&\txt{$(D\otimes B, \cdot, \Delta)$ \\ {\tiny an ASI bialgebra}}
\ar[r]^-{\mbox{\tiny Pro. \ref{pro:ASI-Liebia}}}
& \txt{$(D\otimes B, [-,-], \delta)$ \\ {\tiny a Lie bialgebra}}}\qquad\qquad\qquad\qquad
$$
Moreover, by Theorem \ref{thm:indu-triASI} and Proposition \ref{pro:qtAss-qtLie},
we get $(D\otimes B, [-,-], \delta)$ is a triangular Lie bialgebra if $(D, \dashv,
\vdash, \theta_{\dashv}, \theta_{\vdash})$ is a triangular diassociative bialgebra.
Furthermore, by using the one-to-one correspondence between the symmetric (resp.
skew-symmetric) solution of the (classical) Yang-Baxter equation and the
$\mathcal{O}$-operators, we have the following exchange diagram:
$$
\xymatrix@C=1.4cm@R=0.5cm{
\txt{$(D, \dashv, \vdash, \theta_{\dashv,r}, \theta_{\vdash,r})$ \\
{\tiny a triangular diassociative bialgebra}}
\ar[d]_{{\rm Thm.}~\ref{thm:indu-triASI}} &
\txt{$r$ \\ {\tiny a symmetric solution} \\ {\tiny of the $\DAYBE$ in $(D, \dashv, \vdash)$}}
\ar[d]_-{{\rm Pro.}~\ref{pro:DAYBE-AYBE}}\ar[r]^-{{\rm Pro.}~\ref{pro:o-dia}}
\ar[l]_-{{\rm Pro.}~\ref{pro:tri-di}}    &
\txt{$r^{\sharp}$\\ {\tiny an $\mathcal{O}$-operator of $(D, \dashv, \vdash)$} \\
{\tiny associated to coregular bimodule}} \ar[d]^-{\mbox{$-\otimes\kappa^{\sharp}$}} \\
\txt{$(D\otimes B, \cdot, \Delta_{\widetilde{r}})$ \\
{\tiny a triangular ASI bialgebra}}  \ar[d]_-{{\rm Pro.}~\ref{pro:qtAss-qtLie}} &
\txt{$\widetilde{r}$ \\ {\tiny a skew-symmetric solution} \\ {\tiny of the $\AYBE$ in
$(D\otimes B, \cdot)$}} \ar[r]^-{{\rm Pro.}~\ref{pro:o-ass}}
\ar[l]_-{{\rm Pro.}~\ref{pro:quasass-bia}}  \ar[d]_-{{\rm Pro.}~\ref{pro:ass-Lie-YBE}}   &
\txt{$\widetilde{r}^{\sharp}=r^{\sharp}\otimes\kappa^{\sharp}$ \\
{\tiny an $\mathcal{O}$-operator of $(D\otimes B, \cdot)$ } \\
{\tiny associated to coregular bimodule}} \ar[d]^-{{\rm Pro.}~\ref{pro:o-ass-lie}}\\
\txt{$(D\otimes B, [-,-], \delta_{\widetilde{r}})$ \\ {\tiny a triangular Lie bialgebra}} &
\txt{$\widetilde{r}$ \\ {\tiny a skew-symmetric solution} \\ {\tiny of the $\CYBE$ in
$(D\otimes B, [-,-])$}} \ar[r]^-{{\rm Pro.}~\ref{pro:o-lie}}
\ar[l]_-{{\rm Pro.}~\ref{pro:splie-bia}}   &
\txt{$\widetilde{r}^{\sharp}=r^{\sharp}\otimes\kappa^{\sharp}$ \\
{\tiny an $\mathcal{O}$-operator of $(D\otimes B, [-,-])$ } \\
{\tiny associated to coregular representation}}}
$$

On the other hand, a diassociative bialgebra induces a Leibniz bialgebra, which we can
use to construct a Lie bialgebra. Following, we will elaborate on this construction
method in detail.
Recently, in \cite{HLLZ,Lu}, it has also been proven that the result in Proposition
\ref{pro:ASI-Liebia} is true for diassociative bialgebras and Leibniz bialgebras.
Let us first review some basic facts of Leibniz bialgebras. Recall that a
{\bf representation} of a Leibniz algebra $(L, \ast)$ is a triple $(V, \tilde{\kl},
\tilde{\kr})$, where $V$ is a vector space, $\tilde{\kl}, \tilde{\kr}: L\rightarrow\gl(V)$
are linear maps such that the following equalities hold for all $x_{1}, x_{2}\in L$,
\begin{align*}
&\qquad\; \tilde{\kl}(x_{1}\ast x_{2})=\tilde{\kl}(x_{1})\tilde{\kl}(x_{2})
-\tilde{\kl}(x_{2})\tilde{\kl}(x_{1}),\\
& \tilde{\kr}(x_{1})\tilde{\kr}(x_{2})=\tilde{\kr}(x_{2}\ast x_{1})
-\tilde{\kl}(x_{2})\tilde{\kr}(x_{1})=-\tilde{\kr}(x_{1})\tilde{\kl}(x_{2}).
\end{align*}
In particular, if we define $\tilde{\fl}_{L}, \tilde{\fr}_{L}: L\rightarrow\gl(L)$ by
$\tilde{\fl}_{L}(x_{1})(x_{2})=x_{1}\ast x_{2}=\tilde{\fr}_{L}(x_{2})(x_{1})$, then
$(L, \tilde{\fl}_{L}, \tilde{\fr}_{L})$ is a representation of $(L, \ast)$, which is
called the {\bf regular representation}. The {\bf coregular representation} of
$(L, \ast)$ is given by $(L^{\ast}, \tilde{\fl}_{L}^{\ast},
-\tilde{\fl}_{L}^{\ast}-\tilde{\fr}_{L}^{\ast})$.
A {\bf Leibniz bialgebra} is a triple $(L, \ast, \vartheta)$, where $(L, \ast)$ is a
Leibniz algebra, $(L, \vartheta)$ is a Leibniz coalgebra and the following equations hold:
\begin{align*}
&\qquad\qquad\qquad\qquad\quad \tau((\tilde{\fr}_{L}(x_{2})\otimes\id)(\vartheta(x_{1})))
=(\tilde{\fr}_{L}(x_{1})\otimes\id)(\vartheta(x_{2})),\\
&\vartheta(x_{1}\ast x_{2})=(\id\otimes\tilde{\fr}_{L}(x_{2})
-(\tilde{\fl}_{L}+\tilde{\fr}_{L})(x_{2})\otimes\id)
((\id\otimes\id+\tau)(\vartheta(x_{1})))+(\id\otimes\tilde{\fl}_{L}(x_{1})
+\tilde{\fl}_{L}(x_{1})\otimes\id)(\vartheta(x_{2})),
\end{align*}
for any $x_{1}, x_{2}\in L$.

\begin{pro}[\cite{HLLZ,Lu}]\label{pro:DASI-Leibbia}
Let $(D, \dashv, \vdash, \theta_{\dashv}, \theta_{\vdash})$ be a diassociative bialgebra
and $(D, \ast)$ be the Leibniz algebra induced by $(D, \dashv, \vdash)$.
Define a linear map $\vartheta: D\rightarrow D\otimes D$ by $\vartheta=\theta_{\vdash}
-\tau\theta_{\dashv}$. Then $(D, \ast, \vartheta)$ is a Leibniz bialgebra, which is
called the {\bf Leibniz bialgebra induced by $(D, \dashv, \vdash,
\theta_{\dashv}, \theta_{\vdash})$}.
\end{pro}

The classical Yang-Baxter equation in Leibniz algebras and triangular Leibniz bialgebras
have been studied in \cite{TS}. Let $(L, \ast)$ be a Leibniz algebra.
If there exists an element $r\in L\otimes L$ such that $(L, \ast, \vartheta_{r})$ is a
Leibniz bialgebra, where $\vartheta_{r}: L\rightarrow L\otimes L$ is given by
\begin{align}
\vartheta_{r}(x)=\big((\tilde{\fl}_{L}+\tilde{\fr}_{L})(x)\otimes\id
-\id\otimes\tilde{\fr}_{L}(x)\big)(r),         \label{cobLeb}
\end{align}
for any $x\in L$, then $(L, \ast, \vartheta_{r})$ is called a {\bf coboundary Leibniz
bialgebra} associated with $r$. Let $(L, \ast)$ be a Leibniz algebra and
$r=\sum_{i}x_{i}\otimes y_{i}\in L\otimes L$. The equation
$$
\mathbf{L}_{r}=r_{12}\ast r_{13}-r_{12}\ast r_{23}-r_{23}\ast r_{12}+r_{23}\ast r_{13}=0
$$
is called the (classical) {\bf Leibniz Yang-Baxter equation} ($\LYBE$) in the Leibniz
algebra $(L, \ast)$, where $r_{12}\ast r_{13}=\sum_{i,j}(x_{i}\ast x_{j})\otimes y_{i}
\otimes y_{j}$, $r_{12}\ast r_{23}=\sum_{i,j}x_{i}\otimes(y_{i}\ast x_{j})\otimes y_{j}$,
$r_{23}\ast r_{12}=\sum_{i,j}x_{j}\otimes(x_{i}\ast y_{j})\otimes y_{i}$ and
$r_{23}\ast r_{13}=\sum_{i,j}x_{j}\otimes x_{i}\otimes(y_{i}\ast y_{j})$.

\begin{pro}[\cite{TS,BLST}]\label{pro:sLib-bia}
Let $(B, \ast)$ be a Leibniz algebra, $r\in B\otimes B$ and $\vartheta_{r}:
B\rightarrow B\otimes B$ is given by Eq. \eqref{cobLeb}.
If $r$ is a symmetric solution of the $\LYBE$ in $(B, \ast)$, then $(B, \ast,
\vartheta_{r})$ is a Leibniz bialgebra, which is called a {\bf triangular Leibniz
bialgebra} associated with $r$.
\end{pro}

For the relationship between the solutions of $\DAYBE$ and the solutions
of $\LYBE$ in the induced Leibniz algebra, we have the following conclusion.

\begin{pro}[\cite{HL}]\label{pro:diass-Leib-YBE}
Let $(D, \dashv, \vdash)$ be a diassociative algebra and $(D, \ast)$ be the induced
Leibniz algebra from $(D, \dashv, \vdash)$. If $r\in D\otimes D$ is a symmetric solution
of the $\DAYBE$ in $(D, \dashv, \vdash)$, then $r$ is also a symmetric solution
of the $\LYBE$ in $(D, \ast)$.
\end{pro}

Thus, considering the triangular diassociative bialgebras and triangular Leibniz bialgebras,
we have:

\begin{pro}[\cite{HL}]\label{pro:qtdiAss-qtLeib}
Let $(D, \dashv, \vdash)$ be a diassociative algebra and $r\in D\otimes D$.
Suppose $(D, \dashv, \vdash, \theta_{\dashv,r}, \theta_{\dashv,r})$ is a diassociative
bialgebra and $(D, \ast, \vartheta)$ is the induced Leibniz bialgebra, where
$\theta_{\dashv,r}, \theta_{\dashv,r}$ are given by Eq. \eqref{cobdi}. If $(D, \dashv,
\vdash, \theta_{\dashv,r}, \theta_{\dashv,r})$ is triangular, that is, $r$ is a
symmetric solution of the $\DAYBE$ in $(D, \dashv, \vdash)$, then $(D, \ast, \vartheta)
=(D, \ast, \vartheta_{r})$ is also a triangular Leibniz bialgebra.
\end{pro}

Let $(L, \ast)$ be a Leibniz algebra and $(V, \bar{\kl}, \bar{\kr})$ be a representation of
$(L, \ast)$. Recall that a linear map $T: V\rightarrow L$ is called an {\bf
$\mathcal{O}$-operator of $(L, \ast)$ associated to $(V, \bar{\kl}, \bar{\kr})$} if for
any $v_{1}, v_{2}\in V$,
$$
T(v_{1})\ast T(v_{2})=T\big(\kl(T(v_{1}))(v_{2})+\kr(T(v_{2}))(v_{1})\big).
$$

\begin{pro}[\cite{TS,BLST}]\label{pro:o-leib}
Let $(L, \ast)$ be a Leibniz algebra and $r\in L\otimes L$ be symmetric.
Then $r$ is a solution of the $\LYBE$ in $(L, \ast)$ if and only if
$r^{\sharp}: L^{\ast}\rightarrow L$ is an $\mathcal{O}$-operator of $(L, \ast)$
associated to the coregular representation $(L^{\ast}, \bar{\fl}_{L}^{\ast},
-\bar{\fl}_{L}^{\ast}-\bar{\fr}_{L}^{\ast})$.
\end{pro}

For the $\mathcal{O}$-operator of an associative algebra and the $\mathcal{O}$-operator of
the induced Lie algebra, similar to Proposition \ref{pro:o-ass-lie}, we have:

\begin{pro}[\cite{HL}]\label{pro:o-diass-leib}
Let $(D, \dashv, \vdash)$ be a diassociative algebra and $(D, \ast)$ be the induced
Leibniz algebra. If $T: D^{\ast}\rightarrow D$ is an $\mathcal{O}$-operator of $(D, \dashv,
\vdash)$ associated to coregular bimodule $(D^{\ast}, \fr_{\vdash}^{\ast} -\fr_{\dashv}^{\ast},
-\fl_{\vdash}^{\ast}, -\fr_{\dashv}^{\ast}, \fl_{\dashv}^{\ast}-\fl_{\vdash}^{\ast})$,
then $T$ is also an $\mathcal{O}$-operator of $(D, \ast)$ associated to
coregular representation $(D^{\ast}, \fl_{\vdash}^{\ast}-\fr_{\dashv}^{\ast},
\fl_{\dashv}^{\ast}+\fr_{\dashv}^{\ast}-\fl_{\vdash}^{\ast}-\fr_{\vdash}^{\ast})$.
\end{pro}

Since the operad of Leibniz algebras and the operad of Zinbiel algebras are Koszul dual,
we have constructed Lie bialgebras by using the tensor product of a Leibniz bialgebras and
a quadratic Zinbiel algebra \cite{HL1}.

\begin{thm}[\cite{HL1}]\label{thm:liebia-LZ}
Let $(L, \ast, \vartheta)$ be a Leibniz bialgebra, $(B, \diamond, \varpi)$
be a quadratic Zinbiel algebra and $(L\otimes B, [-,-])$ be the induced Lie
algebra by $(L, \ast)$ and $(B, \diamond)$. Define a linear map $\delta:
L\otimes B\rightarrow(L\otimes B)\otimes(L\otimes B)$ by
\begin{align}
\delta(x\otimes b)=(\id\otimes\id-\tau)(\vartheta(x)\bullet\nu_{\varpi}(b))
:=(\id\otimes\id-\tau)\Big(\sum_{(x)}\sum_{(b)}
(x_{(1)}\otimes b_{(1)})\otimes(x_{(2)}\otimes b_{(2)})\Big),\label{copro}
\end{align}
for any $x\in L$ and $b\in B$, where $\vartheta(x)=\sum_{(x)}x_{(1)}\otimes
x_{(2)}$ and $\nu_{\varpi}(b)=\sum_{(b)}b_{(1)}\otimes b_{(2)}$ in the Sweedler notation.
Then $(L\otimes B, [-,-], \delta)$ is a Lie bialgebra, which is called the
{\bf Lie bialgebra induced from $(L, \ast, \vartheta)$ by $(B, \diamond, \varpi)$}.
\end{thm}

Let $(D, \dashv, \vdash, \theta_{\dashv}, \theta_{\vdash})$ be a diassociative bialgebra
and $(B, \diamond, \varpi)$ be a quadratic Zinbiel algebra. Here we obtained the second
way to obtain Lie bialgebra from diassociative bialgebra:
$$
{\bf (II)}: \qquad\qquad\quad
\xymatrix@C=2cm@R=0.6cm{
\txt{$(D, \dashv, \vdash, \theta_{\dashv}, \theta_{\vdash})$ \\
{\tiny a diassociative bialgebra}}
\ar[r]^-{\mbox{\tiny Pro. \ref{pro:DASI-Leibbia}}}
&\txt{$(D, \ast, \vartheta)$ \\ {\tiny a Leibniz bialgebra}}
\ar[r]^-{\mbox{\tiny Thm. \ref{thm:liebia-LZ}}}
& \txt{$(D\otimes B, [-,-], \delta)$ \\ {\tiny a Lie bialgebra}}}\qquad\qquad\qquad\qquad
$$
What we are going to say next is that starting from a given diassociative bialgebra,
the Lie bialgebras obtained using these two methods are the same.

\begin{thm}\label{thm:commdig}
Let $(D, \dashv, \vdash, \theta_{\dashv}, \theta_{\vdash})$ be a diassociative
bialgebra and $(B, \diamond, \varpi)$ be a quadratic Zinbiel algebra. Then
we have the following commutative diagram:
$$
\xymatrix@C=2cm@R=0.7cm{
\txt{$(D, \dashv, \vdash, \theta_{\dashv}, \theta_{\vdash})$ \\
{\tiny a diassociative bialgebra}}
\ar[d]_{{\rm Pro.}~\ref{pro:DASI-Leibbia}} \ar[r]^{{\rm Thm.}~\ref{thm:dias-asbia}}
&\txt{$(D\otimes B, \cdot, \Delta)$\\  {\tiny an ASI bialgebra}}
\ar[d]^{{\rm Pro.}~\ref{pro:ASI-Liebia}} \\
\txt{$(D, \ast, \vartheta)$ \\ {\tiny a Leibniz bialgebra}}
\ar[r]^{{\rm Thm.}~\ref{thm:liebia-LZ}\quad}
& \txt{$(D\otimes B, [-,-], \delta)$ \\ {\tiny a Lie bialgebra}}}
$$
\end{thm}

\begin{proof}
Denote the Lie bialgebras obtain from methods {\bf (I)} and {\bf (II)} by
$(D\otimes B, [-,-]_{1}, \delta_{1})$ and $(D\otimes B, [-,-]_{2}, \delta_{2})$ respectively.
By Proposition \ref{pro:comm-diag}, we know that $[-,-]_{1}=[-,-]_{2}$. Hence we only
need to show that $\delta_{1}=\delta_{2}$. For any $d\in D$ and $b\in B$,
\begin{align*}
\delta_{1}(d\otimes b)
&=\theta_{\vdash,\omega}(d)\bullet\nu(b)+\theta_{\dashv,\omega}(d)\bullet\tau(\nu(b))
-\tau(\theta_{\vdash,\omega}(d))\bullet\tau(\nu(b))
-\tau(\theta_{\dashv,\omega}(d))\bullet\nu(b)\\
%&=(\id\otimes\fd)(\vartheta(p))\bullet\theta_{\omega}(c)
%+\tau((\fd\otimes\id)(\vartheta(p)))\bullet\tau(\theta_{\omega}(c))\\
%&=(\id\otimes\hat{\fd})(\vartheta(p)\bullet\theta_{\omega}(c)+\tau(\vartheta(p))
%\bullet\tau(\theta_{\omega}(c)))\\
&=(\id\otimes\id-\tau)\big(\theta_{\vdash}(d)\bullet\nu_{\varpi}(b)
-\tau\theta_{\dashv}(d)\bullet\nu_{\varpi}(b)\big)\\
&=\delta_{2}(d\otimes b).
\end{align*}
That is, $\delta_{1}=\delta_{2}$. Thus, $(D\otimes B, [-,-]_{1}, \delta_{1})=
(D\otimes B, [-,-]_{2}, \delta_{2})$ as Lie bialgebras.
\end{proof}

The relationship between solutions of the $\LYBE$ in a Leibniz algebra and solutions of
the $\CYBE$ in the induced Lie algebra has been discussed in detail \cite{HL1}.
In particular, we have the following proposition.

\begin{pro}[\cite{HL1}]\label{pro:LYBE-CYBE}
Let $(L, \ast)$ be a Leibniz algebra, $(B, \diamond, \varpi)$ be a quadratic Zinbiel
algebra and $(L\otimes B, [-,-])$ be the induced Lie algebra. Suppose that
$r=\sum_{i}x_{i}\otimes y_{i}\in L\otimes L$ is a symmetric solution of the $\LYBE$
in $(L, \ast)$, then
\begin{align}
\widetilde{r}=\sum_{i,j}(x_{i}\otimes e_{j})\otimes(y_{i}\otimes f_{j})
\in(L\otimes B)\otimes(L\otimes B)  \label{Lr-max}
\end{align}
is a skew-symmetric solution of the $\CYBE$ in $(L\otimes B, [-,-])$, where $\{e_{1},
e_{2},\cdots, e_{n}\}$ is a basis of $B$ and $\{f_{1}, f_{2},\cdots, f_{n}\}$ is the
dual basis of $\{e_{1}, e_{2},\cdots, e_{n}\}$ with respect to $\varpi(-,-)$.
\end{pro}

Thus, by using the correspondence between the skew-symmetric solution of $\CYBE$ in
a Lie algebra and triangular Lie bialgebra, the symmetric
solution of $\LYBE$ in a Leibniz algebra and triangular Leibniz bialgebra, we have:

\begin{pro}[\cite{HL1}]\label{pro:sLeib-sLie}
Let $(L, \ast, \vartheta)$ be a Leibniz bialgebra, $(B, \diamond, \varpi)$
be a quadratic Zinbiel algebra and $(L\otimes B, [-,-], \delta)$ be the Lie bialgebra
induced from $(L, \ast, \vartheta)$ by $(B, \diamond, \varpi)$.
If $(L, \ast, \vartheta)$ is triangular, i.e., there exists a symmetric solution $r$
of the $\LYBE$ in $(L, \ast)$ such that $\vartheta_{r}$, then $(L\otimes B, [-,-],
\delta)=(L\otimes B, [-,-], \delta_{\widetilde{r}})$ is a triangular Lie bialgerba,
where $\widetilde{r}$ is given by Eq. \eqref{Lr-max}.
\end{pro}

Moreover, as a consequence of Theorem \ref{thm:commdig}, we can provide a commutative
diagram about triangle bialgebra structures in the following corollary.

\begin{cor}\label{cor:comdig}
Let $(D, \dashv, \vdash, \theta_{\dashv,r}, \theta_{\vdash,r})$ be a triangle diassociative
bialgebra and $(B, \diamond, \varpi)$ be a quadratic Zinbiel algebra. Then
we have the following commutative diagram:
$$
\xymatrix@C=2cm@R=0.7cm{
\txt{$(D, \dashv, \vdash, \theta_{\dashv,r}, \theta_{\vdash,r})$ \\
{\tiny a triangle diassociative bialgebra}}
\ar[d]_{{\rm Pro.}~\ref{pro:qtdiAss-qtLeib}} \ar[r]^{{\rm Thm.}~\ref{thm:indu-triASI}}
&\txt{$(D\otimes B, \cdot, \Delta_{\widetilde{r}})$\\  {\tiny a triangle ASI bialgebra}}
\ar[d]^{{\rm Pro.}~\ref{pro:qtAss-qtLie}} \\
\txt{$(D, \ast, \vartheta_{r})$ \\ {\tiny a triangle Leibniz bialgebra}}
\ar[r]^{{\rm Pro.}~\ref{pro:sLeib-sLie}\quad}
& \txt{$(D\otimes B, [-,-], \delta_{\widetilde{r}})$ \\ {\tiny a triangle Lie bialgebra}}}
$$
\end{cor}

The $\mathcal{O}$-operator is considered to be the operator form of solution of the
classical Yang-Baxter equation. By using the correspondence between symmetric
solutions of $\LYBE$ in a Leibniz algebra and skew-symmetric solutions of $\CYBE$
in the induced Lie algebra, we can give a corresponding between $\mathcal{O}$-operators
of this Leibniz algebra and $\mathcal{O}$-operators of induced Lie algebra.

\begin{pro}\label{pro:o-Lei-L}
Let $(L, \ast)$ be a Leibniz algebra, $(B, \diamond)$ be a Zinbiel algebra and
$(L\otimes B, [-,-], \delta)$ be the Lie bialgebra induced from $(L, \ast)$ and
$(B, \diamond)$. If $r$ is a symmetric solution of the $\LYBE$ in $(L, \ast)$, then
we have the following commutative diagram:
$$
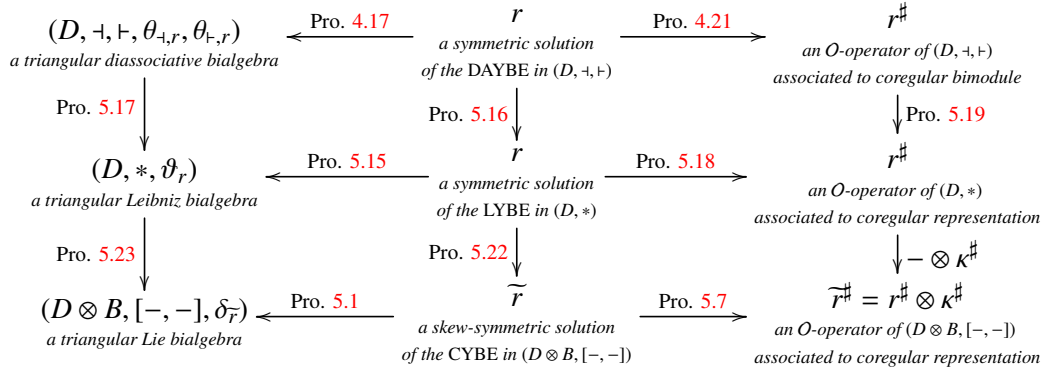
\xymatrix@C=1.4cm@R=0.5cm{
\txt{$(D, \dashv, \vdash, \theta_{\dashv,r}, \theta_{\vdash,r})$ \\
{\tiny a triangular diassociative bialgebra}}
\ar[d]_{{\rm Pro.}~\ref{pro:qtdiAss-qtLeib}} &
\txt{$r$ \\ {\tiny a symmetric solution} \\ {\tiny of the $\DAYBE$ in $(D, \dashv, \vdash)$}}
\ar[d]_-{{\rm Pro.}~\ref{pro:diass-Leib-YBE}}\ar[r]^-{{\rm Pro.}~\ref{pro:o-Zinb}}
\ar[l]_-{{\rm Pro.}~\ref{pro:qtr-zib}}    &
\txt{$r^{\sharp}$\\ {\tiny an $\mathcal{O}$-operator of $(D, \dashv, \vdash)$} \\
{\tiny associated to coregular bimodule}} \ar[d]^-{{\rm Pro.}~\ref{pro:o-diass-leib}} \\
\txt{$(D, \ast, \vartheta_{r})$ \\
{\tiny a triangular Leibniz bialgebra}}  \ar[d]_-{{\rm Pro.}~\ref{pro:sLeib-sLie}} &
\txt{$r$ \\ {\tiny a symmetric solution} \\ {\tiny of the $\LYBE$ in
$(D, \ast)$}} \ar[r]^-{{\rm Pro.}~\ref{pro:o-leib}}
\ar[l]_-{{\rm Pro.}~\ref{pro:sLib-bia}}  \ar[d]_-{{\rm Pro.}~\ref{pro:LYBE-CYBE}}   &
\txt{$r^{\sharp}$ \\ {\tiny an $\mathcal{O}$-operator of $(D, \ast)$ } \\
{\tiny associated to coregular representation}} \ar[d]^-{\mbox{$-\otimes\kappa^{\sharp}$}} \\
\txt{$(D\otimes B, [-,-], \delta_{\widetilde{r}})$ \\ {\tiny a triangular Lie bialgebra}} &
\txt{$\widetilde{r}$ \\ {\tiny a skew-symmetric solution} \\ {\tiny of the $\CYBE$ in
$(D\otimes B, [-,-])$}} \ar[r]^-{{\rm Pro.}~\ref{pro:o-lie}}
\ar[l]_-{{\rm Pro.}~\ref{pro:splie-bia}}   &
\txt{$\widetilde{r}^{\sharp}=r^{\sharp}\otimes\kappa^{\sharp}$ \\
{\tiny an $\mathcal{O}$-operator of $(D\otimes B, [-,-])$ } \\
{\tiny associated to coregular representation}}}
$$
where $\kappa:=\sum_{j}e_{j}\otimes f_{j}\in B\otimes B$, $\{e_{1}, e_{2},\cdots, e_{n}\}$
is a basis of $B$ and $\{f_{1}, f_{2},\cdots, f_{n}\}$ is the dual basis of $\{e_{1}, e_{2},
\cdots, e_{n}\}$ with respect to $\varpi(-,-)$.
\end{pro}

\begin{proof}
Here we only need to show $\widetilde{r}^{\sharp}=r^{\sharp}\otimes\kappa^{\sharp}$.
This conclusion can be found in the proof process of \cite[Theorem 3.20]{HL1},
or in a proof similar to Proposition \ref{pro:o-dia-ass}, which directly verifies it.
\end{proof}

Thus, by the commutative diagram  before Proposition \ref{pro:DASI-Leibbia},
the commutative diagrams in Propositions \ref{cor:comdig} and \ref{pro:o-Lei-L},
we can provide the complete three-dimensional commutative diagram presented in Section
\ref{sec:intr}. At the end of this paper, continuing with the calculations in
Example \ref{ex:ind-triASI}, we provide a small example of the three-dimensional
commutative diagram in Section \ref{sec:intr}.

\begin{ex}\label{ex:3-comdiagm}
Let $(D\otimes B, \cdot, \Delta)$ be the $16$-dimensional triangular ASI bialgebra
associated with $\widetilde{r}$ given in Example \ref{ex:ind-triASI}. On the one hand,
this ASI bialgebra induces a Lie bialgebra $(D\otimes B, [-,-], \delta)$, where the
nonzero bracket and coproducts are given by
\begin{align*}
& [x_{2}\otimes e_{1},\; x_{2}\otimes e_{4}]=3x_{1}\otimes e_{2}-3x_{1}\otimes e_{3},
\qquad\qquad\quad [x_{2}\otimes e_{1},\; x_{3}\otimes e_{1}]=x_{4}\otimes e_{2},\\
& [x_{3}\otimes e_{1},\; x_{2}\otimes e_{4}]=x_{4}\otimes e_{2}-2x_{4}\otimes e_{3},
\qquad\qquad\quad \ \  [x_{2}\otimes e_{4},\; x_{3}\otimes e_{4}]=x_{4}\otimes e_{3},\\
& [x_{2}\otimes e_{4},\; x_{3}\otimes e_{4}]=2x_{4}\otimes e_{2}-x_{4}\otimes e_{3},\\
& \Delta(x_{2}\otimes e_{1})=(x_{1}\otimes e_{2})\otimes(x_{4}\otimes e_{2})
+(x_{1}\otimes e_{2})\otimes(x_{4}\otimes e_{3})
-2(x_{1}\otimes e_{3})\otimes(x_{4}\otimes e_{2})\\[-1mm]
&\qquad\qquad\quad-(x_{4}\otimes e_{2})\otimes(x_{1}\otimes e_{2})
-(x_{4}\otimes e_{3})\otimes(x_{1}\otimes e_{2})
+2(x_{4}\otimes e_{2})\otimes(x_{1}\otimes e_{3}),\\
& \Delta(x_{2}\otimes e_{4})=2(x_{1}\otimes e_{2})\otimes(x_{4}\otimes e_{3})
-(x_{1}\otimes e_{3})\otimes(x_{4}\otimes e_{3})
-(x_{1}\otimes e_{3})\otimes(x_{4}\otimes e_{2})\\[-1mm]
&\qquad\qquad\quad
-2(x_{4}\otimes e_{3})\otimes(x_{1}\otimes e_{2})
+(x_{4}\otimes e_{3})\otimes(x_{1}\otimes e_{3})
+(x_{4}\otimes e_{2})\otimes(x_{1}\otimes e_{3}),\\
&\qquad \delta(x_{3}\otimes e_{1})=3(x_{4}\otimes e_{3})\otimes(x_{4}\otimes e_{2})
-3(x_{4}\otimes e_{2})\otimes(x_{4}\otimes e_{3})=\delta(x_{3}\otimes e_{4}).
\end{align*}
On the other hand, one can check that the element $\widetilde{r}\in(D\otimes B)\otimes
(D\otimes B)$ given in Example \ref{ex:ind-triASI} is a skew-symmetric solution of the
$\CYBE$ in the $16$-dimensional lie algebra $(D\otimes B, [-,-])$, and the Lie bialgebra
$(D\otimes B, [-,-], \delta_{\widetilde{r}})$ associated with $\widetilde{r}$ is
exactly the $16$-dimensional Lie bialgebra we provided above. This actually provides
a specific example of the top surface of the three-dimensional commutative diagram
in Section \ref{sec:intr} being commutative.
\end{ex}

\bigskip
\noindent
{\bf Acknowledgements. } This work was financially supported by National
Natural Science Foundation of China (No.11771122).

\smallskip
\noindent
{\bf Declaration of interests.} The authors have no conflicts of interest to disclose.

\smallskip
\noindent
{\bf Data availability.} Data sharing is not applicable to this article as no new data were
created or analyzed in this study.

 \end{document}